\documentclass[11pt, reqno]{amsart}
\usepackage{url}
\usepackage[colorlinks]{hyperref}
\usepackage{amssymb,amsmath,amscd,amsthm}
\usepackage{graphicx}
\usepackage[table,dvipsnames]{xcolor}
\usepackage[verbose]{wrapfig}
\usepackage[font=smaller,labelfont=bf]{caption}
\usepackage{mathtools}
\usepackage[margin=1in]{geometry}
\usepackage{soul}
\usepackage{comment}

\usepackage{floatrow}
\DeclareFloatSeparators{mysep}{\hskip-1em}

\usepackage{thmtools}
\usepackage{thm-restate}

\usepackage[ruled,vlined]{algorithm2e}
\SetAlgoSkip{}
\usepackage{amsfonts}
\usepackage{tikz-cd}
\usepackage[title]{appendix}
\graphicspath{ {figures/} }

\declaretheorem[numberwithin=section]{theorem}

\declaretheorem[numberlike=theorem]{lemma}

\declaretheorem[numberlike=theorem, style=definition]{definition}
\declaretheorem[numberlike=theorem, style=definition]{example}

\declaretheorem[numberlike=theorem, style=remark]{remark}

\numberwithin{equation}{section}

\newcommand{\mapx}{\chi} 
\newcommand{\mapy}{\Upsilon}
\newcommand{\field}{\mathbb{F}}
\newcommand{\restrictiondef}{R}
\newcommand{\restrictionalgo}{\mathfrak{R}}
\newcommand{\cedef}{E}
\newcommand{\cealgo}{\mathfrak{E}}

\newcommand{\selectIntX}{\tau} 
\newcommand{\intY}{\gamma} 
\newcommand{\intFilt}{\SCOne^{\psi} \cap \filtTwo} 
\newcommand{\intXdelta}{\rho} 
\newcommand{\intFiltBirth}{\widehat{\beta}}
\newcommand{\intFiltDeath}{\widehat{\delta}}

\newcommand{\SCOne}{Z}
\newcommand{\SCTwo}{Y}
\newcommand{\filtOne}{\SCOne^\bullet}
\newcommand{\filtTwo}{\SCTwo^\bullet} 

\newcommand{\IntDecZ}{\mathcal{B}}

\newcommand{\IntDecDeltaLetter}{F}
\newcommand{\IntDecDeltaLower}{\MakeLowercase{\IntDecDeltaLetter}}
\newcommand{\IntDecDelta}{\mathcal{\IntDecDeltaLetter}}

\newcommand{\IntDecDeltaOneLetter}{F}
\newcommand{\IntDecDeltaTwoLetter}{G}
\newcommand{\IntDecDeltaOneLower}{\MakeLowercase{\IntDecDeltaOneLetter}}
\newcommand{\IntDecDeltaTwoLower}{\MakeLowercase{\IntDecDeltaTwoLetter}}
\newcommand{\IntDecDeltaOne}{\mathcal{\IntDecDeltaOneLetter}}
\newcommand{\IntDecDeltaTwo}{\mathcal{\IntDecDeltaTwoLetter}}

\newcommand{\IntDecYLetter}{D}
\newcommand{\IntDecYLower}{\MakeLowercase{\IntDecYLetter}}
\newcommand{\IntDecY}{\mathcal{\IntDecYLetter}}
\newcommand{\IntDecYLetterTwo}{G}

\newcommand{\IntDecYTwo}{\mathcal{\IntDecYLetterTwo}}

\newcommand{\Rho}{\text{\sc{bars}}}

\newcommand{\bc}{\text{\sc bc}}
\newcommand{\HH}{\smash{\widetilde{H}}}

\newif\ifshow 
\showfalse 

\ifshow
  \includecomment{wrap}
\else
  \excludecomment{wrap} 
\fi

\begin{document}
 
\title[Persistent extensions and analogous bars]{Persistent extensions and analogous bars: data-induced relations between persistence barcodes}

\author{Hee Rhang yoon}
\address{Department of Mathematical Sciences, University of Delaware, Newark, DE 19716}
\email{irishryoon@gmail.com}

\author{Robert Ghrist}
\address{Department of Mathematics and Electrical \& Systems Engineering, University of Pennsylvania, Philadelphia, PA 19104}
\email{ghrist@math.upenn.edu}

\author{Chad Giusti}
\address{Department of Mathematical Sciences, University of Delaware, Newark, DE 19716}
\email{cgiusti@udel.edu}

\date{\today}

\begin{abstract}
A central challenge in topological data analysis is the interpretation of barcodes. The classical algebraic-topological approach to interpreting homology classes is to build maps to spaces whose homology carries semantics we understand and then to appeal to functoriality. However, we often lack such maps in real data; instead, we must rely on a cross-dissimilarity measure between our observations of a system and a reference. In this paper, we develop a pair of computational homological algebra approaches for relating persistent homology classes and barcodes: \emph{persistent extension}, which enumerates potential relations between cycles from two complexes built on the same vertex set, and the method of \emph{analogous bars}, which utilizes persistent extension and the witness complex built from a cross-dissimilarity measure to provide relations across systems. We provide an implementation of these methods and demonstrate their use in comparing cycles between two samples from the same metric space and determining whether topology is maintained or destroyed under clustering and dimensionality reduction.
\end{abstract}

\maketitle

\section{Introduction}

Persistent homology \cite{ELZ_persistence} measures how structure (encoded as \emph{homology classes}) varies with the parameter in a parameterized space. In principle, to apply this tool to data, we design parameterized combinatorial encodings in which the presence or absence of homology classes describes features of interest. In practice, it is prohibitively difficult and time-consuming to construct novel combinatorial structures which would be easy to interpret in terms of our data sets. Rather, we employ standard encodings for which theory and computational tools are readily available.

Currently, the most common approach is to take as input a pairwise dissimilarity measure $M_P$ for system constituents $P = \{p_1, \dots, p_n\}$ and to construct from this data a \emph{weighted clique complex} $X_P^\bullet = X(M_P)^\bullet;$ vertices correspond to the elements in the system and simplices are weighted by the maximum pairwise dissimilarity of their vertices. When $M_P$ is a metric, this produces the usual weighted \emph{Vietoris-Rips complex} for a point cloud. We then employ computational homological algebra to determine the complex's $k$th persistent reduced\footnote{All complexes in this paper are non-empty. To simplify statements of results, we will always use reduced homology, which we denote $\HH_k$, and we will omit the word \emph{reduced} to avoid clutter.} homology, $P\HH_k(X_P^\bullet).$ Most current software returns this information in the form of a \emph{barcode}, $\bc_k(X_P^\bullet)$, which encodes the homology classes and their \emph{birth} and \emph{death} parameters.

We are then left with the problem of interpreting the barcode. This is difficult in part due to the intricate structure of the weighted clique complexes, including trade-offs between fidelity of representations and their combinatorial and computational complexity, and in part due to the abstract nature of persistent homology classes, made up of sequences of affine subspaces of quotient vector spaces. As such, most modern applications of persistent homology involve vectorization of the barcode as a statistic for differentiating classes of data, or \textit{ad hoc} interpretations of classes.

\medskip

Success stories from classical algebraic topology suggest that to understand homology classes we should appeal to functoriality. By constructing maps to or from reference objects which are better understood, topologists have developed \emph{semantics} for (co)homology classes in a range of contexts. Perhaps the best known classical examples are classifying maps and characteristic classes for vector bundles \cite{Milnor1956}, which have been adapted to provide nonlinear coordinate systems for data sets \cite{Perea_circular_coord}. 

In the context of applied topology, our reference  often comes in the form of more data: observations of some collection $Q= \{q_1, \dots, q_m\}$ of inputs to, outputs from, or known correlates for activity in the system $P$. From a dissimilarity matrix $M_Q$ derived from these data $Q$, we can obtain another combinatorial space and compute a barcode $\bc_k(X_Q^\bullet)$. When this second data set is accessible, experimentally or theoretically, we are better equipped to assign semantics to classes in $P\HH_k(X_P^\bullet)$. However, \emph{a priori} there is insufficient information to construct the maps we need to apply functoriality, and we are left with the problem of determining whether and how classes from $P\HH_k(X_Q^\bullet)$ correspond to classes in $P\HH_k(X_P^\bullet)$.
  
To do so, we require some notion of how $P$ and $Q$ are related. In scientific and engineering applications, a common and effectively minimal way to satisfy this requirement is through a measure  $M_{P,Q}$ of cross-dissimilarity between $P$ and $Q$. For parameterized clique complexes, $M_{P,Q}$ measures dissimilarity between the vertices of $X_P^\bullet$ and $X_Q^\bullet$. It is tempting to construct a Vietoris-Rips complex on the joint dissimilarity measure on $P\cup Q$ given by combining $M_P$, $M_Q$, and $M_{P,Q}$
and then to apply recently developed methods, including induced matching \cite{induced_matching, PD_as_D}, cycle registration \cite{cycle_registration}, or basis-independent partial matching \cite{basis_indep_matching}, to match classes via the zig-zag of induced inclusion maps. However, as we will discuss in Section \ref{approaches}, even in simple cases involving pairs of point clouds in the same metric space this approach can fail to produce matches which satisfy our intuition.

\medskip

The alternative we develop here is to observe that $M_{P,Q}$ is precisely the data of a weighted \emph{witness complex} \footnote{We use a slightly different notion of witness complex than the one introduced in \cite{Witness}. Our witness complex would be identical to that of \cite{Witness} if we assign a birth time of 0 to all vertices and take $v = 0$ in their definition. Our witness complex is a filtration of the Dowker complex. } \cite{Dowker, Witness}. The witness complex $W_{P, Q}^\bullet$ has vertices $P$ and faces indexed by $Q$: the face $q_c$ at parameter $\ell$ has vertices $\{p_r \in P \;:\; (M_{P,Q})_{r,c} \leq \ell\}$. The Functorial Dowker Theorem \cite{Dowker, functorial_dowker} provides an explicit isomorphism $P\HH_k(W_{P, Q}^\bullet) \cong P\HH_k(W_{Q, P}^\bullet)$, and thus our cross-dissimilarity matrix provides a bridge between persistent homology classes in simplicial complexes with vertices $P$ and $Q$. 

To apply this information, in Section \ref{Extension} we  introduce \emph{persistent extension}, a general method for comparing persistent homology classes between complexes $Z^\bullet$ and $Y^\bullet$ supported on the same vertex set. Suppose $Y^\epsilon$ is contractible for large $\epsilon,$ and fix some parameter $\psi$ of $Z^\bullet$. Persistent extension takes as input a class $[\tau] \in \HH_k(Z^\psi)$ and outputs an enumeration of persistent homology classes in $P\HH_k(Y^\bullet)$ that contain representatives of $[\tau]$. To do so, we observe that classes of infinite persistence in the \emph{auxiliary} filtered complex $Z^\psi \cap Y^\bullet$ are precisely the elements of $\HH_k(Z^\psi)$. We then leverage a zig-zag through the auxiliary persistence module to set up systems of linear equations that describe the image in  $P\HH_k(Y^\bullet)$ of all elements of $P\HH_k(Z^\psi \cap Y^\ell)$ which map to $[\tau]$ as $\ell \to \infty.$ As it is common to speak of persistent homology in terms of the barcode, in Section \ref{section:bar_to_bars_extension} we develop the necessary tools to translate between the persistence module-level information and the barcodes.

In Section \ref{analogous_intervals}, we combine the isomorphism from Dowker's theorem with persistent extension to obtain a pair of processes for relating persistent homology classes in $P\HH_k(X^\bullet_Q)$ to those in $P\HH_k(X^\bullet_P).$ The first, the \emph{feature-centric analogous bars} method, focuses on an individual bar $\tau \in \bc_k(X^\bullet_Q)$ and enumerates subsets of $\bc_k(X^\bullet_P)$ which potentially correspond to $\tau$ given the structure of $M_{P,Q}$. The \emph{similarity-centric analogous bars} method, on the other hand, focuses on significant features in $\bc_k(W^\bullet_{P,Q})$, which intuitively describe very strong relationships between subsets of bars in $\bc_k(X_Q^\bullet)$ and $\bc_k(X_P^\bullet)$ in the same way long lifetimes in Vietoris-Rips complexes of point clouds intuitively correspond to significant features of the support of the underlying distribution. 

\medskip

This paper is organized as follows. In Section \ref{preliminaries}, we recall the relevant algebraic and topological notions and set the notations and assumptions, and in Section \ref{approaches}, we discuss specific difficulties that arise when attempting to match bars between simplicial complexes, including a discussion of existing methods. Having established these preliminaries, in Sections \ref{Extension}-\ref{analogous_intervals}, we provide the details of the methods of persistent extension and analogous bars and demonstrate their use on simple examples. In Section \ref{applications} we demonstrate some simple applications: using analogous bars to compare the barcodes of two samples from the same distribution on a metric space, and using persistent extension to determine if topological structures are preserved under clustering and dimensionality reduction. Finally, in Section \ref{conclusion} we discuss context and future directions. To improve readability, we defer technical proofs to appendices. 

\subsection{Contributions}
The principal contributions of this paper are:
\begin{itemize}
    \item the \emph{persistent extension} method, which compares persistent cycles and barcodes between filtered simplicial complexes built on a common vertex set,
    \item the \emph{analogous bars} method, which compares persistent cycles and barcodes between two distinct clique complexes using a cross-dissimilarity measure,
    \item implementation of these methods built using the Eirene persistent homology package \cite{henselmanghristl6}, available at \url{https://github.com/UDATG/analogous_bars}, and 
    \item demonstration on some toy examples, along with a discussion of potential applications.
\end{itemize}

\begin{remark}
\label{rem:other_work}
While the authors were writing this paper,  independent work \cite{JNT_barcode_bases} investigating the structure of the set of bases for persistence barcodes appeared. While the fundamental aims of the two papers are different, \cite{JNT_barcode_bases} establishes results that generalize our Lemmas \ref{lemma:basis_change}, \ref{lemma:diagonal}, and \ref{lemma:LZ}, in the course of developing a more complete picture of the relationship between  persistence modules and barcodes. We retain our versions of these results in this manuscript for completeness and to save the reader effort in translating their work to our language, notation, and perspective.
\end{remark}

\subsection{Acknowledgements}

The authors would like to thank Gregory Henselman-Petrusek for a great many enlightening conversations. CG is supported by NSF-1854683 and AFOSR FA9550-21-1-0266. HY is supported by NSF-1854683 and ONR N00014-16-1-2010. RG is supported by ONR N00014-16-1-2010 and NSF-1934960.

\section{Preliminaries}\label{preliminaries}

We begin by recalling relevant definitions and results from persistent homology. Along the way we will introduce some new terminology and results about change of basis of a persistence module. We assume readers are familiar with the general theory of persistent homology. A more complete development of this material can be found in, for example, \cite{ELZ_persistence, compute_PH, Carlsson_topology_data, Ghrist_barcodes}. Fix a field $\mathbb{F}$ throughout.

\subsection{Primer on persistent homology} 

\begin{definition} 
A persistence module is a $\mathbb{Z}$-graded $\mathbb{F}$-vector space $V^\bullet = \bigoplus_{\ell = -\infty}^\infty V^\ell$ equipped with linear \emph{structure maps} $\{ \phi^\ell : V^\ell \to V^{\ell + 1} \}$, thought of as a $\mathbb{F}[x]$-module in which $x$ acts on elements of $V^\ell$ via $\phi^\ell.$ We say a persistence module has \emph{finite support} if $V^\ell = 0$ for all $\ell$ outside some compact interval. 
\end{definition}

We will omit the field from the notation and simply say ``persistence module" throughout. Similarly,  we will usually refer to a persistence module by the name of its underlying graded vector space, leaving the structure maps implicit. Finally, as all of our persistence modules have finite support, we will omit vector spaces and maps outside of the support from diagrams without further comment.

The prototypical examples of persistence modules are the \emph{interval modules}.

\begin{definition}
For $\beta, \delta \in \mathbb{Z}$, the  \emph{interval module} $I[\beta, \delta)^\bullet$ is the persistence  module $\bigoplus_{\ell = -\infty}^\infty I[\beta, \delta)^\ell$, where $$I[\beta, \delta)^\ell \cong \begin{cases}
\mathbb{F} & \beta \leq \ell < \delta\\
0 & \text{else}\end{cases},$$ along with structure maps $\phi^\ell : I[\beta, \delta)^\ell \to I[\beta, \delta)^{\ell+1}$ for all $\ell \in \mathbb{Z}$ given by $$\phi^\ell(f) = \begin{cases} f & \ell \neq \beta-1, \delta \\
0 &\text{ else }\end{cases}.$$
\end{definition}

Our primary example of interest arises in the context of filtered topological spaces. 

\begin{definition}
Let
$$\filtOne = \SCOne^1 \xhookrightarrow{\iota^1} \SCOne^2 \xhookrightarrow{\iota^2} \cdots \xhookrightarrow{\iota^{N-2}} \SCOne^{N-1} \xhookrightarrow{\iota^{N-1}} \SCOne^N$$ be a filtered topological space. The \emph{degree-$k$ reduced persistent homology} of $\filtOne$ is the persistence module $P\HH_k(\filtOne)$ with
$$\left(P\HH_k(\filtOne)\right)^\ell =\begin{cases}
\HH_k(\SCOne^\ell; \mathbb{F})& 1 \leq \ell \leq N \\
0 & \text { else }
\end{cases}$$ with structure maps given by the induced maps on reduced homology $\iota_\ast^\ell$ where applicable and zero maps elsewhere.
\end{definition}

A fundamental result in the study of persistent homology tells us that, for finite data, persistent homology can be decomposed as a collection interval modules. 

\begin{theorem}[Interval decomposition for persistent homology \cite{compute_PH}]\label{thm:structure}
Let $\filtOne$ be a filtered finite simplicial complex and $k$ a non-negative integer. Then there exists a triple $$(\bc_k(\filtOne), \beta, \delta)$$ called the \emph{barcode} of the persistence module $P\HH_k(\filtOne)$, where
\begin{enumerate}
    \item $\bc_k(\filtOne)$ is a finite set of \emph{bars},
    \item $\beta, \delta: \bc_k(\filtOne) \to \mathbb{Z}$ are functions that respectively record the \emph{birth} and \emph{death parameters} of each bar $\tau \in \bc_k(\filtOne),$ so that
    \item $\beta(\tau) < \delta(\tau)$ for all $\tau \in \bc_k(\filtOne)$, 
\end{enumerate} 
that is unique up to isomorphism\footnote{That is, set isomorphism on $\bc_k(\filtOne)$ along with precomposition by the inverse of that isomorphism for $\beta$ and $\delta.$},
along with an isomorphism of persistence modules
$$\mathcal{B} : \mathbb{I}_{\bc_k(\filtOne)} \xrightarrow{\cong}  P\HH_k(\filtOne),$$
where $\mathbb{I}_{ \bc_k(\filtOne)} = \bigoplus_{\tau\in \bc_k(\filtOne)} I[\beta(\tau), \delta(\tau))$ is the \emph{barcode module} for $\filtOne$ in dimension $k$. This isomorphism is called an \emph{interval decomposition} of $P\HH_k(\filtOne).$
\end{theorem}

By abuse, we will usually refer to the barcode $(\bc_k(\filtOne), \beta, \delta)$  using only the name of its underying set, $\bc_k(\filtOne).$ Further, we will often refer to the collection of bars which are \emph{alive} at a particular parameter $\ell \in \mathbb{Z}$ by
 $$\bc_k(\SCOne^\ell) = \{\tau\in \bc_k(\filtOne) \, | \,  \beta(\tau) \leq \ell < \delta(\tau)\}.$$

The \emph{length} or \emph{lifetime} of a bar $\tau \in \bc_k(\filtOne)$ is $\delta(\tau) - \beta(\tau).$ Further, for a filtered topological space with $N$ filtration levels, when $\delta(\tau) = N+1$, it is common to write instead $\delta(\tau) = \infty$ and say that $\tau$ has \emph{infinite} length or lifetime. We will require this convention in Section \ref{cycle_to_cycle_ext} when we discuss a filtration for which the final level is not a contractible space. 

Bars and barcodes are the central actors in our story. Throughout this paper we will carefully think about how bars correspond to homology classes. In particular, it will be important to think about how the vector space structure on homology is reflected by the bars in a barcode. To that end, we require the following terminology regarding persistence modules.

\begin{definition}[\cite{zigzag}] \label{def:persistent_affine}
Let $V^\bullet$ be a persistence module with structure maps $\{\phi^\ell\}$. A \emph{persistent subspace} (respectively, \emph{persistent affine subspace}) of $V^\bullet$ is given by a choice of $\beta < \delta \in \mathbb{Z}$ and $W^\bullet = \bigoplus_{\ell =-\infty}^\infty W^\ell$, where 
\begin{enumerate}
    \item $W^\ell \subseteq V^\ell$ is a subspace (respectively, affine subspace),
    \item $W^\ell = 0$ if $\ell < \beta$ or $\ell \geq \delta$, and
    \item $\phi^\ell(W^\ell) = W^{\ell+1}$ if $\ell \neq \beta-1, \delta-1$. 
\end{enumerate}
\end{definition}

\noindent That is, a persistent (affine) subspace of $V^\bullet$ is a choice of (affine) subspaces of each $V^\ell$ that are consistent with the structure maps and supported on a bounded interval. If $\text{dim}(W^\ell) = d$ for some fixed $d$ whenever $\beta \leq \ell < \delta$, we say $W^\bullet$ is a \emph{persistent $d$-dimensional subspace}.

Observe that given a persistence module $P\HH_k(\filtOne)$, a bar $\tau \in \bc_k(\filtOne)$, and an interval decomposition $\mathcal{B} : \mathbb{I}_{\bc_k(\filtOne)} \to P\HH_k(\filtOne),$ the image of the corresponding interval module $I[\beta(\tau), \delta(\tau))$ under $\mathcal{B}$ is a   persistent 1-dimensional subspace of $P\HH_k(\filtOne)$. Thus, another way to describe an interval decomposition is as a choice of identification of bars in the barcode with persistent 1-dimensional subspaces in the persistence module. Before we move on, it will be useful to take a closer look at what this choice entails.

\subsection{The relationship between bars and homology classes} 
\label{sec:class_correspondence}

Theorem \ref{thm:structure}  describes $P\HH_k(\filtOne)$ as a finitely generated $\mathbb{F}[x]$-module by providing a minimal set of generators for the module and their annihilators. While the barcode which enumerates these pairs of births and deaths is unique, the explicit identification of these generators as elements of the module is usually quite the opposite. Indeed, this is precisely a choice of basis for a finite-dimensional vector space that is compatible with the structure maps of the persistence module. Since we need cycle representatives, we must check that our constructions are invariant under this choice. Thus, we need analogues of the usual change of basis formalism from linear algebra.

\begin{definition}
\label{def:basis_vector_for_bar}
Let $\filtOne$ be a filtered topological space, and let $\bc_k(\filtOne)$ and $\mathcal{B}: \mathbb{I}_{\bc_k(\filtOne)} \to P\HH_k(\filtOne)$ be the barcode and interval decomposition for $P\HH_k(\filtOne)$, made up of the following commutative diagram of vector spaces and linear isomorphisms
\begin{equation}
\begin{tikzcd}
\mathbb{I}_{\bc_k(\filtOne)} \arrow[d, "\mathcal{B}"] & I^1 \arrow[r, rightarrow]  \arrow[d, "\mathcal{B}^1"]&  \cdots \arrow[r, rightarrow] & I^\ell \arrow[r, rightarrow] \arrow[d, "\mathcal{B}^\ell"]&\cdots \arrow[r, rightarrow]  & I^N \arrow[d, "\mathcal{B}^N"] \\
P\HH_k(\filtOne) & \HH_k(\SCOne^1)  \arrow[r, rightarrow]& \cdots \arrow[r, rightarrow]& \HH_k(\SCOne^\ell)  \arrow[r, rightarrow] & \cdots \arrow[r, rightarrow] & \HH_k(\SCOne^{N})
\end{tikzcd}
\end{equation}
where $I^\ell$ is a direct sum of copies of $\mathbb{F}$ indexed over  $\bc_k(\SCOne^\ell)$, 
$$I^\ell = \bigoplus_{\bc_k(\SCOne^\ell)}\mathbb{F}.$$
For each $\ell \in \mathbb{Z}$, let $\{ \vec{e}^{\; \ell}_\tau : \tau \in \bc_k(\SCOne^\ell)\}$ be the standard basis of $\smash{I^{\ell}}$, where $\vec{e}^{\;\ell}_\tau$ is the standard basis vector corresponding to the $\tau$-summand.  We call $\vec{e}^{\;\ell}_{\tau}$ the \emph{basis vector for $\tau$ at $\ell$.} 
\end{definition}
Conceptually, we do not require an ordering on the bars and adding one requires another layer of notation that we would like to avoid. However, when we are working with matrix computations, it will be convenient to set a linear order $\tau_1 < \tau_2 < \dots < \tau_{|\bc_k(\SCOne^\ell)|}$ on the bars alive at parameter $\ell$. In this case, $\smash{\vec{e}^{\; \ell}_{\tau_r}}$ is the vector whose $\smash{r^{\text{th}}}$ component is $1$ and all other components are zero.

We can now set some terminology for the correspondence between bars in $\bc_k(\filtOne)$ and classes in $P\HH_k(\filtOne)$.

\begin{definition}
\label{def:class_correspondence}
Let $\filtOne$ be a filtered topological space, and let $\bc_k(\filtOne)$ and $\IntDecZ:\mathbb{I}_{\bc_k(\filtOne)} \to P\HH_k(\filtOne)$ be the barcode and interval decomposition of $P\HH_k(\filtOne)$. Let $\ell \in \mathbb{Z}$ and $\tau \in \bc_k(Z^{\ell})$. Take $\vec{e}_{\tau}^{\; \ell}$ to be the basis vector for $\tau$ at $\ell$. Write $[\tau^{\IntDecZ, \ell}] = [\IntDecZ^\ell(\vec{e}_{\tau}^{\;\ell})]\in \HH_k(\SCOne^\ell)$ for the class that \emph{$\IntDecZ$-corresponds to $\tau$ at parameter $\ell$} and say that any cycle representative of $[\tau^{\IntDecZ, \ell}]$ is a \emph{$\IntDecZ$-representative of the bar $\tau$ at parameter $\ell$.} 
\end{definition}

Definition \ref{def:class_correspondence} finds the homology class that correspondences to a given bar. Conversely, we can also find bar representations of a given homology class as the following.

\begin{definition} 
\label{def:bar_representation_of_class}
Let $\filtOne$ be a filtered topological space, and let $\bc_k(\filtOne)$ and $\IntDecZ:\mathbb{I}_{\bc_k(\filtOne)} \to P\HH_k(\filtOne)$ be the barcode and interval decomposition of $P\HH_k(\filtOne)$. Given any homology class $[z] \in \HH_k(Z^\ell)$, we say that the collection 
\begin{align*}
S^{\IntDecZ}_{[z]} = \{ c_1\tau_1, \dots, c_K \tau_K \; | & \; \tau_1, \dots, \tau_K \in \bc_k(Z^\ell), \\
& [z] = \sum_{i=1}^Kc_i [\tau^{\IntDecZ, \ell}_{i}] \text{ for nonzero }c_1, \dots, c_K \in \field \}
\end{align*}
is the ($\IntDecZ$-)\emph{bar representation of} $[z]$ (\emph{at} $\ell$). 
\end{definition}
We can visually represent the persistent 1-dimensional subspace generated by a $\IntDecZ$-bar representation $S_{[z]}^{\IntDecZ} = \{c_1\tau_1, \dots, c_K\tau_K \}$ by highlighting the right half-bars of $\tau_1, \dots, \tau_K\in \bc_k(\filtOne)$ starting from parameter $\ell$ (see Figures \ref{fig:trefoil_knot_example} and \ref{fig:PCA_example} in Section \ref{applications} for such visual depictions). When we work with $\field_2$ coefficients where bar representations are in correspondence with various subsets of the bars, this visual representation provides complete information.

The correspondence between bars and homology classes in Definitions \ref{def:class_correspondence} and \ref{def:bar_representation_of_class} depend on the specific interval decomposition $\IntDecZ$. To avoid choice-dependent constructions, we consider all homology classes corresponding to $\tau$ under different interval decompositions of $P\HH_k(\filtOne)$. As noted, interval decompositions are vector space isomorphisms that are compatible with the structure maps. Thus, any interval decomposition $\mathcal{C}$ of $P\HH_k(\filtOne)$ can be obtained from a given interval decomposition $\IntDecZ$ by the composition $\mathcal{C} = \IntDecZ \circ \mathcal{L}$ for some automorphism  $\mathcal{L}: \mathbb{I}_{\bc_k(\filtOne)} \to \mathbb{I}_{\bc_k(\filtOne)}$ of the barcode module, given by the following commutative diagram of vector spaces and linear isomorphisms.

\begin{equation}
\label{diag:barcode_iso}
\begin{tikzcd}
\mathbb{I}_{\bc_k(\filtOne)} \arrow[d, "\mathcal{L}"] & I^1 \arrow[r, rightarrow]  \arrow[d, "\mathcal{L}^1"]&  \cdots \arrow[r, rightarrow] & I^\ell \arrow[r, rightarrow] \arrow[d, "\mathcal{L}^\ell"]&\cdots \arrow[r, rightarrow]  & I^N \arrow[d, "\mathcal{L}^N"] \\
\mathbb{I}_{\bc_k(\filtOne)} & I^1  \arrow[r, rightarrow]& \cdots \arrow[r, rightarrow]& I^\ell \arrow[r, rightarrow] & \cdots \arrow[r, rightarrow] & I^N
\end{tikzcd}
\end{equation}

We will require the following two technical lemmas\footnote{Per Remark \ref{rem:other_work}, a generalized version of these results can be found in \cite{JNT_barcode_bases}.} to work with the set of all such automorphisms in Section \ref{section:bar_to_bars_extension}. 
As their statements involve matrices, we fix linear orderings on each $\bc_k(\SCOne^\ell)$. Let $L^\ell$ denote the corresponding matrix for the linear automorphisms $\mathcal{L}^\ell$, and let $L^\ell_{r,c}$ indicate the matrix entry at row $r$ and column $c$. 

\begin{restatable}{lemma}{basischange} 
\label{lemma:basis_change} 
Given a filtered topological space $\filtOne$, let $\mathcal{L}:\mathbb{I}_{\bc_k(\filtOne)} \to \mathbb{I}_{\bc_k(\filtOne)}$ be an automorphism of its barcode module. Let $\ell \in \mathbb{Z}$, and let $\tau_1, \dots, \tau_{|\bc_k(Z^\ell)|}$ be some ordering of the bars in $\bc_k(Z^\ell)$. If $L^\ell_{r,c} \neq 0$ for $r \neq c$, then the bars $\tau_r$ and $\tau_c$ must satisfy the relations
\begin{equation}
\label{interval_relations}
\beta(\tau_r) \leq \beta(\tau_c)  < \delta(\tau_r) \leq \delta(\tau_c).
\end{equation}
\end{restatable}

\begin{proof}
Deferred to Appendix \ref{pf:lemma:basis_change}.
\end{proof}


\begin{restatable}{lemma}{diagonal}
\label{lemma:diagonal} 
Given a filtered topological space $\filtOne$, assume that if  $\tau \neq \tau' \in \bc_k(\filtOne)$ then either $\beta(\tau) \neq \beta(\tau')$ or $\delta(\tau) \neq \delta(\tau')$. Let $\mathcal{L}:\mathbb{I}_{\bc_k(\filtOne)} \to \mathbb{I}_{\bc_k(\filtOne)}$ be an automorphism of the barcode module. Then $L^\ell_{j,j} \neq 0$ for every $1 \leq j \leq |\bc_k(Z^\ell)|$ and $\ell \in \mathbb{Z}$. 
\end{restatable}

\begin{proof}
Deferred to Appendix  \ref{pf:lemma:diagonal}. 
\end{proof}

In Section \ref{section:bar_to_bars_extension}, we will be enumerating matrices of isomorphism $\mathcal{L}^\ell$ that extend to a barcode module automorphism $\mathcal{L}$. The following lemma will be useful.

\begin{restatable}{lemma}{linindep}
\label{lemma:linear_independence}
Let $\bc_k(\filtOne)$ be the barcode of some persistence module $P\HH_k(\filtOne)$. Assume that $\beta(\tau) \neq \beta(\tau')$ or $\delta(\tau) \neq \delta(\tau')$ for all $\tau \neq \tau' \in \bc_k(Z^\bullet)$. Given $\ell \in \mathbb{Z}$, let $\tau_1, \dots, \tau_{|\bc_k(Z^\ell)|}$ be the bars in $\bc_k(Z^\ell)$. For $1 \leq t \leq |\bc_k(Z^\ell)|$, let $L \in \field^{|\bc_k(Z^\ell)| \times t}$ be any matrix whose entries satisfy
\begin{equation}
\label{eq:matrixL}
L_{r,r} \neq 0 \text{ and } L_{r, c} = 0 \text{ if } \beta(\tau_r) > \beta(\tau_c) \text{ or } \delta(\tau_r) > \delta(\tau_c) \text{ for all }r, c.
\end{equation}
Then, the columns of $L$ are linearly independent. 
\end{restatable}

\begin{proof}
Deferred to Appendix \ref{pf:lemma:linear_independence}.\end{proof} 
To apply the above lemmas, we will need to assume that all bars have unique birth-death pairs. 
To ensure that the enumeration of $\mathcal{L}^\ell$ is finite, we assume that we are working over a finite field. 
If we further restrict our attention to homology with $\mathbb{F}_2$-coefficients, as is common in applied topology, our persistent affine subspaces will correspond to the power set of some subset of the bars in a barcode. This is appealing both for intuition building and from the perspective of studying data visually using these tools.

\ul{Therefore, for the remainder of the paper we will assume that all fields are finite and all bars have unique birth-death pairs. In addition, our computational examples will use $\mathbb{F}_2$-coefficients. In all cases, we will suppress the field from our notation unless we need to refer to it specifically.} We will occasionally recall these assumptions when they are pertinent for our discussion.

\subsection{Birth and death parameters of homology classes.}

The barcode of a persistence module has been the primary method for characterizing its structure in practice, so it is common to discuss the birth and death time of bars. The same terminology is often applied to persistent homology classes only  informally to avoid dealing directly with the interval decomposition. However, to work with these classes algorithmically we will need to make this terminology precise. Thus, we will use the correspondence between bars and homology classes in Definition \ref{def:class_correspondence} to state and describe how to compute the birth and death parameters of a class $[w]\in \HH_k(\SCOne^\ell).$

\begin{definition}
\label{def:cycle_birth_death}
Let $P\HH_k(Z^\bullet)$ be a persistence module with structure maps $\{\phi^\ell \}$. Given a class $[w] \in \HH_k(Z^\ell)$, let $W^\bullet$ be the persistent subspace of $P\HH_k(Z^\bullet)$ satisfying the following:
\begin{enumerate}
    \item $W^\ell = \text{span}_\field \Big\{[w] \Big\}$
    \item $W^i = \text{span}_\field \Big\{ [x] \in \HH_k(Z^i) \; | \; [w] = \phi^{\ell-1} \circ \cdots \circ \phi^{i} ([x]) \Big\} $ for $i < \ell$
    \item $W^i = \phi^{i-1}(W^{i-1})$ for $i > \ell.$
\end{enumerate}
    The \emph{birth} and \emph{death} parameters of $[w]$ are given by $\beta([w]) = \min \{i \leq \ell \; | \; W^i \neq 0 \} $ and $\delta([w])= \min \{i > \ell \; | \; W^i =0 \}.$ 
\end{definition}

That is, $W^\bullet$ is the maximal persistent subspace of $P\HH_k(Z^\bullet)$ containing $[w]$. To compute the birth and death parameters, we use the following lemma.  

\begin{restatable}{lemma}{classbirthdeath}
\label{lemma:class_birth_death}
Given a persistence module $P\HH_k(\filtOne)$, let $\mathcal{B}: \mathbb{I}_{\bc_k(\filtOne)} \to P\HH_k(\filtOne)$ be any interval decomposition. Given $[w] \in \HH_k(Z^\ell)$, express $[w]$ using the bars alive at $\ell$ as $[w] = \sum_{\tau \in \bc_k(Z^\ell)} f_\tau [\tau^{\mathcal{B}, \ell}]$ for some coefficients $f_\tau \in \field$. The birth and death times of $[w]$ are given by $\max \{ \beta(\tau) \; | \; \tau \in \bc_k(Z^\ell) \text{ and } f_\tau \neq 0\}$ and $\max \{ \delta(\tau) \; | \; \tau \in \bc_k(Z^\ell) \text{ and }f_\tau \neq 0 \},$ respectively.
\end{restatable}

Note that even though we fix an interval decomposition $\mathcal{B}$ for the computation, the birth and death times of $[w]$ is independent of the choice of $\mathcal{B}$. The proof of this lemma is not particularly enlightening to read, so we leave it as an exercise to the interested reader.

\subsection{Clique and witness complexes}
\label{clique_witness}
In this paper we restrict our attention to two special classes of filtered topological spaces: filtered \emph{clique complexes} and filtered \emph{witness complexes}. The former is the standard tool for constructing simplicial complexes from data represented by a collection of pairwise dissimilarity measures among elements of a population. The ubiquitous example is the filtered Vietoris-Rips complex of a point cloud in a metric space. The latter is less commonly used but naturally encodes common cross-population dissimilarity measures in a fashion that is compatible with the usual applications of the clique complex.

\begin{definition}
Let $A$ be an $(n \times n)$ matrix with zero diagonal and $N$ unique non-negative real off-diagonal entries $\alpha_1 < \dots < \alpha_N$. The \emph{(filtered) clique complex} for $A$ is the filtered simplicial complex $X(A)^\bullet$
where $X(A)^\ell$ has vertices $[n] = \{1, \dots, n\}$ and faces $\{\sigma \subseteq 2^{[n]}\; :\; \max_{\{r,c\}\in\sigma}A_{r,c} \leq \alpha_\ell\}.$
\end{definition}

Our standard context for constructing a clique complex will involve a population $P = \{p_1, \dots p_n\}$ endowed with a pairwise dissimilarity matrix $M_P$. In this case, we will write $X_P^\bullet$ for $X(M_P)^\bullet.$

\begin{definition}[\cite{Dowker, Witness, functorial_dowker}]
Let $B$ be an $(n \times m)$ matrix with $N$ unique non-negative real entries $\beta_1 < \dots < \beta_N$.  The \emph{(filtered) witness complex} of $B$ is the filtered simplicial complex $W(B)^\bullet$ where $W(B)^\ell$ has vertices $[n] = \{1, \dots, n\}$ and faces $\{\{r \in [n] : B_{r,c} \leq \beta_\ell\} \;:\;  c \in [m]\}.$ We refer to the set $[n]$ as the set of \emph{landmarks} and $[m]$ as the set of \emph{witnesses}.
\end{definition}

Note that some of the faces of $W(B)^\ell$ may be duplicates or subfaces of other faces. When working with two populations $P = \{p_1, \dots p_n\}$ and $Q = \{q_1, \dots q_m\}$ equipped with a $(n \times m)$ cross-dissimilarity matrix $M_{P,Q}$, we will write $W_{P, Q}^\bullet$ for $W(M_{P,Q})^\bullet.$

The apparent asymmetry between rows and columns -- the sets of landmarks and witnesses -- in the construction may seem off-putting: there's no intrinsic reason to choose one of $P$ or $Q$ as the vertex set. Fortunately, this choice no longer matters once we pass to persistent homology.

\begin{theorem}[Functorial Dowker Theorem \cite{Dowker, functorial_dowker}, paraphrased]\label{thm:dowker} Let $B$ be an $(n \times m)$ matrix with non-negative real entries. Then
$P\HH_k(W(B)^\bullet) \cong P\HH_k(W(B^T)^\bullet)$
as persistence modules for every dimension $k$. In particular $\bc_k(W(B)^\bullet) \cong \bc_k(W(B^T)^\bullet).$ 
\end{theorem}

This result is remarkable: it states that the roles of landmark and witness are interchangeable from a topological perspective. Dowker's theorem was originally stated for unweighted complexes, and it has been extended to the persistent homology setting \cite{functorial_dowker}. Further, the isomorphism is explicit, so we can compute it.  We will leverage this perspective shift as a foundation for comparing barcodes across populations \footnote{Throughout this paper, we will use the term \emph{Dowker's theorem} to refer to both the original theorem and the functorial theorem.}. 

\medskip
Note for both $X_P^\bullet$ and $W_{P,Q}^\bullet$, the corresponding persistence modules  are necessarily finitely generated. Therefore, their homology groups are finite-dimensional vector spaces. Further, observe that both $X_P^\bullet$ and $W_{P,Q}^\bullet$ are contractible for large enough filtration parameters. Since we are working with reduced homology, neither construction admits a bar $\tau$ with $\delta(\tau) = \infty.$

\section{Understanding the problem}
\label{approaches}

Now that we have the appropriate language, we pause to discuss more thoroughly the motivation for this work. Recall that we are interested in comparing two populations, $P = \{p_1, \dots, p_n\}$ and $Q = \{q_1, \dots, q_m\}$; for example, $P$ and $Q$ can represent point clouds sampled from some distributions on metric spaces or agents in some complex systems. We will take as our data three matrices: $M_P$ and $M_Q$, symmetric non-negative dissimilarity matrices for $P$ and $Q$, and $M_{P,Q}$, a non-negative $(n \times m)$ matrix whose $(r,c)$ entry measures dissimilarity between $p_r \in P$ and $q_c \in Q$. 

Suppose now that we compute the degree $k$ persistent homology of the clique complexes $X_P^\bullet$ and $X_Q^\bullet$ and obtain barcodes $\bc_k(X^\bullet_P)$ and $\bc_k(X^\bullet_Q).$ Given a bar of interest $\tau \in \bc_k(X^\bullet_P),$ corresponding to some 1-dimensional persistent subspace of $P\HH_k(X^\bullet_P)$, how can we determine which bars in $\bc_k(X^\bullet_Q)$ represent topological features similar to $\tau$, as measured by $M_{P,Q}$?

Perhaps the most straightforward option is to apply recently developed methods for comparing barcodes using the functoriality of homology, including induced matching \cite{induced_matching, PD_as_D}, cycle registration \cite{cycle_registration}, basis-independent partial matching \cite{basis_indep_matching}, or quiver-representations of correspondences \cite{sampled_persistence}. In this section, we take $P$ and $Q$ as point clouds in a common metric space to explore why these methods may not provide the answers a practitioner might expect or desire.

\subsection{The problem of comparing clique complexes on $P$ and $Q$}

First, let us briefly recall how the above-mentioned methods work. Given a map of persistence modules $f: U^\bullet \to V^\bullet$, \emph{induced matching} \cite{induced_matching, PD_as_D, Bauer_2021} factorizes $f$ as $U^\bullet \twoheadrightarrow \left(\text{im } f\right)^\bullet \hookrightarrow V^\bullet$ and defines a matching between $\bc_k(U^\bullet)$ and $\bc_k(V^\bullet)$ via canonical injections through $\bc_k((\text{im }f)^\bullet)$. The canonical injections match bars only when their endpoints are aligned.
\emph{Cycle registration} \cite{cycle_registration} uses induced matching to find matchings between $\bc_k(U^\bullet)$ and $\bc_k(V^\bullet)$ given maps of the form $U^\bullet \to Y^\bullet \leftarrow V^\bullet$. \emph{Basis-independent partial matching} \cite{basis_indep_matching}, on the other hand, considers the map $f: U^\bullet \to V^\bullet$ as a commutative diagram and reports the dimension of the appropriate subspace that corresponds to given bars in $\bc_k(U^\bullet)$ and $\bc_k(V^\bullet)$. Lastly, if $X$ and $Y$ are topological spaces with a continuous map $g: X \to Y$ and $P$ is a finite sample of $X$, then one can use \emph{quiver-representations of correspondences} \cite{sampled_persistence} to approximate $g_* : H_k(X) \to H_k(Y)$ from the sampled data $g|_P$. We can adapt this construction to the context of persistent homology by considering maps of persistence modules $PH_k(X^\bullet) \leftarrow PH_k(G^\bullet) \to PH_k(Y^\bullet),$ where $G^\bullet$ represents a filtered simplicial complex associated to the graph of the sampled map $g|_P$.

When we have access to explicit maps on persistence modules or maps on finite samples of topological spaces, these methods provide the best possible comparison of topological features; they recapitulate the usual functoriality on homology in the context of barcodes. However, we often build our persistence modules using samples from unknown distributions on unknown spaces, so the required maps of persistence modules of the form $U^\bullet \to V^\bullet$ or $U^\bullet \rightarrow Y^\bullet \leftarrow V^\bullet$ are not known. Indeed, if the point clouds $P$ and $Q$ represent time series data on two different systems, there may not be an explicit map between $P$ and $Q$ at all. 

In simple cases, we might still attempt to impute the necessary maps from the data. If $P = \{p_1, \dots, p_n\}$ and $Q = \{q_1, \dots q_m\}$ are sampled from a common metric space, then the metric provides enough information to na\"ively construct a diagram of persistence modules as
\begin{equation}
\label{map_pm}
P\HH_k(X_P^\bullet) \rightarrow P\HH_k(X_{P\cup Q}^\bullet) \leftarrow P\HH_k(X_Q^\bullet).
\end{equation}
However, attempts to apply the above-listed methods to this diagram suffer from two major issues.

First, the lifetimes of the topological features in $P$ and $Q$ must be very closely aligned in the parameter range. Consider two probability distributions on an annulus illustrated in green on the top of Figure \ref{fig:union_barcodes}. Let $P$ and $Q$ be samples from the two probability distributions. Because they surround the hole in the middle of the annulus, we intuit that the two samples share a common topological feature. However, as the bottom panel in as Figure \ref{fig:union_barcodes} illustrates, the bars in $\bc_1(X_{P\cup Q}^\bullet)$ and $\bc_1(X_Q^\bullet)$ are not aligned because the scale of the two point clouds is different. As a result, existing methods will fail to identify the longest bar in $\bc_1(X_P^\bullet)$ with the longest bar in $\bc_1(X_Q^\bullet)$, even though intuition suggests they represent the same feature. When working with measurements taken from multiple sources, different calibrations of instruments or small changes in experimental conditions could easily result in these kinds of shifts in scale.

\begin{figure}[t]
\centering
\floatbox[{\capbeside\thisfloatsetup{capbesideposition={right,top},capbesidewidth=0.5\textwidth, capbesidesep=mysep}}]{figure}[\FBwidth]
{\caption{
\textbf{Complexities arise in applying induced matching with related point clouds.} (top) Point clouds $P$, $Q$, $R$ sampled on an annulus. $P$ is sampled from the distribution in green concentrated around the inner boundary, while $Q$ and $R$ are sampled from identical, centrally concentrated distributions. (bottom) Barcodes in dimension 1 for the Vietoris-Rips complexes of the point clouds $P$, $P\cup Q$, $Q$, $Q \cup R$, and $R$. Because the unique bar in $\bc_1(X_Q^\bullet)$ is misaligned with the others, induced matching through the unions will fail to identify the feature in $Q$ with non-trivial features in either $P$ or $R$. Blue bars indicate the only possible induced match between $Q \cup R$ and $R.$ }
\label{fig:union_barcodes}}
{\includegraphics[width=0.41\textwidth]{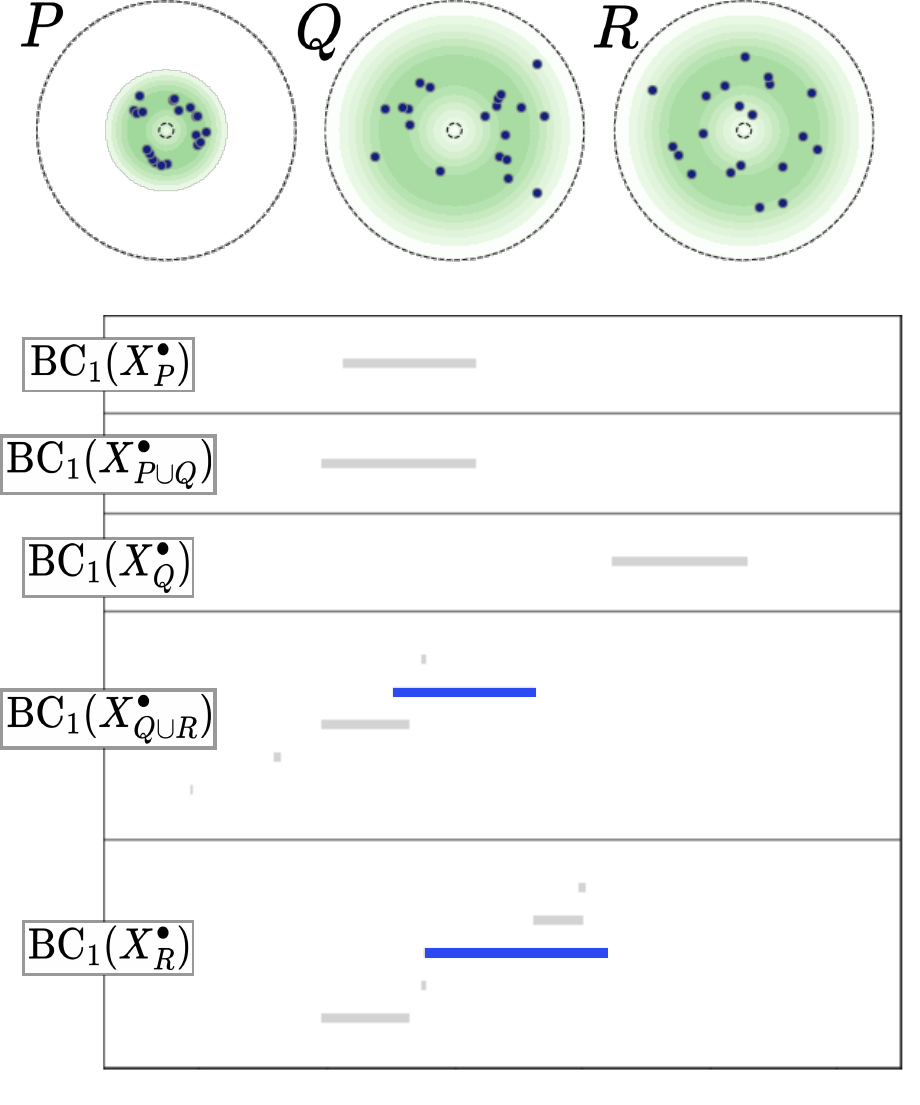}
}
\end{figure}

The second issue with the construction in Equation \ref{map_pm} is that we must have a dense sample of points to ensure that persistent features have similar scale. Consider the populations $Q$ and $R$ in Figure \ref{fig:union_barcodes} sampled from the same distribution. Each individual point cloud has a barcode that suggests a strong 1-dimensional feature. However, when we apply induced matching, only the highlighted blue bars in $\bc_1(X^\bullet_{Q \cup R})$ and $\bc_1(X^\bullet_R)$ can possibly match; the unique bar in $\bc_1(X^\bullet_Q)$ does not have any matching due to an apparent scale difference. Note, however, that $P$ and $R$ \emph{can} be matched using existing methods; thus, the existence or non-existence of a matching is unreliable even in this simple setting.

\subsection{The problem of comparing clique complexes and witness complex.}

Some of the earliest papers in applied topology suggest an alternative to  the union for investigating the relationship between two point clouds in a common metric space: the witness complex $W_{P, Q}^\bullet$ \cite{Witness}. Empirically,  witness complexes have been observed to mitigate many (though not all) of the issues with scale and density we describe above. However, we no longer have an inclusion map and cannot directly apply induced matching to compare persistent homology\footnote{One can construct a map $P\HH_k(W^\bullet_{P,Q}) \to P\HH_k(X^\bullet_P)$ by doubling the parameter. However, doubling the scale at which we consider a complex can be very destructive to homology.}.

Simple examples show that features in $\bc_k(W^\bullet_{P,Q})$ need not correspond in any meaningful way to those in $\bc_k(X^\bullet_P)$. Taking $P$ and $Q$ to be dense samples from complementary semi circles on the unit circle, the barcode $\bc_1(W_{P,Q}^\bullet)$ will have a single nontrivial bar, while $\bc_1(X^\bullet_P)$ and $\bc_1(X^\bullet_Q)$ will both be empty. On the other hand, taking $P$ to be concentrated along the equator of a sphere and at its poles, and $Q$ only at the poles, $\bc_1(W_{P,Q}^\bullet)$ will be empty while $\bc_1(X_P^\bullet)$ will have a persistent bar. To get traction, we will need a method for comparing homology classes between these complexes.

In Section \ref{Extension}, we propose a solution to this problem we call the  \emph{persistent extension} method. Persistent extension uses a zig-zag through the intersection of two simplicial complexes on the same vertex set to identify classes which contain identical representatives at the chain level for some range of parameters, thus identifying topological features one could reasonably consider to be ``the same." In Section \ref{analogous_intervals}, we apply this method and Dowker's theorem to develop the method of \emph{analogous bars}, which provides our desired method for comparing topological features of $X_P^\bullet$ and $X_Q^\bullet$ using only the data in the cross-dissimilarity measure $M_{P,Q}.$

\section{Persistent extension}    
\label{Extension}
In this section, we develop the persistent extension method for comparing persistent homology classes between two filtered simplicial complexes supported on the same vertex set $P$. In Section \ref{cycle_to_cycle_ext}, we introduce the cycle-to-cycles extension method, which, given a homology class in one simplicial complex, enumerates classes in the other that share its cycle representatives. In Section \ref{section:bar_to_bars_extension}, we extend this algorithm to a bar-to-bars extension method, which, given a bar in one barcode, enumerates representations in the other barcode. In Section \ref{variations} we discuss considerations for reducing the computational overhead of applying the persistent extension method. 

\subsection{Cycle-to-cycles extension}    
\label{cycle_to_cycle_ext} 

Suppose $\filtOne$ and $\filtTwo$ are filtered simplicial complexes on vertices $P$. Fix some parameter $\psi$ and take $[\tau] \in \HH_k(Z^\psi)$. The goal of the cycle-to-cycles extension is to find all parameters $\ell$ and classes $[y] \in \HH_k(Y^\ell)$ for which $[\tau]$ and $[y]$ share cycle representatives. We begin by setting some terminology with which to describe the steps in our algorithm. 

To simplify this section, we assume that the filtered simplicial complex $\filtTwo$ has the property that $\SCTwo^N$ is the complete simplex on its vertices, as is the case for both clique and witness complexes as defined in Section \ref{clique_witness}. It is straightforward but tedious to modify the following for settings where this is not the case. 

\begin{definition}
\label{restriction_def}
Let $\filtOne$ and $\filtTwo$ be filtered simplicial complexes on common vertex set $P$ and fix a parameter $\psi$ of $\filtOne$. Given a parameter $\ell$, let $\mapx_{\ell}$ and $ \mapy_{\ell}$ each denote the following induced maps
\[\HH_k(\SCOne^{\psi}) \xleftarrow{\mapx_{\ell}} \HH_k(\SCOne^{\psi} \cap \SCTwo^\ell) \xrightarrow{\mapy_{\ell}} \HH_k(\SCTwo^\ell).\]
Given a homology class $[\tau] \in \HH_k(Z^\psi)$, define the \emph{set of restrictions (of $[\tau]$) at} $\ell$ as 
\[
\restrictiondef_{\ell} = \{ [w]  \in \HH_k(\SCOne^{\psi} \cap \SCTwo^\ell) \: | \:  \: [\tau] = \mapx_{\ell} [w] \},
\]
and define the \emph{set of cycle extensions (of $[\tau]$) at} $\ell$ as 
\[ \cedef_{\ell} = \{ [y] \in \HH_k(\SCTwo^\ell) \: | \: [y] = \mapy_{\ell} [w] \text{ for } [w] \in \restrictiondef_{\ell}  \}. \]
We refer to any $[y] \in E_\ell$ for which $[y] = \mapy_{\ell} [w]$ as a \emph{cycle extension } (\emph{via}  $[w]$). We refer to $E = \bigcup_{\ell} E_{\ell}$ as the \emph{set of cycle extensions (of $[\tau]$)}.
\end{definition}

For a fixed parameter $\ell$, the set $\restrictiondef_\ell$ can be expressed as the solution to a chain-level linear equation, and $\cedef_\ell$ is then given by the solutions to a set of linear equations indexed by $\restrictiondef_\ell$. To naively compute $E$, we would have to iterate this process across every parameter $\ell$, which would be infeasible for even moderately sized data. However, the following two lemmas characterize how restrictions and cycle extensions at different parameters are related, allowing us to limit our computations to a much smaller set of linear systems. 

Given parameters ${\ell} \leq \ell'$, let
\begin{align*}
\eta^{\ell}_{\ell'}: & \quad \HH_k(\SCOne^{\psi} \cap Y^\ell) \to \HH_k(\SCOne^{\psi} \cap Y^{\ell'}), \\
\lambda^{\ell}_{\ell'}: & \quad \HH_k(\SCTwo^\ell) \to \HH_k(Y^{\ell'})
\end{align*}
denote the structure maps for the persistence modules, induced by inclusion.
\begin{lemma}
\label{extension_later}
Let $\filtOne, \filtTwo$ be filtered simplicial complexes on a common vertex set $P$. Given $[\tau] \in \HH_k(Z^\psi)$, fix parameter ${\ell}$ and let $[y] \in \HH_k(\SCTwo^\ell)$ be a cycle extension of $[\tau]$ at $\ell$ via some $[w] \in \HH_k(\SCOne^{\psi} \cap Y^\ell) $. If $\ell' > {\ell}$, then $\lambda^{\ell}_{\ell'}[y]$ is a cycle extension of $[\tau]$ at $\ell'$ via $\eta^{\ell}_{\ell'}[w]$.
\end{lemma}

That is, any cycle extension at $\ell$ passes to a cycle extension at any later parameter ${\ell'}$ via the structure maps. 

\begin{restatable}{lemma}{extensionearlier}
\label{extension_earlier}
Let $\filtOne, \filtTwo$ be filtered simplicial complexes on a common vertex set $P$. Given  $[\tau] \in \HH_k(Z^{\psi})$, let $[y'] \in \HH_k(Y^{\ell'})$ be a cycle extension of $[\tau]$ at $\ell'$ via some $[w'] \in \HH_k(Z^\psi \cap Y^{\ell'})$. Assume that there exists a parameter $\ell < \ell'$ and $[w] \in \HH_k(Z^\psi \cap Y^{\ell})$ such that $\eta^{\ell}_{\ell'}[w] = [w']$. Then, there exists a cycle extension $[y] \in \HH_k(Y^\ell)$ of $[\tau]$ at $\ell$ via $[w]$ satisfying $[y'] = \lambda^{\ell}_{\ell'}[y]$.
\end{restatable}

The proofs of both lemmas follows directly from commutativity of the relevant diagrams of structure maps and morphisms induced by inclusions. Details are left to the interested reader.

Lemma \ref{extension_earlier} states that if a restriction $[w']$ is in the image of the structure map $\eta^{\ell}_{\ell'},$ then any cycle extension at parameter ${\ell'}$ via $[w']$ is in the image of the structure map $\lambda^{\ell}_{\ell'},$ so this extension is induced by an extension at $\ell.$ 

In particular, suppose $[w] \in \HH_k(\SCOne^{\psi} \cap \SCTwo^\ell)$ is a restriction at $\ell$. Let $\hat{\beta}[w] \leq \ell$ denote the birth parameter of $[w]$ in $P\HH_k(Z^\psi \cap Y^\bullet)$. Then, there exists some $[\hat{w}] \in \HH_k(Z^\psi \cap Y^{\hat\beta[w]})$ satisfying $[w] = \eta^{\hat\beta[w]}_{\ell} [\hat{w}]$. Let $[\hat{y}]$ be a cycle extension of $[\tau]$ at $\hat\beta[w]$ via $[\hat{w}]$. Together, Lemmas \ref{extension_later} and \ref{extension_earlier} tell us that any cycle extension $[y]$ of $[\tau]$ at ${\ell}$ via $[w]$ must satisfy $[y] = \lambda^{\hat\beta[w]}_{\ell}[\hat{y}]$. That is, in order to find cycle extensions of $[\tau]$ via $[w]$ at all parameters, it suffices to find the cycle extensions via $[\hat{w}]$ at $\hat{\beta}[w]$. 

We leverage this observation to enumerate all possible cycle extensions while considering only a substantially restricted set of parameters. To do so, we  utilize the auxiliary filtration 
\begin{equation}
\label{auxfilt}
    \intFilt : \SCOne^{\psi} \cap \SCTwo^{1} \to \SCOne^{\psi} \cap \SCTwo^{2} \to \cdots \to \SCOne^{\psi} \cap \SCTwo^{N} = \SCOne^{\psi} 
    \end{equation}
along with its barcode ${(\bc_k(\intFilt), \intFiltBirth, \intFiltDeath)}$. The benefit of using this filtration is that if $\intFiltDeath(\rho) = \infty$ for $\rho \in \bc_k(\intFilt),$  then $\rho$ corresponds to a cycle in $\HH_k(\SCOne^{\psi})$ under any choice of interval decomposition of $P\HH_k(Z^\psi \cap Y^\bullet)$. 

We now refer the reader to Algorithm \ref{alg:cycle_to_cycle_extension}, where we describe the cycle-to-cycles extension method in full. In step (1), we compute the auxiliary filtration in Equation \ref{auxfilt}. In step (2), we find the collection of parameters that correspond to the birth times of all restrictions of $[\tau]$. In step (3), we find all restrictions at a given birth parameter and find the corresponding cycle extensions.

\begin{algorithm}[t]
\caption{Extension method (cycle-to-cycles extensions) }
\label{alg:cycle_to_cycle_extension}
\textbf{Input:} 
\begin{itemize}
    \item filtered simplicial complexes $\filtOne$, $\filtTwo$ on vertex set $P$, 
    \item a parameter $\psi,$ and
    \item a homology class $[\tau] \in \HH_k(\SCOne^{\psi})$
\end{itemize}
\textbf{Output:} 
\begin{itemize}
    \item a collection of parameters $p_Y$, and 
    \item a collection $\cealgo$ of cycle extensions of $[\tau]$ to $P\HH_k(\filtTwo)$.
\end{itemize}
\textbf{Steps:}
\begin{enumerate}
    \item Compute the persistent homology $P\HH_k(\intFilt)$, the barcode $(\bc_k(\intFilt), \intFiltBirth, \intFiltDeath)$, and an interval decomposition ${\IntDecDelta:\mathbb{I}_{\bc_k(\intFilt)} \to P\HH_k(\intFilt)}$ of the auxiliary filtration in Equation (\ref{auxfilt}),
    \item Find a sufficient collection of parameters $p_Y$ at which to look for cycle extensions of $[\tau]$.
	\begin{enumerate}
     \item Let  $S^{\IntDecDelta}_{[\tau]} =    \{\IntDecDeltaLower^*_1\intXdelta_1, \dots, \IntDecDeltaLower^*_m\intXdelta_m  \}$ be an $\IntDecDelta$-bar representation of $[\tau]$ at $N$. 
     \item Let $\Rho^{\IntDecDelta}_{\tau} = \{\intXdelta_1, \dots, \intXdelta_m \}$ be the set of bars in this $\IntDecDelta$-bar representation.%
\item Let $\ell_0 = \max_{\intXdelta \in \Rho^{\IntDecDelta}_{\tau}} \{ \intFiltBirth(\intXdelta) \}.$ 
\item Let $\omega = \{\mu \in \bc_k(\intFilt) \;|\; \ell_0 < \intFiltBirth(\mu) <  \intFiltDeath(\mu) < \infty\},$  
\item Let ${p_Y} = \{ \ell_0 \} \cup \{ \intFiltBirth(\mu) \;|\; \mu \in \omega\}$.
\end{enumerate}

\item Find a complete set of cycle extensions of $[\tau]$ to $P\HH_k(\filtTwo).$
\begin{enumerate}
\item Let $\Rho_{\text{short}} = \{ \intXdelta \in \bc_k(\intFilt) \; | \;  \ell_0 < \intFiltDeath(\intXdelta) < \infty \}.$
\item For each $\ell\in p_Y$,
\begin{enumerate}
    \item Let $\Rho^{\ell}_{\text{short}} = \{ \intXdelta \in \Rho_{\text{short}} \; | \; \intFiltBirth(\intXdelta) \leq \ell < \intFiltDeath(\intXdelta) < \infty \}.$
    \item Let $V^{\ell}_{\text{short}} = \{[\rho^{\mathcal{F}, \ell}] \;|\; \intXdelta \in \Rho^{\ell}_{\text{short}} \}$.
    \item Let $\restrictionalgo^{\IntDecDelta}_{\ell} =  \sum_{i=1}^m \IntDecDeltaLower^*_i[\rho_i^{\IntDecDelta, \ell}] + \text{span}_{\field}V_{\text{short}}^\ell.$ 
    \item Let $\cealgo_{\ell} = \{ \mapy_{\ell}([w])  \; | \; [w] \in \restrictionalgo_{\ell}^{\IntDecDelta} \}$
    where $\mapy_{\ell}: \HH_k(\SCOne^{\psi} \cap \SCTwo^{\ell}) \to \HH_k(\SCTwo^{\ell})$ is the map induced by inclusion.
 \end{enumerate}
\end{enumerate}
\item Return $p_Y$ and $\cealgo = \bigcup_{\ell \in p_Y} \cealgo_{\ell}.$

 \end{enumerate}
\end{algorithm}

 There are several facets of Algorithm \ref{alg:cycle_to_cycle_extension} that require justification. We need to show that the collections $p_Y$ and $\restrictionalgo^{\IntDecDelta}_{\ell}$ are independent of the choice of interval decomposition $\IntDecDelta$ of $P\HH_k(Z^\psi \cap Y^\bullet)$, that the affine subspace $\restrictionalgo^{\IntDecDelta}_{\ell}$ coincides with the set $\restrictiondef_{\ell}$ of restrictions\footnote{In particular, note that we avoid solving a linear system to extract the set of solutions.} as defined in Definition \ref{restriction_def}, and that the output $\cealgo$ of the algorithm provides the complete set $E$ of cycle extensions of $[\tau]$ from Definition \ref{restriction_def}, as stated in the following Theorem.
 
 \begin{restatable}{theorem}{mainthm}
 \label{main_thm}
 Let $\filtOne, \filtTwo$ be filtered simplicial complexes on a common vertex set $P$. Fix a parameter $\psi$ and homology class $[\tau] \in \HH_k(Z^\psi)$. The output of Algorithm \ref{alg:cycle_to_cycle_extension} suffices to recover all cycle extensions of $[\tau]$. That is, let $E$ be the set of all cycle extensions of $[\tau]$ per Definition \ref{restriction_def}. Given parameters $\ell < \ell'$, let $\lambda^{\ell}_{\ell'}: \HH_k(\SCTwo^{\ell}) \to \HH_k(\SCTwo^{\ell'})$ be the map induced by inclusion. Let 
\[\cealgo^* = \{[y] \in \HH_k(\SCTwo^{\ell'}) \; | \; \text{ there exists }\ell \in p_Y, [y_{\ell}] \in \cealgo_{\ell} \text{ such that } [y] = \lambda^{\ell}_{\ell'}([y_{\ell}]) \},\]
where $p_Y$ and $\cealgo = \cup_{\ell \in p_Y} \cealgo_\ell$ are the outputs of Algorithm \ref{alg:cycle_to_cycle_extension}. Then $E = \cealgo^*$.
 
 \end{restatable}
 
The proof is technical, so we defer it to Appendix \ref{proof_of_extension_algorithm} where we break it down into a sequence of smaller theorems, and we instead proceed with an example that illustrates the method. 

\begin{example} 
\label{ex:cycle-to-cycles}
We visualize the key steps of Algorithm \ref{alg:cycle_to_cycle_extension} using a set of points $P$ sampled from a flattened torus and a set of points $Q$ sampled near a essential circle on that torus, as depicted in Figure \ref{fig:ex_alg_1}. We assume that homology is computed with $\field_2$ coefficients and that all bars have unique birth-death pairs. 

Let $\filtOne = W^\bullet_{P,Q}$, the Witness filtration with $P$ as landmarks and $Q$ as witnesses computed from the cross-dissimilarity matrix $M_{P,Q}$, and let $\filtTwo = X^\bullet_P,$ computed from the pairwise dissimilarity matrix $M_P$. Let $[\tau] \in \HH_1(W^\psi_{P,Q})$ be the homology class illustrated in Figure \ref{fig:ex_alg_1}.

We now execute Algorithm \ref{alg:cycle_to_cycle_extension}. The right panel in Figure \ref{fig:ex_alg_1} illustrates the barcode $\bc_1(W_{P,Q}^\psi \cap X^\bullet_P)$, and we will refer to it throughout. In step (1), we fix an interval decomposition $\IntDecDelta: \mathbb{I}_{\bc_1(W_{P,Q}^{\psi} \cap X^\bullet_P)} \to P\HH_1(W_{P,Q}^{\psi} \cap X^\bullet_P)$. In step (2), we gather information about this barcode. For our choice of $\IntDecDelta$, the unique $\IntDecDelta$-bar representation of $[\tau]$ at $N$ is $S^{\IntDecDelta}_{[\tau]} = \{ \rho_1\}$, where $\rho_1$ is the green bar. The parameter $\ell_0 = \hat{\beta}(\rho_1)$ at which this bar is born is indicated by the vertical dotted line. Note that $\ell_0$ is the minimum parameter at which there exists a restriction of $[\tau]$. The set $\omega$ consists of the blue bars above the green bar, each being born and dying within the interval $(\ell_0, \infty)$. Finally, take our list of parameters $p_Y$ to be the set of birth times of the green and blue bars. The collection $p_Y$ corresponds to the birth parameters of restrictions of $[\tau]$ in $P\HH_k(Z^\psi \cap Y^\bullet)$. 

For step (3), we take $\Rho_{\text{short}}$ to be the set of all bars that die after the green bar is born. In the first pass through the loop in step (3-b), taking $\ell_0 \in p_Y$ as our parameter, the purple bars in Figure \ref{fig:ex_alg_1} indicate the intervals $\Rho^{\ell_0}_{\text{short}} = \{ \intXdelta_2, \dots, \intXdelta_{15} \}$ selected in step (3-b-i). The restrictions at $\ell_0$ in step (3-b-iii) thus constitute the affine subspace \[ \restrictionalgo_{\ell_0} = \{  [\intXdelta_1^{\IntDecDelta, \ell_0}] +  c_2[\intXdelta_2^{\IntDecDelta, \ell_0}] +  \dots + c_{15} [\intXdelta_{15}^{\IntDecDelta, \ell_0}]  \: | \: c_2, \dots, c_{15} \in \field_2 \}.\]

The inset panel in Figure \ref{fig:ex_alg_1} illustrates cycle representatives of $[\intXdelta_1^{\IntDecDelta, \ell_0}]$ in green and $[\intXdelta_2^{\IntDecDelta, \ell_0}]$ in purple. Note that both $[\intXdelta_1^{\IntDecDelta, \ell_0}]$ and $[\intXdelta_1^{\IntDecDelta, \ell_0}] + [\intXdelta_2^{\IntDecDelta, \ell_0}]$ describe cycles whose representatives are similar to that of $[\tau]$; the latter is simply a deformation of the former to pass through a local cycle with short lifetime.  Indeed, all elements of $\restrictionalgo_{\ell_0}$ are such deformations. Finally, in step (3-b-iv), we push these into $\HH_k(Y^{\ell_0})$ as elements in $\cealgo_{\ell_0}$, which is the collection of cycle extensions at $\ell_0$. 

Step 3(b) repeats the above process for all $\ell \in p_Y$. For a fixed $\ell$, the set $\restrictionalgo_{\ell}$ corresponds to the set of restrictions at $\ell$, and $\cealgo_{\ell}$ corresponds to the set of cycle extensions at $\ell$.

\end{example}

\begin{figure}[t]
\centering
 
\includegraphics[width=0.9\textwidth]{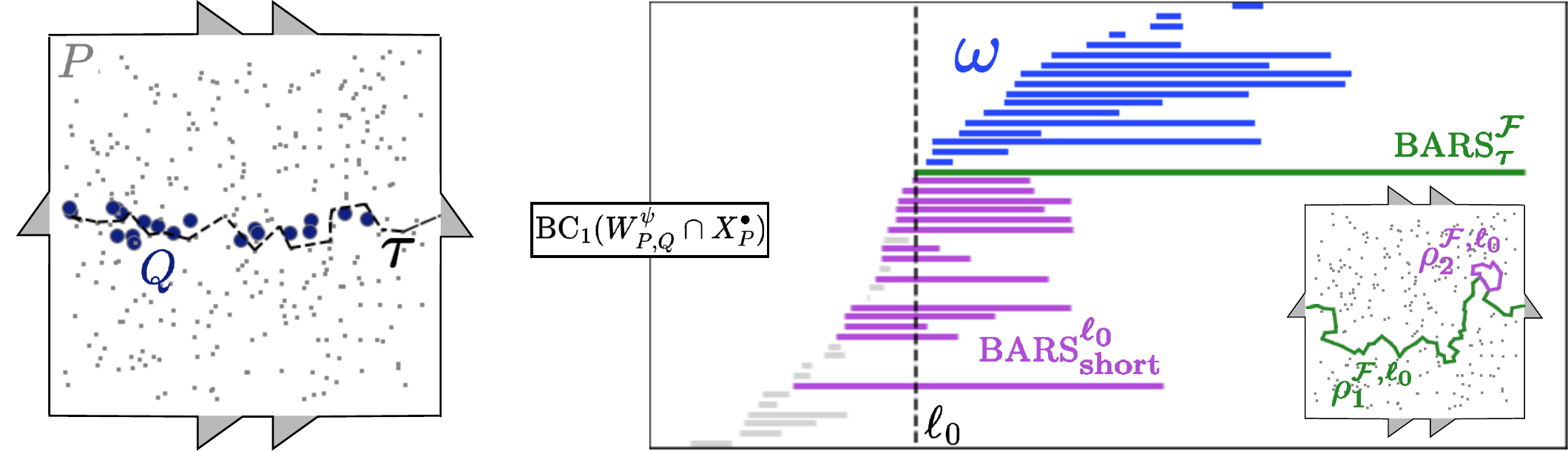}

\caption{\textbf{Elements of Algorithm 1 in Example \ref{ex:cycle-to-cycles}.} (left) Grey points $P$ sampled from a flattened torus and blue points $Q$ sampled near an essential circle of that torus, along with a representative $\tau$ of a class $[\tau]\in \HH_1(W^\psi_{P,Q})$ that represents the essential circle. (right) The barcode of the auxiliary filtration, $\bc_1(W^\psi_{P,Q}\cap X^\bullet_P)$. For a fixed interval decomposition $\IntDecDelta$ of $\bc_1(W^\psi_{P,Q}\cap X^\bullet_P)$, the long green bar corresponds to the only element $\rho_1 \in \Rho^\IntDecDelta_\tau,$ which is born at parameter $\ell_0.$ Short purple bars with lifetimes intersecting $\ell_0$ make up $\Rho^{\ell_0}_{\text{short}}$, and finite blue bars born after $\ell_0$ constitute the set $\omega$. (inset) Representatives of $[\rho_1^{\IntDecDelta,\ell_0}]$ in green and  $[\rho_2^{\IntDecDelta,\ell_0}]$ in purple corresponding to a bar in $\Rho^{\ell_0}_{\text{short}}.$ }
\label{fig:ex_alg_1}
\end{figure}

\subsection{Bar-to-bars extensions}
\label{section:bar_to_bars_extension}
It is common in applications to refer to bars, not cycles, as the features of a data set. Thus, for practitioners it may be of use to understand relations between bars rather than cycles. The bar-to-bars extension method enumerates ways in which a selected bar $\tau \in \bc_k(\filtOne)$ coincides at the chain level with elements of $P\HH_k(\filtTwo)$ and represents the corresponding cycles in terms of $\bc_k(Y^\bullet)$. To do so, we first find all homology class representations of $\tau$ in $P\HH_k(Z^\bullet)$. For each homology class representation, we find all cycle extensions, and, for each cycle extension, we find the corresponding bar representation in $\bc_k(\filtTwo)$. We discuss the notion of bar extensions before presenting the algorithm.

\medskip 
The first step is to find all homology classes that correspond to a given bar $\tau \in \bc_k(Z^\bullet)$. If we fix an interval decomposition $\mathcal{B}$ of $P\HH_k(\SCOne^\bullet)$ and some parameter $\psi$ for $\filtOne$, then we automatically have the homology class $[\tau^{\IntDecZ, \psi}]$ that $\mathcal{B}$-corresponds to $\tau$.

Which parameter $\psi$ should we choose? In the absence of further information, the
safest choice is to select $\psi$ which maximizes the size of our homology class of interest\footnote{In specific applications where we understand the complexes involved more completely, it may be sensible to select a different parameter at which to consider classes corresponding to $\tau$.} but avoiding introducing enough equivalences to trivialize our feature of interest. Thus, we consider the last parameter at which the bar $\tau$ is alive.

\begin{definition}
\label{def:terminal class}
Let $\IntDecZ: \mathbb{I}_{\bc_k(\filtOne)} \to P\HH_k(\filtOne)$ be an interval decomposition. The \emph{terminal class of  a bar $\tau$ in $\bc_k(\filtOne)$  (under $\IntDecZ$)} is the class $[\tau^\IntDecZ_\ast] = [\tau^{\IntDecZ, \delta(\tau)-1}]\in \HH_k(\SCOne^{\delta(\tau)-1})$. When $\IntDecZ$ is unambiguous, we will suppress it and write $[\tau_\ast]$.
\end{definition}

Recall that a bar $\tau \in \bc_k(\filtOne)$ corresponds to a 1-dimensional persistent subspace $V^\bullet$ of $P\HH_k(Z^\bullet)$ under $\mathcal{B}$. 

\begin{remark}
\label{remark:decompositions}
Without a preferred interval decomposition $\IntDecZ,$ it will be necessary to consider all possible choices of $\IntDecZ$ to obtain a comprehensive list of possible terminal classes of $\tau.$ In Section \ref{simplifying_assumptions}, we will show that the terminal class is independent of the interval decomposition $\IntDecZ$ if homology is computed with $\field_2$ coefficients and if the bars of $\bc_k(\filtOne)$ have unique death parameters. Thus, for readers interested in that very common case, it is reasonable to ignore the dependence on $\IntDecZ$ in what follows. 
\end{remark}
Once we fix the terminal class $[\tau_*^{\mathcal{B}}]$ of $\tau$, one can compute the cycle extensions of $[\tau_*^{\mathcal{B}}]$. In the context of comparing barcodes, we need to describe these cycle extensions in terms of the bars in $\bc_k(\filtTwo)$. Given an interval decomposition $\mathcal{D}:\mathbb{I}_{\bc_k(\filtTwo)} \to P\HH_k(\filtTwo)$ and a cycle extension $[y] \in \HH_k(Y^\ell)$ of $[\tau_*^{\mathcal{B}}]$ at $\ell$, recall from Definition \ref{def:bar_representation_of_class} that $S^{\mathcal{D}}_{[y]}$ denotes the $\mathcal{D}$-bar representation of $[y]$. Using this terminology, we can describe cycle extensions and bar extensions of a bar $\tau \in \bc_k(Z^\bullet)$. 

\begin{definition}
\label{def:extensions_of_bar}
Let $Z^\bullet, Y^\bullet$ be filtered simplicial complexes on a common vertex set $P$. Let $\tau \in \bc_k(Z^\bullet)$. Given an interval decomposition $\mathcal{B}$ of $P\HH_k(Z^\bullet)$, let $E_{\mathcal{B}}$ denote the set of cycle extensions of $[\tau^{\mathcal{B}}_*]$. Let $B$ and $D$ each denote the set of all possible interval decompositions of $P\HH_k(Z^\bullet)$ and $P\HH_k(Y^\bullet)$. Define the \emph{cycle extensions of} $\tau$ as 
\[ E(\tau, Y^\bullet) = \bigcup_{\mathcal{B} \in B} E_{\mathcal{B}}, \]
and define the \emph{bar extensions of} $\tau$ as 
\[S(\tau, Y^\bullet) = \{ S^{\mathcal{D}}_{[y]} \; | \; [y] \in E(\tau, Y^\bullet), \mathcal{D} \in D \}. \]
When given a homology class $[\tau] \in \HH_k(Z^\psi)$, we use the corresponding terminology \emph{bar extensions of} $[\tau]$ to refer to
\[ S([\tau], Y^\bullet) = \{S^{\mathcal{D}}_{[y]} \; | \; [y] \in E, \mathcal{D} \in D \}, \]
where $E$ denotes the set of cycle extensions of $[\tau]$ from Definition \ref{restriction_def}.
\end{definition}

\medskip
Now that we defined the bar extensions of $\tau$, we can discuss the overall structure of our bar-to-bars extension algorithm. Given $\tau \in \bc_k(Z^\bullet)$ we first find all terminal classes of $\tau$ by enumerating all possible interval decompositions of $P\HH_k(Z^\bullet)$. We then use the resulting homology classes as input for Algorithm \ref{alg:cycle_to_cycle_extension} and generate all cycle extensions of $\tau$. Finally, for each cycle extension, we enumerate all interval decompositions of $P\HH_k(Y^\bullet)$ to find the bar representations of the cycle extensions. 

In Sections \ref{sec:alt_decomp_one} and \ref{sec:alt_decomp_two}, we discuss details of performing this enumeration efficiently and summarize the process in Algorithm \ref{alg:bar_to_bar_extension}. Finally, in Section \ref{simplifying_assumptions}, we show that a much simpler process (Algorithm \ref{alg:bar_to_bar_extension_F2}) suffices when we compute persistent homology with $\mathbb{F}_2$ coefficients and when our bars have unique death parameters. We first discuss the computational details of enumerating the interval decompositions. 

\subsubsection{Finding all terminal classes of a bar}
\label{sec:alt_decomp_one}
Given a bar $\tau \in \bc_k(\filtOne)$, we must find an explicit parameter $\psi$ and homology class $[\tau] \in \HH_k(Z^\psi)$ to be used as input to the cycle-to-cycles extension method (Algorithm \ref{alg:cycle_to_cycle_extension}). As discussed in Definition \ref{def:terminal class}, we fix $\psi = \delta(\tau)-1$ and the terminal class $[\tau^{\IntDecZ}_*] \in \HH_k(Z^{\delta(\tau)-1})$ under some fixed interval decomposition $\IntDecZ: \mathbb{I}_{\bc_k(\filtOne)} \to P\HH_k(\filtOne)$. To avoid imposing an arbitrary choice of $\mathcal{B}$, it is necessary to find all terminal classes $[\tau^{\mathcal{C}}_*]$ under different interval decompositions $\mathcal{C}$ of $P\HH_k(\filtOne)$ and apply Algorithm \ref{alg:cycle_to_cycle_extension} to these terminal classes. Note that the iteration over all interval decompositions of $P\HH_k(\filtOne)$ allows us to find all 1-dimensional persistent subspaces of $P\HH_k(\filtOne)$ that can be represented by $\tau$. We then find all homology classes representing such persistent subspaces.

Recall from Section \ref{sec:class_correspondence} that any interval decomposition $\mathcal{C}$ of $P\HH_k(Z^\bullet)$ can be obtained from $\mathcal{B}$ as $\mathcal{C} = \IntDecZ \circ \mathcal{L}$ for some automorphism $\mathcal{L}$ of the barcode module  $\mathbb{I}_{\bc_k(\filtOne)}$. The terminal class under $\mathcal{C}$ can be computed by $[\tau^{\mathcal{C}}_*] = [\mathcal{C}^{\psi}(\vec{e}^{\,\psi}_\tau)] = [\mathcal{B}^\psi \circ \mathcal{L}^\psi(\vec{e}^{\, \psi}_{\tau})] =[B^\psi \cdot L^\psi \cdot \vec{e}^{\, \psi}_{\tau}],$ where $\psi = \delta(\tau)-1$, $\vec{e}^{\, \psi}_{\tau}$ is the basis vector for $\tau$ at $\psi$, and $B^\psi$ and $L^\psi$ are the matrix representations of the maps $\mathcal{B}^\psi$ and $\mathcal{L}^\psi$. Thus, to find all possible terminal class $[\tau^{\mathcal{C}}_*]$, it suffices to find all linear isomorphisms $\mathcal{L}^{\psi}: (\mathbb{I}_{\bc_k(\filtOne)})^\psi \to (\mathbb{I}_{\bc_k(\filtOne)})^\psi$ resulting from the restriction of the automorphism $\mathcal{L}:\mathbb{I}_{\bc_k(\filtOne)} \to \mathbb{I}_{\bc_k(\filtOne)}$. We use the following lemma for this purpose\footnote{Per Remark \ref{rem:other_work}, a generalization of this result was established indepdently in \cite{JNT_barcode_bases}. }. 

\begin{restatable}{lemma}{LZ}
\label{lemma:LZ}
Given persistence module $P\HH_k(\filtOne)$, let $\bc_k(Z^{\psi}) = \{ \tau_1, \dots, \tau_m \}$ be some ordering of bars alive at $\psi$. Let $L_Z$ be the collection of $m \times m$ matrices
\[L_Z = \{ L \in \field^{m \times m} \; |  \; L_{r,r} \neq 0, L_{r,c} = 0 \text{ if } \beta(\tau_r) > \beta(\tau_c) \text{ or } \delta(\tau_r) > \delta(\tau_c) \text{ for all } r,c \}. \]
The collection $L_Z$ is precisely the collection of matrices representing the map $\mathcal{L}^{\psi}: (\mathbb{I}_{\bc_k(\filtOne)})^{\psi} \to (\mathbb{I}_{\bc_k(\filtOne)})^{\psi}$ under all possible choices of barcode module automorphism $\mathcal{L} : \mathbb{I}_{\bc_k(\filtOne)} \to \mathbb{I}_{\bc_k(\filtOne)}$. 
\end{restatable}
\begin{proof}
Deferred to Appendix \ref{Appendix:pf_LZ}.
\end{proof}

Since $L_Z$ contains all of the pertinent linear isomorphisms $\mathcal{L}^{\psi}$, we can enumerate the terminal  classes of interest via $T = \{[B^{\psi} \cdot L \cdot \vec{e}^{\;\psi}_{\tau}] \in \HH_k(Z^{\psi}) \; | \; L \in L_Z \} $, where $B^{\psi}$ denotes the matrix representation of $\mathcal{B}^{\psi}$. Step (2) of Algorithm \ref{alg:bar_to_bar_extension} performs this process\footnote{Since we are working with finite fields, it is possible to enumerate $L_Z$.}.

\medskip 
Once we find the collection $T$, we apply the cycle-to-cycles extension method (Algorithm \ref{alg:cycle_to_cycle_extension}) to each terminal class in $T$, which results in a collection of cycle extensions $[y] \in \HH_k(Y^\ell)$ at various parameters $\ell$. The final step we need to discuss is finding all bar representations of cycle extensions.

\subsubsection{Finding all bar-representations of a cycle extension.}
\label{sec:alt_decomp_two}
Let $[y] \in \HH_k(Y^{\ell})$ be a cycle extension output from Algorithm \ref{alg:cycle_to_cycle_extension}. Fix some interval decomposition $\mathcal{D}$ of $P\HH_k(\filtTwo)$, and let $S^{\mathcal{D}}_{[y]}$ denote the $\mathcal{D}$-bar representation of $[y]$. Our goal is to find all $\mathcal{G}$-bar representation of $[y]$ for all possible interval decompositions $\mathcal{G}$ of $P\HH_k(Y^\bullet)$. 

Any interval decomposition $\IntDecYTwo$ of $P\HH_k(\filtTwo)$ can be obtained
by $\IntDecYTwo= \IntDecY \circ \mathcal{L}^{-1},$ where $\mathcal{L}$ is an automorphism of $ \mathbb{I}_{\bc_k(\filtTwo)}.$ Given $\mathcal{L}^{\ell}$, we can find the $\mathcal{G}$-bar representation of $[y]$ from $S^{\mathcal{D}}_{[y]}$ as follows. Let $D^{\ell}$, $G^{\ell}$, and $L^{\ell}$ each denote the matrix representing the linear isomorphism  $\mathcal{D}^{\ell}$, $\mathcal{G}^{\ell}$, and $\mathcal{L}^{\ell}$. Let $S^{\IntDecY}_{[y]} = \{ \IntDecYLower_1 \intY_{j_1}, \dots, \IntDecYLower_t \intY_{j_t} \}$. That is, $d_1[\gamma_{j_1}^{\mathcal{D},\ell}] + \dots +  d_t[\gamma_{j_t}^{\mathcal{D},\ell}] = [y]$ in $\HH_k(Y^\ell)$ for nonzero $d_1, \dots, d_t$, i.e., 
\[D^{\ell}(\IntDecYLower_1\vec{e}^{\; \ell}_{{j_1}} + \dots + \IntDecYLower_t \vec{e}^{\; \ell}_{{j_t}}) = [y] \text{ in } \HH_k(Y^\ell) \text{ for nonzero }d_1, \dots, d_t,\]
where $\vec{e}^{\;\ell}_j$ is the basis vector for $\gamma_j$ at parameter $\ell$. Note that $S^{\IntDecYTwo}_{[y]} = \{ c_{1} \intY_{i_1}, \dots, c_{s} \intY_{i_s} \}$ if $G^{\ell}(c_{1} \vec{e}^{\; \ell}_{i_1} + \dots + c_{s} \vec{e}^{\; \ell}_{i_s}) = [y]$ for nonzero $c_1, \dots, c_s$. Since $L^{\ell} = ({G^{\ell}})^{-1} \circ D^{\ell}$, this means that $L^{\ell}(\IntDecYLower_1 \vec{e}^{\; \ell}_{j_1} + \dots + \IntDecYLower_t \vec{e}^{\; \ell}_{j_t}) = c_{1} \vec{e}_{i_1}^{\, \ell} + \dots + c_{s} \vec{e}_{i_s}^{\, \ell}$. Thus, given the matrix $L^{\ell}$ and $S^{\IntDecY}_{[y]}$, we can directly compute $S^{\IntDecYTwo}_{[y]}$.

\begin{algorithm}[t]
\caption{Extension method (bar-to-bars extensions)}
\label{alg:bar_to_bar_extension}
\textbf{Input}:
\begin{itemize}
    \item filtered simplicial complexes $\filtOne$, $\filtTwo$ on vertex set $P$, 
    \item a bar $\tau \in \bc_k(\filtOne)$
    \end{itemize}
\textbf{Output}:
\begin{itemize}
    \item a collection $E(\tau, \filtTwo)$ of all cycle extensions of $\tau$ to $P\HH_k(\filtTwo),$ and
    \item a collection $S(\tau, \filtTwo)$ of all bar extensions of $\tau$ to $\bc_k(\filtTwo).$
\end{itemize}
\textbf{Steps:}
\begin{enumerate}
\item Fix interval decompositions $\IntDecZ$ of $P\HH_k(\filtOne)$ and $\IntDecY$ of $P\HH_k(\filtTwo).$
\item 
Fix an ordering on $\bc_k(Z^{\delta(\tau)-1}) = \{ \tau_1, \dots, \tau_m \}$. 
\item Enumerate the collection of matrices
\[L_Z = \{ L \in \field^{m \times m} \; | \; L_{r,r} \neq 0, L_{r,c} = 0 \text{ if } \beta(\tau_r) > \beta(\tau_c) \text{ or } \delta(\tau_r) > \delta(\tau_c) \text{ for all } r,c \}.  \]
\item 
Enumerate the collection of possible terminal classes $$T = \{[B^{\delta(\tau)-1} \cdot L \cdot \vec{e}^{\; \delta(\tau)-1}_{\tau}] \in \HH_k(Z^{\delta(\tau)-1}) \; | \; L \in L_Z \} ,$$ where $ \vec{e}^{\; \delta(\tau)-1}_{\tau}$ is the basis vector for $\tau$ at parameter $\delta(\tau)-1$ (Definition \ref{def:basis_vector_for_bar}) and $B^{\delta(\tau)-1}$ denotes the matrix representation of $\mathcal{B}^{\delta(\tau)-1}$.
\item For each $[\tau] \in T$,
\begin{enumerate}
    \item Let $p_{\SCTwo, [\tau]}$ and $\cealgo_{[\tau]}$ be the output of cycle-to-cycles extension (Algorithm \ref{alg:cycle_to_cycle_extension}) with inputs $\filtOne$, $\filtTwo$, $\delta(\tau)-1$ and $[\tau].$  
\item For each $\ell \in p_{Y, [\tau]}$,
    \begin{enumerate}
    \item Fix an ordering on $\bc_k(Y^{\ell}) = \{ \gamma_1, \dots, \gamma_n\}$.
    \item Enumerate the collection of matrices
    \[L_Y = \{ L \in \field^{n \times n } \mid  L_{r,r} \neq 0, L_{r,c} = 0 \text{ if } \beta(\gamma_r) > \beta(\gamma_c) \text{ or } \delta(\gamma_r) > \delta(\gamma_c) \text{ for all } r,c \} \]
    \item For each $[y] \in (\cealgo_{[\tau]})_{\ell}$, compute the $\mathcal{D}$-bar representation $S^{\mathcal{D}}_{[y]} = \{ d_1 \gamma_{j_1},  \cdots, d_t \gamma_{j_t} \}$. 
    \item For each $L \in L_Y$, compute 
    \[S^{\mathcal{D} \circ L^{-1}}_{[y]} = \{ c_{1} \intY_{i_1}, \dots, c_{s} \intY_{i_s} \;|\;L(d_1 \vec{e}_{j_1}^{\; \ell} + \dots + d_t \vec{e}_{j_t}^{\; \ell}) = c_{1} \vec{e}_{i_1}^{\; \ell} + \dots + c_{s} \vec{e}_{i_s}^{\; \ell} \},\]
    where $\vec{e}^{\;\ell}_j$ is the basis vector for $\gamma_j$ at parameter $\ell$ (Definition \ref{def:basis_vector_for_bar}).
    \end{enumerate}     
    
\end{enumerate}
\item Return $E(\tau, \filtTwo) = \bigcup_{[\tau] \in T} \cealgo_{[\tau]}$ and $S(\tau, \filtTwo) = \{S_{[y]}^{\IntDecY\circ L^{-1}} \;|\; [\tau] \in T, \,\ell \in p_{Y, [\tau]}, \,
[y] \in (\cealgo_\mathcal{[\tau]})_\ell,\, L \in L_Y\}.$
\end{enumerate}
\end{algorithm}

\begin{algorithm}[t]
\caption{Extension method (bar-to-bars extension, $\field_2$-coeffs and unique death times)}
\label{alg:bar_to_bar_extension_F2}
\textbf{Input and Output:}\\
\hspace{0.6cm} As bar-to-bars extension (Algorithm \ref{alg:bar_to_bar_extension}.)

\textbf{Steps:}
\begin{enumerate}

\item Fix interval decompositions $\IntDecZ$ of $P\HH_k(\filtOne)$ and $\IntDecY$ of $P\HH_k(\filtTwo).$
\item Perform step (5) of bar-to-bars extension (Algorithm \ref{alg:bar_to_bar_extension}) with $T = \{ [\tau_\ast^\IntDecZ] \}.$


\item Return $E(\tau, \filtTwo) = \cealgo_{[\tau_*^{\mathcal{B}}]}$ and $S(\tau, \filtTwo) = \{S^{\;\IntDecY \circ L^{-1}}_{[y]} \;|\; \ell \in p_Y, [y] \in (\cealgo_{[\tau_*^{\mathcal{B}}]} )_{\ell}, L \in L_Y \}.$
\end{enumerate}
\end{algorithm}

As in the previous case, we can enumerate  $L^{\ell}$. Let $n = \dim \HH_k(Y^\ell)$ and fix an order of the bars $\bc_k(Y^\ell) = \{ \gamma_1, \dots, \gamma_n \}$. Enumerate $L_Y$:
\begin{align}
L_Y = \{ L \in \field^{n \times n} \; | & \; L_{r,r} \neq 0, L_{r,c} = 0 \text{ if } \beta(\gamma_r) > \beta(\gamma_c) \text{ or } \delta(\gamma_r) > \delta(\gamma_c) \nonumber \\
& \text{ for all } r,c \}. \label{eq:L_Y}
\end{align}
Lemma \ref{lemma:LZ}  showed that $L_Y$ finds all $\mathcal{L}^{\ell}$ that are restrictions of a persistence module automorphism $\mathcal{L}$ to the parameter $\ell$. For each $L \in L_Y$, we can compute the bar representation of $[y]$ under a new interval decomposition as 
\begin{align*}
S^{\mathcal{D} \circ L^{-1}}_{[y]} = \{ c_{1} \intY_{i_1}, \dots, c_{s} \intY_{i_s} \;| & \;L(d_1 \vec{e}_{j_1}^{\; \ell} + \dots + d_t \vec{e}_{j_t}^{\; \ell}) = c_{1} \vec{e}_{i_1}^{\; \ell} + \dots + c_{s} \vec{e}_{i_s}^{\; \ell} \\
& \text{ for } c_1, \dots, c_s \neq 0 \}.  
\end{align*}
We return the collection of all bar representations of $[y]$ 
\[S_{[y]} = \{S^{\mathcal{D} \circ L^{-1}}_{[y]} \; | \; L \in L_Y \}. \]
Step (5-b) of Algorithm \ref{alg:bar_to_bar_extension} performs this process.

\medskip

Lemma \ref{lemma:LZ} guarantees that the enumerations in Sections \ref{sec:alt_decomp_one} and \ref{sec:alt_decomp_two} find all terminal classes and bar-representations. Combining the two, we obtain Algorithm \ref{alg:bar_to_bar_extension}, which provides a complete enumeration of bar extensions of $\selectIntX \in \bc_k(\filtOne)$ to $\bc_k(\filtTwo)$. In step (1), we fix initial interval decompositions $P\HH_k(Z^\bullet)$ and $P\HH_k(Y^\bullet)$. In steps (2) to (4), we compute all terminal classes of bar $\tau$. In step (5), we compute the cycle extensions of each terminal class and find their bar representations in $\bc_k(Y^\bullet)$.

\subsubsection{The case of $\field_2$-coefficients and bars with unique death parameters}
\label{simplifying_assumptions}
Algorithm \ref{alg:bar_to_bar_extension} can be substantially simplified when we make two assumptions that are often satisfied when working with real data: that homology is computed with $\field_2$ coefficients, and that all bars in $\bc_k(\filtOne)$ have unique death parameters. Under such assumptions, the following lemma says we can omit the loop over interval decompositions of $P\HH_k(\filtOne)$. 

\begin{restatable}{lemma}{welldefinedcycles}
\label{lemma:well_defined_cycles} 
Assume that $\field = \field_2$ and that all bars of $\bc_k(\filtOne)$ have unique death parameters. Let $\mathcal{B}, \mathcal{C}: \mathbb{I}_{\bc_k(\filtOne)} \to P\HH_k(\filtOne)$ be two different interval decompositions. Given a bar $\tau \in\bc_k{(\SCOne^\bullet)}$, let $[\tau^{\mathcal{B}}_*]$ and $[\tau^{\mathcal{C}}_*]$ denote the terminal class for $\tau$ under $\mathcal{B}$ and $\mathcal{C}$ respectively. Then, $[\tau^{\mathcal{B}}_*] = [\tau^{\mathcal{C}}_*]$ in $H_k(\SCOne^{\delta(\tau)-1})$. 
\end{restatable}

\begin{proof}
Deferred to Appendix \ref{Pf_well_defined_cycles}.
\end{proof}

Algorithm \ref{alg:bar_to_bar_extension_F2} describes the modified bar-to-bars extension method under these assumptions.

\begin{example}
\label{ex:bars-to-bars}
Recall the point clouds $P$ and $Q$ on the torus from Example \ref{ex:cycle-to-cycles} and Figure \ref{fig:ex_alg_1}, and again let $Z^\bullet = W^\bullet_{P,Q}$, and $Y^\bullet = X^\bullet_P$. Let $\tau$ be the highlighted bar in $\bc_1(W^\bullet_{P,Q})$, as shown in Figure \ref{fig:ex_alg_3}. We will apply Algorithm \ref{alg:bar_to_bar_extension_F2} to determine how the selected bar $\tau$ can be related to bars in $\bc_1(X^\bullet_P)$. 

In step (1), we fix the interval decompositions of $P\HH_k(Z^\bullet)$ and $P\HH_k(Y^\bullet)$. In step (2), we find the terminal class $[\tau_*]$ of $\tau$; the cycle representative depicted in the left panel of Figure \ref{fig:ex_alg_1} is a representative of this $[\tau_*]$. We apply Algorithm \ref{alg:cycle_to_cycle_extension} to this terminal class and find all cycle extensions. Finally, we enumerate possible bar representations of the resulting cycle extensions in $\bc_1(\filtTwo)$. The bottom pane of Figure \ref{fig:ex_alg_3} illustrates two possible bar-representations of some cycle extension $[y]$. The union of the green and purple bars illustrates the $\mathcal{D}$-bar representation $S^{\mathcal{D}}_{[y]}$ under some fixed interval decomposition $\mathcal{D}$. The collection of green bars represent an alternative bar-representation $S^{\mathcal{D}\circ L^{-1}}_{[y]}$ for some $L \in L_Y$.

\begin{figure}
\centering
\floatbox[{\capbeside\thisfloatsetup{capbesideposition={right,top},capbesidewidth=0.5\textwidth,capbesidesep=mysep}}]{figure}[\FBwidth]
{\caption{\textbf{Input and output of the bar-to-bars extension in Example \ref{ex:bars-to-bars}.} (top) Barcode $\bc_1(W^\bullet_{P,Q})$ for the witness complex of point clouds $P$ and $Q$ from Figure \ref{fig:ex_alg_1}. We apply the bar-to-bars extension method to the blue long bar. (bottom) Barcode for the clique complex of the point cloud $P$ from Figure \ref{fig:ex_alg_1}. Green and purple bars together are the $\mathcal{D}$-bar representative $S^{\mathcal{D}}_{[y]}$ produced by Algorithm 3. Omitting the purple bar produces the alternative  bar representation $S^{\mathcal{D}\circ L^{-1}}_{[y]}$ for one choice of $L \in L_Y.$}
\label{fig:ex_alg_3}
}{
    \includegraphics[width=0.5\textwidth]{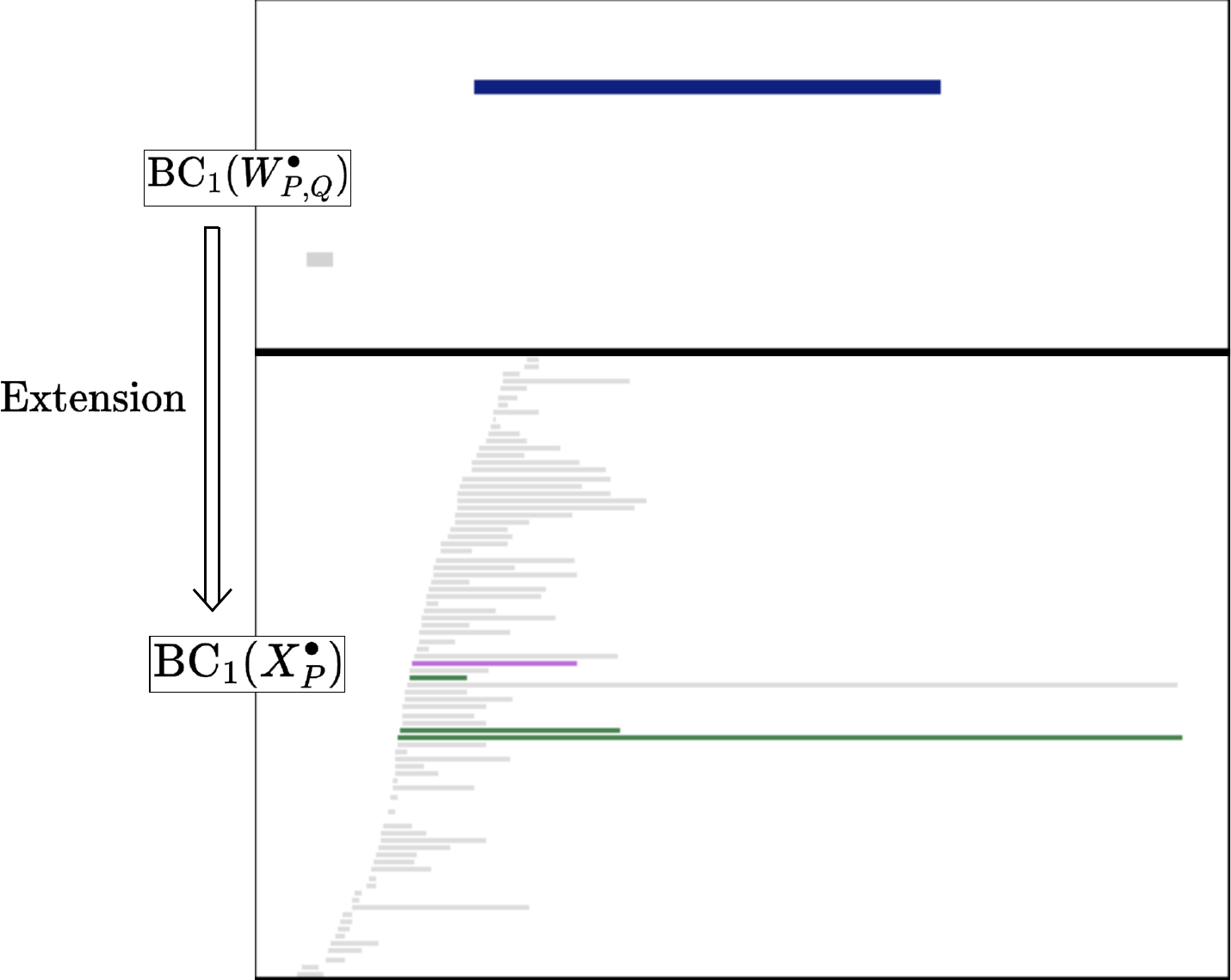}}
 \end{figure}

\end{example}

\subsection{Component-wise extension}
\label{variations}
The process of computing the $\mathcal{D}$-bar representation $S^\IntDecY_{[y]}$ can be  computationally intensive. In step (5-b-iii) of Algorithm \ref{alg:bar_to_bar_extension}, we find the bar extension $S^{\IntDecY}_{[y]}$ for every $[y] \in (\cealgo_{[\tau]})_{\ell}$. From steps (3-b-iii) - (3-b-iv) of Algorithm \ref{alg:cycle_to_cycle_extension}, we know that any $[y] \in (\cealgo_{[\tau]})_{\ell}$ has the form
\begin{equation}
[y] = \mapy_{\ell} \Bigg(  \sum_{i=1}^m\IntDecDeltaLower^*_i[\intXdelta_i^{\IntDecDelta, \ell}] \Bigg) + \sum_{[\rho^{\mathcal{F}, \ell}_j] \in 
V_{\text{short}}^{\ell}}  \mapy_{\ell} \Bigg(\IntDecDeltaLower_j [\rho^{\mathcal{F}, \ell}_j] \Bigg)
\end{equation}
for some $\IntDecDeltaLower_j \in \field$. Note that $f_1^*, \dots, f_m^* \in \field$ are fixed from $S^{\mathcal{F}}_{[\tau]} = \{f_1^* \rho_1, \dots, f_m^* \rho_m \}$ in step (2-a) of Algorithm \ref{alg:cycle_to_cycle_extension}. Thus, the number of cycle extensions $[y] \in (\cealgo_{[\tau]})_{\ell}$ we must consider grows exponentially as $|\field|^{|\Rho^{\ell}_{\text{short}}|}$.

Fortunately, we can leverage the vector space structure on homology to mitigate the resulting explosion in computation time. Suppose that $[y_1], [y_2] \in \HH_k(\SCTwo^\ell)$ and define 
$S^{\IntDecY}_{[y_1]} \oplus S^{\IntDecY}_{[y_2]} = S^{\IntDecY}_{[y_1+y_2]}.$ From Definition \ref{def:bar_representation_of_class}, we see that this operation defines a $\field$-vector space structure on the set $\{S^{\IntDecY}_{[y]} \;|\; [y] \in \HH_k(\SCTwo^\ell)\}$, and that the collection of $\mathcal{D}$-bar representations of $(\cealgo_{[\tau]})_{\ell}$ can be regarded as an affine subspace of this space.

Therefore, instead of computing bar representations for every $[y] \in (\cealgo_{[\tau]})_{\ell}$, we can compute bar representations $S^{\IntDecY}_{\mapy_{\ell}(\sum_{i=1}^m\IntDecDeltaLower^*_i[\intXdelta_i^{\IntDecDelta, \ell}] )}$ and $S^{\IntDecY}_{\mapy_{\ell}[\rho^{\mathcal{F},\ell}]}$ for each $[\rho^{\mathcal{F}, \ell}] \in V^{\ell}_{\text{short}}$. We'll refer to 
$$ B = \mapy_{\ell} \Big(\sum_{i=1}^m\IntDecDeltaLower^*_i[\intXdelta_i^{\IntDecDelta, \ell}] \Big) \text{ and }O =  \{  \mapy_{\ell} \Big( [\rho^{\mathcal{F},\ell}] \Big) \; | \; [\rho^{\mathcal{F},\ell}] \in V^{\ell}_{\text{short}}\} $$ as the \emph{baseline} and \emph{offset cycle extensions}, respectively, with $(\cealgo_{[\tau]})_{\ell} = B \oplus \text{span}_\field O$. We then take 
$$B^{\mathcal{D}} =  S^{\IntDecY}_{\mapy_{\ell}(\sum_{i=1}^m\IntDecDeltaLower^*_i[\intXdelta_i^{\IntDecDelta, \ell}] )} \text{ and  } O^{\mathcal{D}} = \{S^{\IntDecY}_{\mapy_{\ell}([\rho^{\mathcal{F},\ell}])} \;|\; [\rho^{\mathcal{F},\ell}] \in V^{\ell}_{\text{short}}\} $$
to be the collections of   \emph{baseline} and  \emph{offset bar extensions}. Note that $\{S^{\mathcal{D}}_{[y]} \; | \; [y] \in (\cealgo_{[\tau]})_{\ell} \} = B^{\mathcal{D}} \oplus \text{span}_\field O^{\mathcal{D}}$.
Perhaps unsurprisingly, we have found it simpler and faster to visualize the sets $B^{\mathcal{D}}$ and $O^{\mathcal{D}}$ for understanding extensions in examples. Note that the individual members of the baseline and offset cycle extensions depend in on the interval decomposition $\mathcal{F}$. However, when the baseline and offset cycle extensions are considered collectively, the collection of cycle extensions is independent of $\mathcal{F}$ according to Theorem \ref{restrictalgo_basis_independence}. Similarly, the collection of baseline and offset bar extensions is independent of $\mathcal{F}$. 

\begin{remark}
When the number of bars in the barcode is large, the number of linear systems we must solve to carry out the extension method can be prohibitive, even when we consider only these basis elements. Therefore, it may be useful to omit classes of bars based on length statistics for a null model or domain knowledge. We have found this to be a useful strategy in practice, but leave a detailed discussion for future work.
\end{remark}

\section{Analogous bars}
\label{analogous_intervals}

With the  persistent extension method developed in Section \ref{Extension} in hand, we can move on to addressing the motivating question for this paper. Let $Q = \{q_1, \dots, q_m \}$ and $P= \{p_1, \dots, p_n \}$ be two populations equipped with dissimilarity matrices $M_Q$ and $M_P,$ as well as a cross-dissimilarity matrix $M_{Q,P}$. Let $X^\bullet_Q$ and $X^\bullet_P$ each denote the Vietoris-Rips complexes on $Q$ and $P$. Given a bar  $\selectIntX\in \bc_k(X^\bullet_Q)$, our goal is to find all persistent 1-dimensional subspaces of $P\HH_k(X^\bullet_P)$, represented as collections of bars in $\bc_k(X^\bullet_P),$ which could correspond to features similar to $\tau$ under the information provided by the  cross-dissimilarity measure.
 

As discussed in Section \ref{approaches}, we will encode the information in $M_{Q,P}$ as a witness complex and utilize the isomorphism in Dowker's theorem to identify $\bc_k(W_{Q,P}^\bullet)$ and $\bc_k(W_{P,Q}^\bullet).$ We then apply the persistent extension method to associate bars in $\bc_k(X_Q^\bullet)$ and $\bc_k(X_P^\bullet)$ to those in the witness barcodes. There are two perspectives that we can take. 
\begin{enumerate}
    \item First, a \emph{feature-centric} approach, where the goal is to study how a particular bar $\selectIntX \in \bc_k(X_Q^\bullet)$ is represented in $\bc_k(X_P^\bullet)$. To streamline this approach, we give two truncated variations of Algorithm \ref{alg:bar_to_bar_extension}. Algorithm \ref{alg:bar_to_cycle_extension}, the bar-to-cycle extension method, returns only the collection of cycle extensions, thereby eliminating the need to iterate over the interval decompositions of $P\HH_k(\filtTwo)$. Algorithm \ref{alg:cycle_to_bar_extension}, the cycle-to-bar extension method, takes as input an explicit cycle $[\tau] \in \HH_k(Z^{\psi})$, which eliminates the iteration over interval decompositions of $P\HH_k(\filtOne)$. 
    
    Given a bar $\tau \in \bc_k(X^\bullet_Q)$, we first apply Algorithm \ref{alg:bar_to_cycle_extension} to $\selectIntX$, producing a family $E(\selectIntX, W_{Q, P}^\bullet)$ of cycle extensions in $P\HH_k(W^\bullet_{Q,P})$. We find the corresponding collection of cycles $E(\tau, W^\bullet_{P,Q})$ in $P\HH_k(W^\bullet_{P,Q})$ by applying Dowker's Theorem to elements of $E(\selectIntX, W_{Q,P}^\bullet)$. Finally, to each cycle extension $[\sigma] \in E(\selectIntX, W^\bullet_{P,Q})$, we apply Algorithm \ref{alg:cycle_to_bar_extension} to find a family of bar extensions $S(\selectIntX, X^\bullet_P)  = \bigcup_{[\sigma] \in E(\selectIntX, W_{P, Q}^\bullet)} S( [\sigma], X_P^\bullet)$ in $\bc_k(X^\bullet_P)$. Any bar extension in $S([\sigma], X^\bullet_P)\subseteq S( \selectIntX, X^\bullet_P)$ is \emph{analogous to $\tau$ through $[\sigma].$} Algorithm \ref{alg:feature_centric_analogous_bars} provides a complete account of this approach.

    \item Second, a \emph{similarity-centric} approach, where the goal is to understand what relationship a particular bar $\selectIntX \in \bc_k(W^\bullet_{Q,P})$ indicates between bars in $\bc_k(X_Q^\bullet)$ and $\bc_k(X_P^\bullet).$ In this case, we apply Dowker's Theorem to find the corresponding bar $\tau' \in \bc_k(W^\bullet_{P,Q})$, and then we apply the persistent extension method to $\tau$ and $\tau'$ respectively. We end up with a collections of bar extensions $S(\selectIntX, X^\bullet_Q)$ of $\selectIntX$ to $\bc_k(X^\bullet_Q)$ and $S(\selectIntX', X_P^\bullet)$ of $\selectIntX'$ to  $\bc_k(X^\bullet_P)$. Any pair of extensions in $S(\selectIntX, X^\bullet_Q) \times S(\selectIntX', X^\bullet_P)$ can be viewed as \emph{analogous through} $\selectIntX.$ This approach is described in Algorithm \ref{alg:similarity_centric_analogous_bars}.
\end{enumerate}

\setcounter{topnumber}{3}

\begin{algorithm}[H]
\caption{Extension method (bar-to-cycle extension)}
\label{alg:bar_to_cycle_extension}
\textbf{Input}:
\begin{itemize}
   \item filtered simplicial complexes $\filtOne$ and $\filtTwo$ on vertex set $P$, 
   \item a bar $\selectIntX \in \bc_k(\filtOne),$
   \end{itemize}
\textbf{Output}:
\begin{itemize}
    \item 
a family $E(\tau, \filtTwo)$ of collections of cycle extensions of $\tau$ to $P\HH_k(Y^\bullet)$.
\end{itemize}
\textbf{Steps}:
\begin{enumerate}
\item Fix an interval decomposition $\IntDecZ$ of $P\HH_k(\filtOne)$.
\item Perform steps (2) to (5-a) of bar-to-bars extension (Algorithm \ref{alg:bar_to_bar_extension}).
\item Return $E(\tau, \filtTwo) = \bigcup_{[\tau] \in T} \cealgo_{[\tau]}$
\end{enumerate}
\end{algorithm}
\begin{algorithm}[b]
\caption{Extension method (cycle-to-bar extension)}
\label{alg:cycle_to_bar_extension}
\textbf{Input}:
\begin{itemize}
   \item filtered simplicial complexes $\filtOne$ and $\filtTwo$ on vertex set $P$,
   \item a parameter $\psi$, and
   \item a homology class $[\tau] \in \HH_k(Z^{\psi})$.
   \end{itemize}
\textbf{Output}: 
\begin{itemize}
    \item a collection $S([\tau], \filtTwo)$ of bar extensions of $[\tau]$ to $\bc_k(\filtTwo)$.
\end{itemize}
\textbf{Steps}:
\begin{enumerate}
\item Let $p_Y$ and $\cealgo$ be the outputs of cycle-to-cycles extension (Algorithm \ref{alg:cycle_to_cycle_extension}).
\item Fix interval decomposition $\IntDecY$ of $P\HH_k(\filtTwo)$
\item Perform step (5-b) of bar-to-bars extension (Algorithm \ref{alg:bar_to_bar_extension}) with $p_{Y, [\tau]} = p_Y$ and $\cealgo_{[\tau]} = \cealgo$.
\item Return $S([\tau], \filtTwo) = \{S^{\;\IntDecY \circ L^{-1}}_{[y]} \;|\; \ell \in p_Y, [y] \in \cealgo_{\ell}, L \in L_Y \}.$ 
\end{enumerate}
\end{algorithm}
Both methods use Dowker's Theorem to compare $W^\bullet_{Q,P}$ and $W^\bullet_{P,Q}$ and use the extension method to compare the Witness complexes and Vietoris-Rips complexes. Because the involved algorithms enumerate different collections of cycles, in general we expect that the relationship between these two approaches may be quite intricate. 
Nonetheless, we believe it is useful to study how they relate in a  simple example. In the following section, we provide an empirical comparison of the two perspectives on a slightly modified version of our earlier torus example which demonstrates that, in simple settings, the outputs of the two methods can coincide in a reasonable way.
\begin{algorithm}[t!]
\caption{Analogous bars method (feature-centric)}
\label{alg:feature_centric_analogous_bars}
\textbf{Input}:
\begin{itemize}
   \item filtered clique complexes $X_Q^\bullet$ and $X_P^\bullet$, 
   \item cross-dissimilarity witness complexes $W_{Q,P}^\bullet$ and $W_{P,Q}^\bullet$ and
   \item a bar $\selectIntX \in \bc_k(X_Q^\bullet).$
   \end{itemize}
\textbf{Output}:
\begin{itemize}
    \item 
a family $S(\tau, X_P^\bullet)$ of collections of weighted bars  in $\bc_k(X^\bullet_P)$ related to $\tau$ 
\end{itemize}
\textbf{Steps}:
\begin{enumerate}
\item Apply bar-to-cycle extension (Algorithm \ref{alg:bar_to_cycle_extension}) with inputs $\filtOne = X^\bullet_Q$, $\filtTwo = W^\bullet_{Q,P}$, and $\tau$. Denote by $E(\tau, W^\bullet_{Q,P})$ the resulting collection of cycle extensions of $\tau$ in $P\HH_k(W^\bullet_{Q,P}).$
\item Apply Dowker's Theorem to each cycle in $E(\tau, W^\bullet_{Q,P})$ to find the collection
\[E(\tau, W^\bullet_{P,Q}) = \{ [x_P] \in W^\epsilon_{P, Q} \; | \; [x_P] \text{ is Dowker dual to } [x_Q] \in W^\epsilon_{Q,P}, \text{ where } [x_Q] \in E(\tau, W^\bullet_{Q,P}) \}. \]
\item For each $[\sigma] \in E(\tau, W^\bullet_{P, Q})$, 
\begin{enumerate}
    \item Find the death time $\delta([\sigma])$.
    \item Apply cycle-to-bar extension (Algorithm \ref{alg:cycle_to_bar_extension}) with $Z^\bullet = W^\bullet_{P,Q}$, $Y^\bullet = X^\bullet_P$, $\psi = \delta([\sigma]) - 1$, and $[\sigma] \in \HH_k(W^{\psi}_{P,Q})$ as input to find a family of bar extensions $S([\sigma], X^\bullet_P)$ in $\bc_k(X^\bullet_P)$.
\end{enumerate} 
\item Return $S(\tau, X^\bullet_P) = \bigcup_{[\sigma] \in E(\tau, W^\bullet_{P, Q})} S([\sigma], X^\bullet_P)$.
\end{enumerate}
\end{algorithm}

\begin{algorithm}[t!]
\caption{Analogous bars method (similarity-centric)}
\label{alg:similarity_centric_analogous_bars}
\textbf{Input}:
\begin{itemize}
   \item filtered clique complexes $X_Q^\bullet$ and $X_P^\bullet$, 
   \item cross-dissimilarity witness complexes $W_{Q,P}^\bullet$ and $W_{P,Q}^\bullet$, and
   \item a bar $\selectIntX \in \bc_k(W_{Q,P}^\bullet).$
   \end{itemize}
\textbf{Output}:
\begin{itemize}
    \item 
collections $S(\tau, X_Q^\bullet)$ and $S(\tau', X_P^\bullet)$ of weighted bars in $\bc_k(X^\bullet_Q)$ and $\bc_k(X^\bullet_P)$ respectively, for which pairs $(\sigma_Q, \sigma_P) \in S(\tau, X_Q^\bullet) \times S(\tau', X_P^\bullet)$ are related through $\tau$ 
\end{itemize}
\textbf{Steps}:

\begin{enumerate}
\item Apply Dowker's Theorem to $\tau$ to find the corresponding bar $\tau' \in \bc_k(W_{P,Q}^\bullet).$
\item Apply bar-to-bars extension (Algorithm \ref{alg:bar_to_bar_extension}) with $\filtOne = W^\bullet_{Q,P}$, $\filtTwo = X^\bullet_Q$, and $\tau$ as input and output a family $S(\tau, X^\bullet_Q)$ of bar extensions of $\tau$ in $\bc_k(X^\bullet_Q).$
\item Apply bar-to-bars extension (Algorithm \ref{alg:bar_to_bar_extension}) with $\filtOne = W^\bullet_{P,Q}$, $\filtTwo = X^\bullet_P$, and $\tau'$ as input and output a family $S(\tau', X^\bullet_P)$ of bar extensions of $\tau'$ in $\bc_k(X^\bullet_P).$
\item Return $S(\tau, X^\bullet_Q)$ and $S(\tau', X^\bullet_P).$
\end{enumerate}
\end{algorithm}
\begin{example}

Take $P$ to be the collection of points sampled uniformly from a flat torus considered in Examples \ref{ex:cycle-to-cycles} and \ref{ex:bars-to-bars}, and let $Q$ be a set of points sampled along a deformation of one of its essential circles, as illustrated in Figure \ref{fig:feature_centric_example}. We have chosen $Q$ to take a substantial detour from the usual ``taut" essential circle so that both the Vietoris-Rips complex $X_Q^\bullet$ and the witness complex $W^\bullet_{Q,P}$ will, for some range of parameters, contain homology classes representing both of the essential circles in the torus, but that the primary feature will be the vertical circle. The left panel in Figure \ref{fig:feature_centric_example} shows the three relevant barcodes $\bc_1(X^\bullet_Q)$, $\bc_1(W^\bullet_{Q,P}) = \bc_1(W^\bullet_{P,Q})$, and $\bc_1(X^\bullet_Q)$. Note that all bars have unique death times.\\

\setcounter{topnumber}{1}

\begin{figure}[t]
\centering
\includegraphics[width=0.9\textwidth]{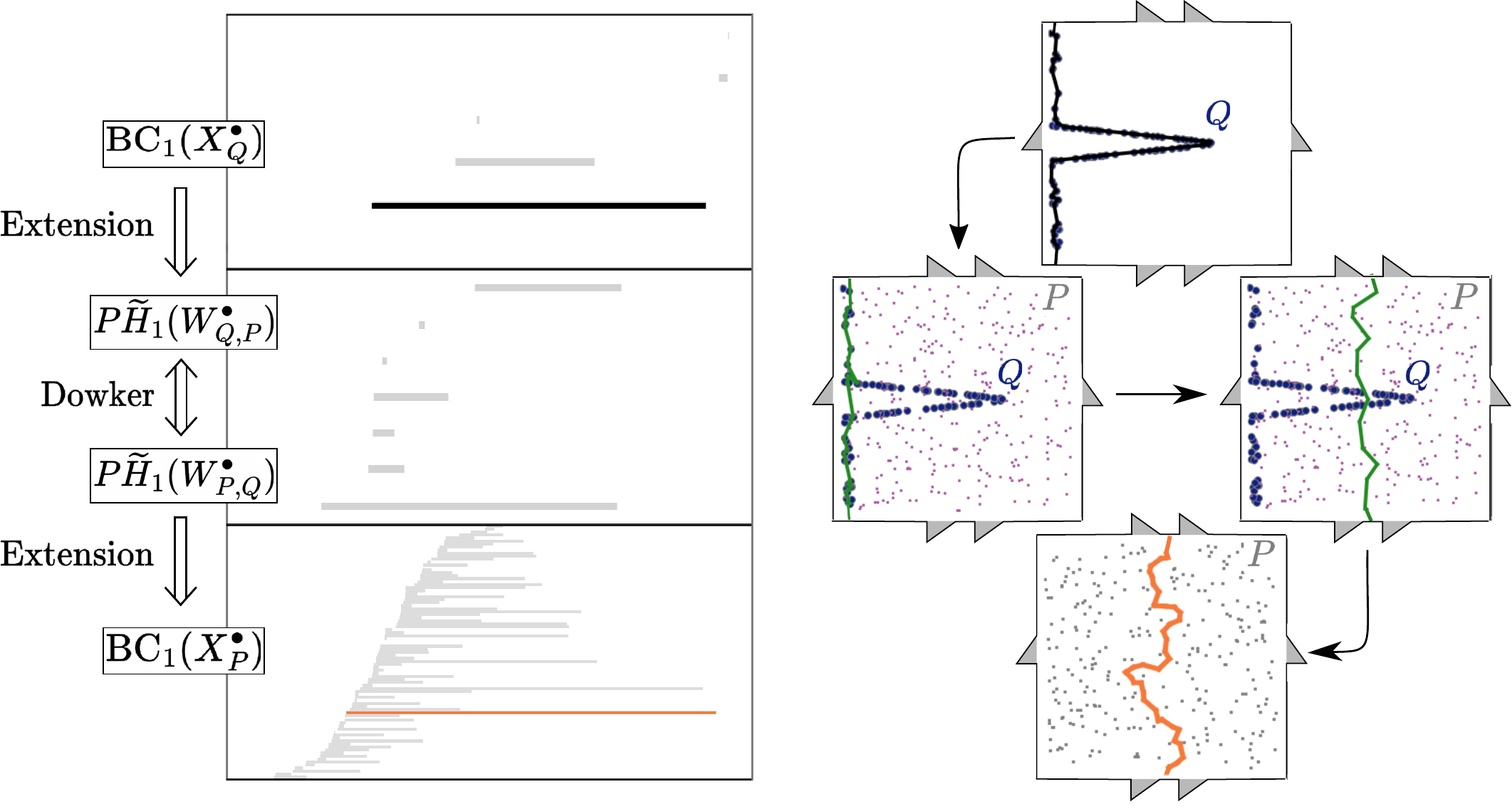}
\caption{\textbf{Illustration of the feature-centric analogous bars method.} Let $P$ be a point cloud sampled from a torus from Example \ref{ex:cycle-to-cycles}, and let $Q$ be a deformation of an essential circle that witnesses both essential circles of the torus. (left column) Beginning with the black bar in $\bc_1(X^\bullet_Q)$, apply persistent extension to obtain cycle extensions in $W^\bullet_{Q,P}$. Using the isomorphism in Dowker's theorem, find the corresponding cycle extension in $W^\bullet_{P,Q}.$ Apply persistent extension again to obtain the orange bar in $\bc_1(X^\bullet_P)$. The highlighted bars in $\bc_1(X^\bullet_Q)$ and $\bc_1(X^\bullet_P)$ are analogous bars. (right column) Cycle representatives illustrating the classes obtained in this example. (top) Chain in $X^\bullet_Q$ representing the long bar in $\bc_1(X^\bullet_Q).$ (middle-left) Cycle representative in $W_{Q,P}^\bullet$ obtained by persistent extension of the bar in the top row is transformed into (middle-right) a cycle in $W_{P,Q}^\bullet,$ supported on $P,$ via Dowker's theorem. (bottom) Applying persistent extension to this cycle, we obtain a corresponding cycle in $X^\bullet_P.$ }
\label{fig:feature_centric_example}
\end{figure}

\noindent \textit{Feature-centric analogous bars:} We will apply the feature-centric analogous bars method (Algorithm \ref{alg:feature_centric_analogous_bars}) to investigate the representation of the longest bar $\tau \in \bc_1(X^\bullet_Q)$ in $\bc_1(X^\bullet_P)$. 
Step (1) is to apply the bar-to-cycle extension (Algorithm \ref{alg:bar_to_cycle_extension}) to find all cycle extensions $E(\tau, W^\bullet_{Q,P})$. Since we are working with $\field_2$ coefficients, the terminal class $[\tau_*] \in \HH_1(X^{\delta(\tau)-1})$ is independent of the interval decomposition of $P\HH_1(X^\bullet_Q)$ by Lemma \ref{lemma:well_defined_cycles}. Algorithm \ref{alg:bar_to_cycle_extension} simplifies since $T = \{ [\tau_*] \}$. The left-middle panel in the right column in Figure \ref{fig:feature_centric_example} illustrates an example cycle in the cycle extension $[x_Q] \in E(\tau, W^\bullet_{Q,P})$. 

In step (2) of Algorithm \ref{alg:feature_centric_analogous_bars}, we find the family of cycles that are dual to $E(\tau, W^\bullet_{Q,P})$ via Dowker's Theorem. The right-middle panel in the right column of Figure \ref{fig:feature_centric_example} illustrates the dual $[x_P] \in E(\tau, W^\bullet_{P,Q})$ to the previously illustrated $[x_Q]$.

In step (3) of Algorithm \ref{alg:feature_centric_analogous_bars}, for each cycle extension $[\sigma] \in E(\tau, W^\bullet_{P, Q})$, we apply the cycle-to-bar extension method (Algorithm \ref{alg:cycle_to_bar_extension}) to find the bar extensions $S([\sigma], X^{\bullet}_P)$ in $\bc_1(X^\bullet_P)$. The baseline bar extension is illustrated on $\bc_1(X^\bullet_P)$ in the left column of Figure \ref{fig:feature_centric_example}. The highlighted bars in $\bc_1(X^\bullet_Q)$ and $\bc_1(X^\bullet_P)$ are analogous through $[\sigma]$. The bottom panel in the right column of Figure \ref{fig:feature_centric_example} illustrates the corresponding cycle representative of the baseline bar extension.

\medskip

\noindent \textit{Similarity-centric analogous bars:} We now illustrate the similarity-centric analogous bars method on the same dataset. The presence of a long bar in $\bc_1(W^\bullet_{P,Q})$ in Figure \ref{fig:similarity_centric_example} suggests that there is a significant feature present in both $P$ and $Q$. Thus, we select the long interval $\tau \in \bc_1(W^\bullet_{Q,P})$ to which to apply the similarity-centric analogous bars method (Algorithm \ref{alg:similarity_centric_analogous_bars}).

We apply Dowker's Theorem to identify the corresponding bar $\tau' \in \bc_1(W^\bullet_{P,Q})$, though morally these are identical. Computationally, we need to change representations to proceed. We then apply the bar-to-bar extension method (Algorithm \ref{alg:bar_to_bar_extension}) to find bar extensions of $\tau$ in $\bc_1(X^\bullet_Q)$ and $\tau'$ in $\bc_1(X^\bullet_P)$. The left column of Figure \ref{fig:similarity_centric_example} illustrates the selected bar $\tau \in \bc_1(W^\bullet_{Q,P})$ and the baseline bar extensions in $\bc_1(X^\bullet_Q)$ and $\bc_1(X^\bullet_P)$ respectively. The highlighted bars of $\bc_1(X^\bullet_Q)$ and $\bc_1(X^\bullet_P)$ are analogous through $\tau$. The cycle representatives show that these analogous bars, indeed, trace qualitatively similar cycles. Comparing Figures \ref{fig:feature_centric_example} and \ref{fig:similarity_centric_example}, we see that at the level of barcodes on simple examples, the feature- and similarity-centric methods produce comparable and reasonable results.

\begin{figure}[t]
\centering
\includegraphics[width=0.9\textwidth]{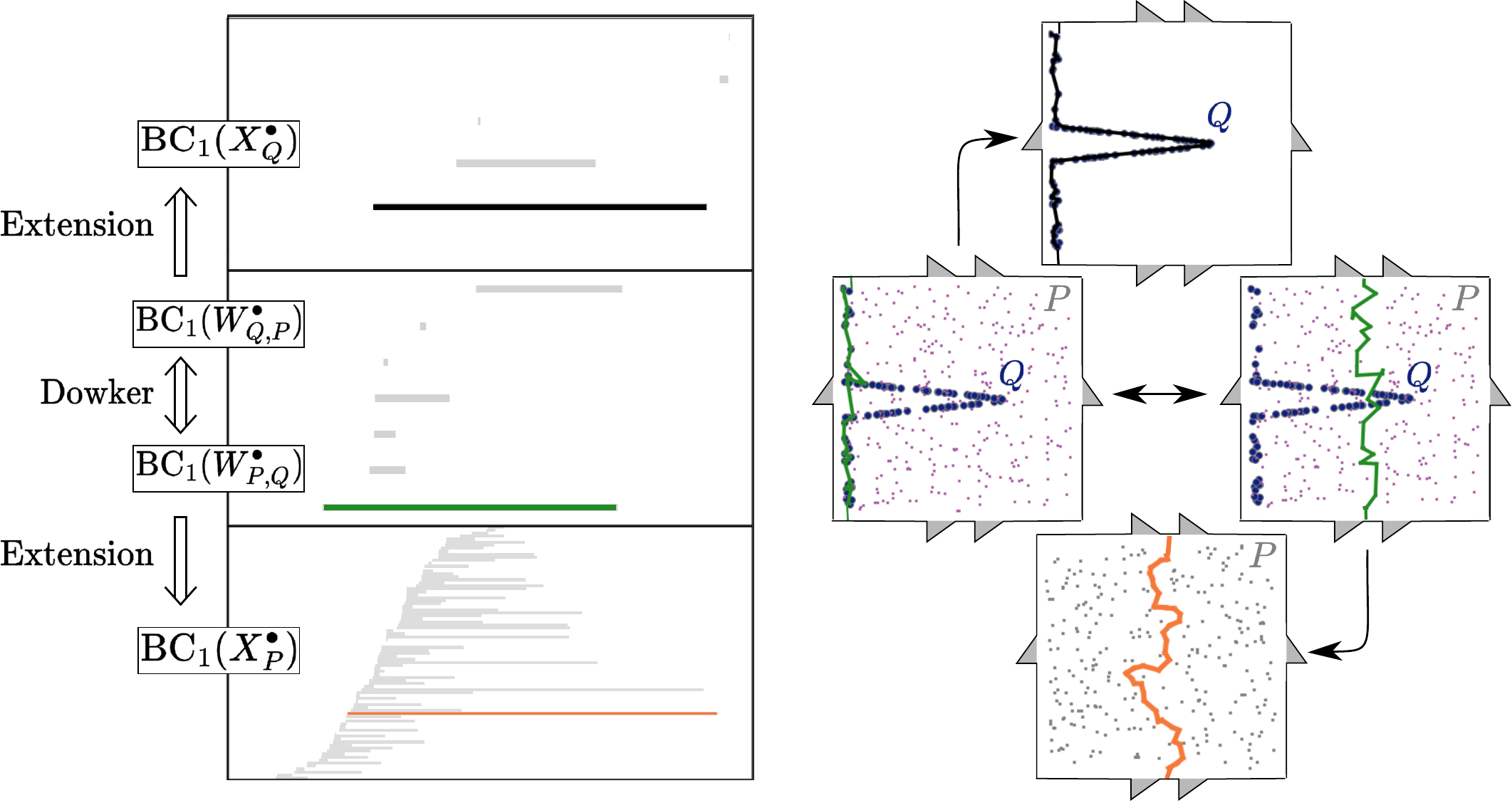}
\caption{\textbf{Illustration of the similarity-centric analogous bars method.} Let $P$ be a point cloud sampled from a torus from in Examples \ref{ex:cycle-to-cycles}, and let $Q$ be a deformation of an essential circle that witnesses both essential circles of the torus. (left column) Beginning with the green bar in $\bc_1(W^\bullet_{P,Q}) = \bc_1(W^\bullet_{Q,P}),$ we apply persistent extension twice to obtain baseline black bar extension in $\bc_1(X^\bullet_{Q})$ and orange bar extension in $\bc_1(X^\bullet_P)$. The highlighted bars in $\bc_1(X^\bullet_Q)$ and $\bc_1(X^\bullet_P)$ are analogous bars. (right column) Cycle representatives illustrating the classes obtained in this example. (middle) Cycle in $W_{Q,P}^\bullet$ representing the selected bar in the left column and a corresponding representative in $W_{P,Q}^\bullet$ obtained via Dowker's theorem. (top, bottom) Cycles in $X^\bullet_Q$ and $X^\bullet_P$ respectively representing the corresponding bars in $\bc_1(X^\bullet_Q)$ and $\bc_1(X^\bullet_P)$.}
\label{fig:similarity_centric_example}
\end{figure}
\end{example}

\section{Applications}
\label{applications} 

Finally, we demonstrate the use of the methods we have developed on some simple data sets: we apply the analogous bars method to identify corresponding features from two samples of the same space and to determine whether topological features are retained under clustering; then, we use persistent extension on its own to investigate how topological features are transformed by dimensionality reduction.

\subsection{Two samples from the same distribution}

\begin{figure}[t]
\centering
\tabskip=0pt
\includegraphics[width=0.9\textwidth]{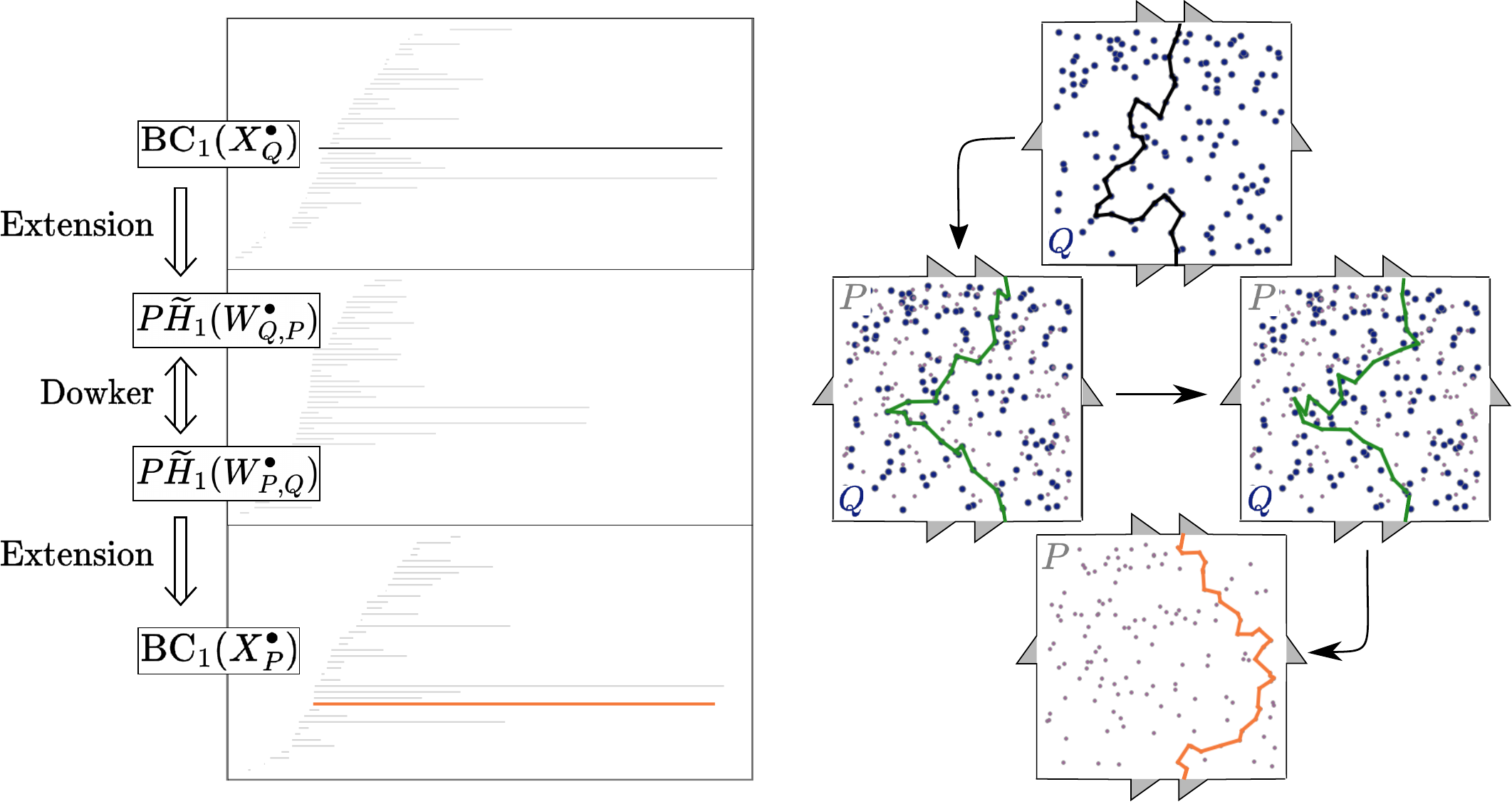}
\caption{\textbf{Using the feature-centric analogous bars method to identify features from independent samples of a torus.} 
Grey points $P$ and blue points $Q$ are independently sampled from the uniform distribution on the torus. (left column) Dimension-one barcodes for $X^\bullet_Q$, $W^\bullet_{Q,P}$ and $X^\bullet_P.$ The black bar in $\tau \in \bc_1(X^\bullet_Q)$ is the input to the method, and the orange bar in $\bc_1(X^\bullet_P)$  is the baseline bar extension. The highlighted bars in $\bc_1(X^\bullet_Q)$ and $\bc_1(X^\bullet_P)$ are analogous bars. No bars in $\bc_1(W^\bullet_{P,Q})$ are highlighted as Algorithm \ref{alg:feature_centric_analogous_bars} bypasses translation through this barcode. (right column) Cycle representatives at each stage of the feature-centric analogous bars method. (top) Cycle representative   of $\tau$ supported on $Q$. (middle-left) Representatives of a cycle $[x_Q] \in E(\tau, W^\bullet_{Q,P})$ produced by step (1) of Algorithm \ref{alg:feature_centric_analogous_bars} and (middle-right)
its Dowker dual $[x_P] \in E(\tau, W^\bullet_{P,Q}).$ (bottom) Representative of $\tau' \in S(\tau, X^\bullet_P)$ produced by step (3) of Algorithm \ref{alg:feature_centric_analogous_bars} applied to the cycle in (middle-right). }
\label{fig:torus_two_samples}
\end{figure}

One of the most common tasks in topological data analysis is the recovery of topological features from a finite sample of a probability distribution on a metric space. If we have multiple samples, for example, from different measurements of a system, understanding how the features in those samples are related would provide utility in understanding what these features mean in terms of the system being studied. This is one of the simplest settings in which we can apply the analogous bars method, using the block sub-matrix of the distance matrix that measures distance between the two samples as a measure of cross-dissimilarity.

In Figure \ref{fig:torus_two_samples}, we illustrate this approach using two point clouds, $P$ and $Q$, sampled uniformly and independently from a square torus. We consider a choice of feature $\tau \in \bc_1(X^\bullet_Q)$, and locate the corresponding feature in  $\bc_1(X^\bullet_P)$ using the feature-centric analogous bars method with $\tau$ as input. The highlighted bars of $\bc_1(X^\bullet_Q)$ and $\bc_1(X^\bullet_P)$ are analogous bars that represent the same essential circle.


\subsection{Clustering}
\label{sec:clustering}

The analogous bars method can similarly be applied to determine if topological features of a point cloud $Q$ are preserved by clustering. In Figure \ref{fig:clustering_results}, we demonstrate such a comparison using the distance matrix between points $Q$ and cluster centroids $P$ as our cross-dissimilarity matrix. 
Because the witness barcode $\bc_1(W^\bullet_{P,Q})$ has a long bar, we expect that the two point clouds share a feature. Applying the similarity-centric analogous bars method, we find the analogous bars in $\bc_1(X^\bullet_Q)$ and $\bc_1(X^\bullet_P)$ and conclude that the major topological feature of $Q$ is, in fact, retained by $P$.

\begin{figure}[ht!]
\centering
\includegraphics[width=0.9\textwidth]{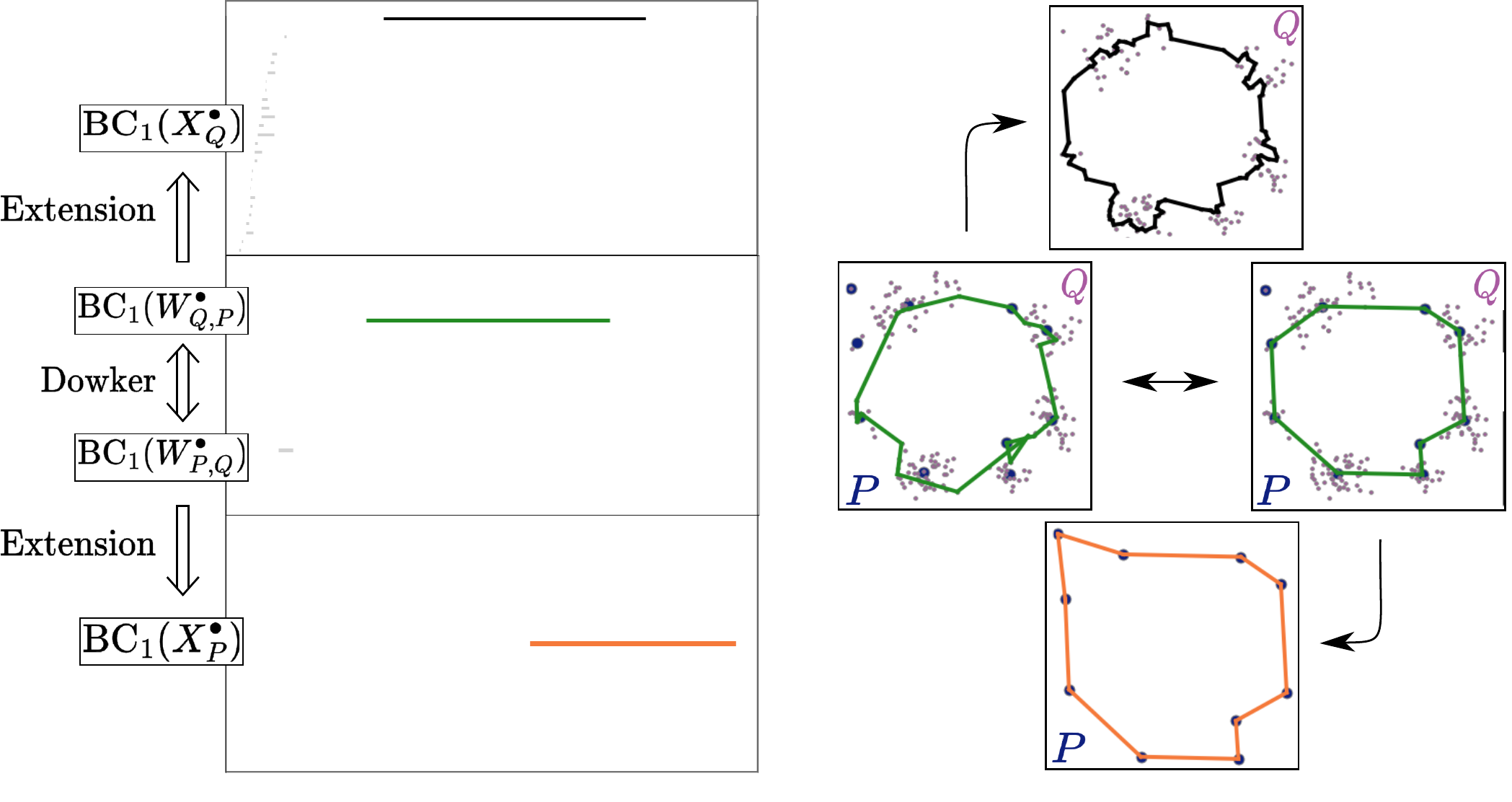}
\caption{\textbf{Application of similarity-centric analogous bars to detect topological features preserved by clustering.} 
Grey points $Q$ sampled from a multimodal distribution in the plane, and blue centroids $P$ induced by clustering. (left column) Barcodes for $X^\bullet_Q$, $W^\bullet_{Q,P}$ and $X^\bullet_P.$ The green bar $\tau \in \bc_1(W^\bullet_{Q,P})$ is the input to the method, and the black bar in $\bc_1(X^\bullet_Q)$ and orange bar in $\bc_1(X^\bullet_P)$ are the baseline bar extensions. These highlighted bars in $\bc_1(X^\bullet_Q)$ and $\bc_1(X^\bullet_P)$ are analogous bars. (right column) Cycle representatives for the input, its Dowker dual, and the outputs. (middle-left) Representative of the cycle in $W^\bullet_{Q,P}$ corresponding to the selected bar and (middle-right) its Dowker dual in $W^\bullet_{P,Q}$ . (top) Cycle representatives in the original data $Q$ and (bottom) in the centroids of the clusters $P$ confirm that this topological feature of $Q$ is preserved by clustering. }
\label{fig:clustering_results}
\end{figure}

\subsection{Dimensionality reduction}
\label{sec:dim_red}
Most dimensionality reduction techniques retain the identities of the projected points. Thus, we can apply the persistent extension method to the Vietoris-Rips complexes of the original point cloud and the projected point cloud to study whether, and how, the resulting lower-dimensional point cloud retains topological features of interest. Here, we provide two toy examples using the extension method in the context of dimensionality reduction via principal component analysis. 

\subsubsection{Trefoil knot}
First, we consider a collection of points $Q \subseteq \mathbb{R}^3$ sampled from the trefoil knot depicted in Figure \ref{fig:trefoil_knot_example} and take $\pi(Q)\subseteq\mathbb{R}^2$ to be its 2-dimensional projection via PCA. We apply the persistent extension method to these two point clouds to identify how the dominant 1-dimensional homological feature of the trefoil knot is represented in its two dimensional projection. 

The 1-dimensional barcodes $\bc_1(X^\bullet_Q)$ and $\bc_1(X^\bullet_{\pi(Q)})$ are depicted in the middle panes of Figure \ref{fig:trefoil_knot_example}. We apply the persistent extension method to the long bar in $\bc_1(X^\bullet_Q)$. The vertical line in the bottom pane indicates the $\ell_0$ parameter, and the highlighted bars in $\bc_1(X^\bullet_{\pi(Q)})$ correspond to one of the bar extensions; in fact, the top bar of $\bc_1(X^\bullet_{\pi(Q)})$ corresponds to the baseline bar extension, and the three bars whose interiors intersect the $\ell_0$-line independently make up the offset bar extensions and so can be included in any combination. The pictured extension indicates that the $S^1$ described by the 3-dimensional trefoil knot can be represented by a combination of the four features detected in $\bc_1(X^\bullet_{\pi(Q)})$. Examining the cycle representatives of the highlighted bars of $\bc_1(X^\bullet_{\pi(Q)})$ depicted on the right of Figure \ref{fig:trefoil_knot_example} illustrates why this is a reasonable answer: the various combinations enclose regions ``internal" to the knot projection.

\setcounter{topnumber}{2}
\begin{figure}[t]
\centering
\includegraphics[width=0.9\textwidth]{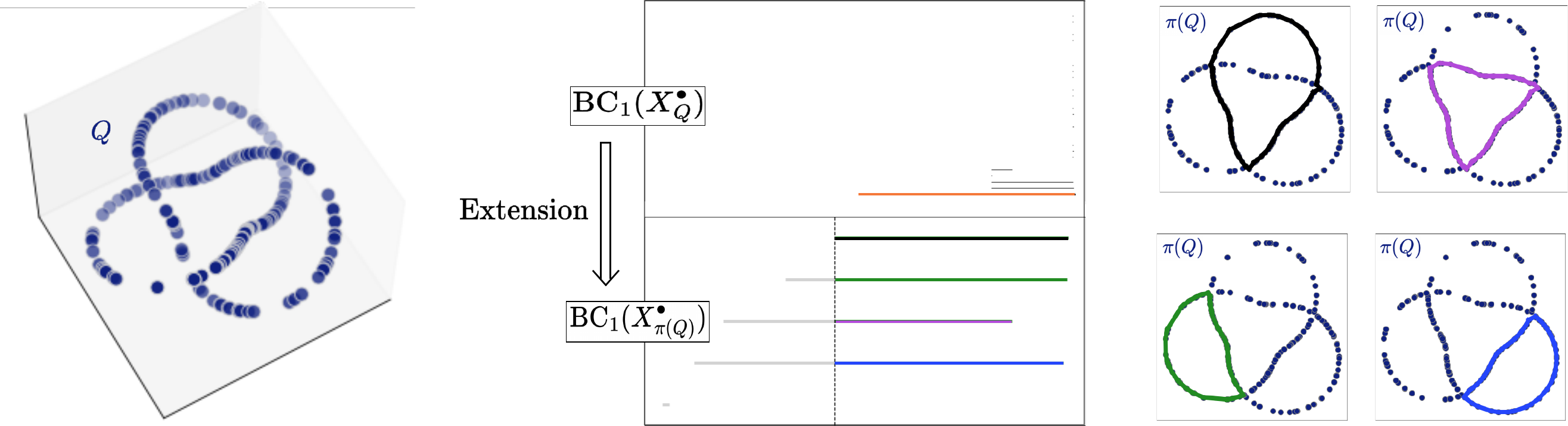}
\caption{\textbf{Application of persistent extension method to planar projection of the trefoil knot via PCA.} (left) Blue points  $Q$ sampled from a standard trefoil knot in three dimensions. (middle) One-dimensional barcodes for the sample $Q$ and its projection $\pi(Q)$ to the first two loading vectors via PCA. Selecting the long orange bar $\tau \in \bc_1(X^\bullet_Q)$ and applying bar-to-bar persistent extension (Algorithm \ref{alg:bar_to_bar_extension_F2}), we find that the top black bar in $\bc_1(X^\bullet_{\pi(Q)})$ is our baseline bar extension and the remainder are individually offsets. (right) Cycle representatives illustrating the four bars appearing in the set of possible extensions.}
\label{fig:trefoil_knot_example}
\end{figure}


\subsubsection{Cylinder}
\begin{figure}[b]
\centering
\tabskip=0pt
\includegraphics[width=0.9\textwidth]{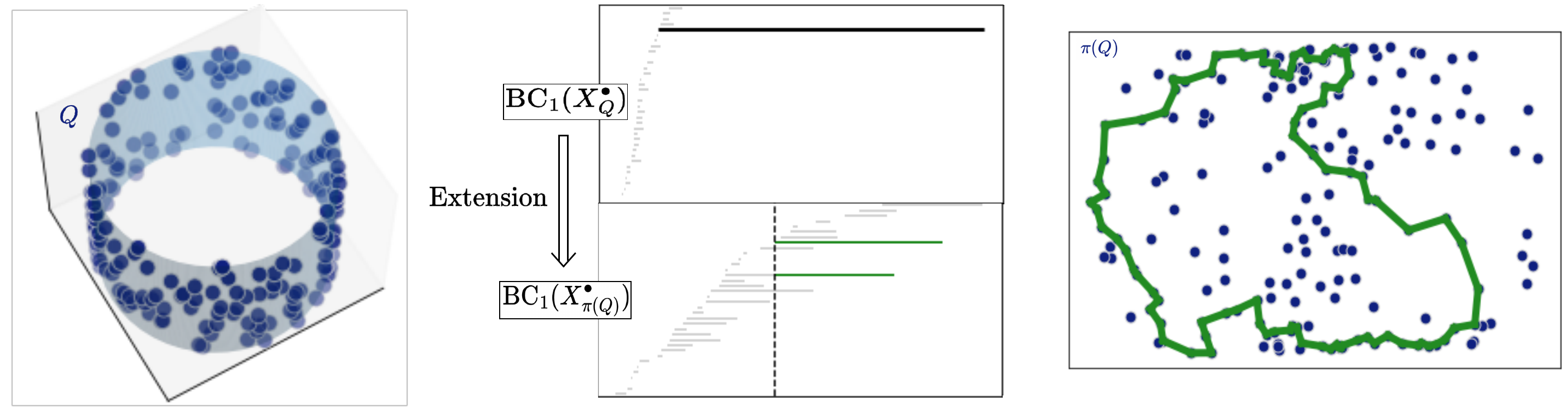}
\caption{\textbf{Application of persistent extension method to planar projection of a sampled cylinder via PCA.} (left) Blue points $Q$ sampled from a cylinder in three dimensions. (middle) One-dimensional barcodes for the sample $Q$ and its projection $\pi(Q)$ to the first two loading vectors via PCA. Selecting the long bar $\tau \in \bc_1(X^\bullet_Q)$ and applying bar-to-bar persistent extension (Algorithm \ref{alg:bar_to_bar_extension_F2}), we find that the green bars in $\bc_1(X^\bullet_{\pi(Q)})$ form our baseline bar extension. (right) Cycle representative for the baseline cycle extension, illustrating how the cycle in the cylinder is flattened. }
\label{fig:PCA_example}
\end{figure}

Sometimes, applying dimensionality reduction to a point cloud flattens the object too much to easily visualize what happens to the topological features. Let $Q$ be a set of points sampled from a cylinder in $\mathbb{R}^3$, as shown in the left panel of Figure \ref{fig:PCA_example}, and let $\pi(Q)$ be the result of projecting the points onto their first two loading vectors. As before, we consider $X^\bullet_Q$ and $X^\bullet_{\pi(Q)},$ with 1-dimensional barcodes as shown in the middle pane of Figure \ref{fig:PCA_example}. Selecting the highlighted long bar in $\bc_1(X^\bullet_Q)$ and applying the persistent extension method produces the highlighted bars in $\bc_1(X^\bullet_{\pi(Q)})$ as our baseline bar extensions. The right panel in Figure \ref{fig:PCA_example} shows the cycle representative of the baseline cycle extension corresponding to the baseline bar extension in $\bc_1(X^\bullet_{\pi(Q)})$. Other choices of extension and representatives will recover how paths ``around" the cylinder project onto this plane.

\section{Conclusion}
\label{conclusion} 
In this paper, we addressed the problem of comparing  topological features computed from a pair of finite systems equipped with internal and cross-dissimilarity measures. We focused on the case where a triple of matrices provides a summary of the observations about such systems, a setting that commonly arises in applications. To leverage this information we developed two methods: 
\begin{enumerate}
    \item the \emph{persistent extension} method, a zig-zag computation which provides a means for testing potential relationships between two complexes built on the same vertex set; and, 
    \item the \emph{analogous bars} method, which leverages the Dowker complex of a cross-dissimilarity matrix to bridge between complexes supported on two distinct vertex sets.
\end{enumerate} 
Our examples in this paper focused on point cloud data, a common setting for topological data analysis. However, there are a variety of settings where these tools are likely to be of use. In particular, there is no explicit requirement that we use clique complexes; we focused on this example because it is the most commonly studied and easily implemented with existing software packages.

For some time there has been discussion of using zig-zag based analyses for time-varying simplicial complexes. The persistent extension method provides a concrete approach for studying how features evolve within a single system. It can also be applied to study how different filtrations affect the topology of a complex. Indeed, both this proposd time series analysis and the dimensionality reduction examples in Section \ref{sec:dim_red} can be cast in this light. We can also envision using this approach to investigate two-dimensional persistence modules, studying how features in given one-dimensional submodules manifest in others. 

The analogous bars method provides a tool for studying simultaneously observed systems and determining whether their structure and dynamics capture common features. It is worth noting that while many potential applications will involve an underlying function, the Dowker complex was originally conceived as a topological method for studying relations \cite{Dowker}, and so our method does not require such an assumption to provide sensible answers. For example, it could be used to study a pair of systems driven by a common topologically interesting input signal and detect signatures of that signal which are preserved in each. In forthcoming work, the authors apply these methods for precisely this kind of analysis.

\medskip

While we have laid out the theoretical and computational foundation for using these tools, we have not attempted to address a large number of immediate questions that arose while we were writing this paper. We close with a selection of these that we find compelling.
\begin{enumerate}
    \item Can we systematically choose the parameter $\psi$ used in Algorithm \ref{alg:cycle_to_cycle_extension} using features of the data? That is, given a homology class $[\tau]$ and its death time $\delta([\tau])$, what is the optimal parameter $\psi$?    
    Empirically, we have found that when the homology class $[\tau]$ is represented by multiple bars, for example, as $S^{\mathcal{B}}_{[\tau]} = \{ \tau_1, \dots, \tau_n\}$, then the cycle extension that one intuitively expects may arise only when we consider a parameter $\psi$ that is smaller than $\delta([\tau])$. This appears to be the result of homology classes becoming too large as the parameter increases, so that the intuitive cycle extension becomes homologous to unexpected cycles. 
    \item How do we characterize the quality of the output of these methods? For example, suppose we consider Vietoris-Rips complexes for samples from some distribution on underlying metric spaces $X$ and $Y$, where our cross-dissimilarity measure is given by $d(f(x),y) + \eta$ for a noise distribution $\eta$ and some underlying map $f :X  \to Y.$ Can we analytically or computationally give a measure of confidence that a pair of analogous bars is correct?
    \item If we consider a sequence of Vietoris-Rips complexes built from increasingly large samples from a distribution on a metric space, do the sequence of analogous bars obtained from these methods stabilize in some manner, as we might hope they would? 
    \item Under what conditions do the feature-centric and similarity-centric analogous bars methods produce comparable results? As the feature-centric method is more intuitive and the similarity-centric method can be much more computationally efficient, understanding how they compare would be useful.
    \item How stable are the analogous bars under deformation of the original complexes and the cross-dissimilarity matrix? We expect the standard stability theorems of persistent homology \cite{stability_PD} and of Dowker persistence diagrams \cite{functorial_dowker} will provide a first approximation to this answer, but can we say more?
    \item There is a form of inverse problem suggested by this work: can we use information about topological features to register objects, such as samples from the same distribution, when no cross-dissimilarity measure is available? A first step could be understanding how information about analogous bars would constrain the cross-dissimilarity measure, allowing Dowker duality to provide the desired relations.
    \item Recent work by Gregory Henselman-Petrusek and the second author \cite{ghrist2021saecular} develops a framework for  persistence over more general (non-field) coefficient systems. Can we leverage this to generalize the algorithms developed in the current paper to, for example, the setting of lattices? 
\end{enumerate}


\appendix
\renewcommand{\thesection}{
\Alph{section}}


\section{Proof of Lemma \ref{lemma:basis_change}}
\label{pf:lemma:basis_change}
\basischange*

\begin{proof}
We will show that if $\tau_r$ and $\tau_c$ do not satisfy Equation \ref{interval_relations}, then $L^\ell_{r,c} = 0$. We consider the two cases $\beta(\tau_r) > \beta(\tau_c)$ and $\delta(\tau_r) > \delta(\tau_c)$ separately. Given a fixed ordering of bars in $\bc_k(Z^\ell)$, let $\{ \vec{e}^{\; \ell}_{\tau_i} \;|\; i=1, \dots, |\bc_k(Z^\ell)|\}$ be the corresponding basis of $I^\ell = (\mathbb{I}_{\bc_k(\filtOne)})^\ell$. 

First, assume that $\beta(\tau_c) < \beta(\tau_r) \leq \ell$. Let $\psi = \beta(\tau_r)-1$. Consider the following commutative diagram, which is obtained from Diagram \ref{diag:barcode_iso} by composing maps from $I^{\psi}$ to $I^\ell$. Let $\iota, L^\ell$, and $L^{\psi}$ denote the matrices representing the maps. 
	\[ 
	\begin{tikzcd}
	I^{\psi} \arrow[d, "L^{\psi}"] \arrow[r, "\iota"] & I^\ell \arrow[d,"L^\ell"] \\
	I^{\psi} \arrow[r, "\iota"] & I^\ell 
	\end{tikzcd}
	\]
Recall that $\vec{e}^{\; \ell}_r$ and $\vec{e}^{\; \ell}_c$ are the basis vectors of $I^\ell$ corresponding to $\tau_r$ and $\tau_c$. Let $\vec{e}^{\; \psi}_{c_*} \in I^{\psi}$ be the basis vector of $I^{\psi}$ corresponding to $\tau_c$. Let's consider $(L^\ell \cdot \iota)_{r, c_*} = (\iota \cdot L^{\psi})_{r, c_*}$. On one hand, since $\iota$ maps $\vec{e}^{\; \psi}_{c_*}$ to $\vec{e}^{\; \ell}_c$, the $c_*^{\text{th}}$ column of matrix $\iota$ has $1$ at component $c$ and $0$ elsewhere. Then, $(L^\ell \cdot \iota)_{r,c_*} = (L^\ell)_{r,c}$. On the other hand, since the bar $\tau_r$ does not exist at parameter $\psi$, the $r^{\text{th}}$ row of matrix $\iota$ is the zero vector, and $(\iota \cdot L^{\psi})_{r, c_*} = 0$. Thus, $L^\ell_{r,c} = (L^\ell \cdot \iota)_{r,c_*} = (\iota \cdot L^{\psi})_{r, c_*} = 0$.

We now assume that $\ell < \delta(\tau_c) < \delta(\tau_r)$. Consider the following commutative diagram obtained from Diagram \ref{diag:barcode_iso}. Again, let $\iota, L^\ell$, and $L^{\delta(\tau_c)}$ denote matrices.
\[ 
\begin{tikzcd}
I^{\ell} \arrow[d, "L^{\ell}"] \arrow[r, "\iota"] & I^{\delta(\tau_c)} \arrow[d,"L^{\delta(\tau_c)}"] \\
I^{\ell} \arrow[r, "\iota"] & I^{\delta(\tau_c)} 
\end{tikzcd}
\]
Let $\vec{e}^{\; \delta(\tau_c)}_{r_*}$ be the basis vector of $I^{\delta(\tau_c)}$ corresponding to $\tau_r$. Proceeding similarly as above, one can show that $L^\ell_{r,c} = (\iota \cdot L^\ell)_{r_*, c} = (L^{\delta(\tau_c)} \cdot \iota)_{r_*,c} = 0$.

\end{proof}

\section{Proof of Lemma \ref{lemma:diagonal}}
\label{pf:lemma:diagonal}
\diagonal*
\begin{proof}
Fix $\ell \in \mathbb{Z}$ and $1 \leq j \leq |\bc_k(Z^\ell)|$. Order the bars of $\bc_k(Z^\ell)$ that are alive at $\ell$ by decreasing birth parameters. If multiple bars share the same birth parameter, then order them by decreasing death parameters. Let $\tau_1, \dots, \tau_{|\bc_k(Z^\ell)|}$ be the resulting order. Without loss of generality, assume that each bar corresponds to the the basis vectors $\vec{e}^{\; \ell}_1, \dots, \vec{e}^{\; \ell}_{|\bc_k(Z^\ell)|}$ of $I^\ell = (\mathbb{I}_{\bc_k(Z^\bullet)})^\ell$. For every $r < c$, either $\beta(\tau_r) > \beta(\tau_c)$ or $\delta(\tau_r) > \delta(\tau_c)$. Thus, by Lemma \ref{lemma:basis_change},  $L^{\ell}_{r, c} =0$ for every $r < c$.

Proceed by induction. From Lemma \ref{lemma:basis_change}, we know that $L^{\ell}_{1, c} = 0$ for all $c > 1$. Since $L^{\ell}$ is a linear isomorphism, it cannot have a row of zeros. Thus, $L^{\ell}_{1,1} \neq 0$.

For the inductive step, assume that $L^{\ell}_{j-1, j-1} \neq 0$. We will show that $L^{\ell}_{j,j} \neq 0$ by assuming the contrary. Assume that $L^{\ell}_{j,j} = 0$. Let $\vec{r}_1, \dots \vec{r}_j$ each denote the first $j$ row vectors of $L^{\ell}$. Let $\vec{r}^{\; *}_i$ denote the vector in $\field^{j-1}$ that consists of the first $j-1$ components of $\vec{r}_i$ for $i=1, \dots, j$. Since $\vec{r}^{\; *}_1, \dots, \vec{r}^{\;*}_j$ are $j$ vectors in $\field^{j-1}$, the collection is linearly dependent. By construction, each $\vec{r}_i$ can be obtained from $\vec{r}^{\; *}_i$ by appending zeros. Thus, the collection $\vec{r}_1, \dots, \vec{r}_j$ is also linearly dependent. This contradicts the fact that $L^{\ell}$ is a linear isomorphism. Thus, it must be the case that $L^\ell_{j,j} \neq 0$.
\end{proof}

\section{Proof of Lemma \ref{lemma:linear_independence}}

\linindep*
\label{pf:lemma:linear_independence}

\begin{proof}
Without loss of generality, assume that the bars of $\bc_k(\SCOne^{\ell})$ are ordered by decreasing birth parameter, and if multiple bars share the same birth parameters, assume that the bars are ordered by decreasing death parameters. So if $r<c$, then $\beta(\tau_r) > \beta(\tau_c)$ or $\delta(\tau_r) > \delta(\tau_c)$. Fix $1 \leq t \leq |\bc_k(Z^\ell)|$. Let $L$ be any matrix satisfying Equation \ref{eq:matrixL}. By construction, $L_{r,c} = 0$ for all $r < c$. Let $\{ L_{*,1}, \dots, L_{*,t} \}$ denote the column vectors of $L$. Assume that 
\begin{equation}
\label{lin_ind_eq}
c_1 L_{*,1} + \dots + c_t L_{*,t} = \vec{0}
\end{equation}
Considering the first component of Equation \ref{lin_ind_eq}, we know that 
\[c_1 L_{1,1} + \dots + c_t L_{1,t} = 0\]
By construction, we know that $L_{1,1} \neq 0$, and $L_{1,2} = \dots = L_{1,t} = 0$. So $c_1 = 0$.

Proceed by induction. Assume $c_1 = \dots = c_{j-1} = 0$. The $j^{\text{th}}$ component of Equation \ref{lin_ind_eq} is
\[c_j L_{j,j} + \dots + c_t L_{j, t}= 0 \]
Again, by construction, we know that $L_{j,j} \neq 0$, and $L_{j, j+1} = \dots = L_{j, t} = 0$. So $c_j = 0$. By induction, $c_1 = \dots = c_t = 0$, and the column vectors of $L$ are linearly independent. 
\end{proof}

\section{Proof of Theorem \ref{main_thm} }

\label{proof_of_extension_algorithm}
To prove Theorem \ref{main_thm}, we need to show the following facts:
\begin{itemize}
    \item the collections $p_Y$ and $\restrictionalgo^{\IntDecDelta}_{\ell}$ are independent of the choice of interval decomposition $\IntDecDelta,$
    \item the collection $\restrictionalgo^{\IntDecDelta}_{\ell}$ coincides with $\restrictiondef_{\ell}$, the set of restrictions as defined in Definition \ref{restriction_def}, and
    \item the output of Algorithm \ref{alg:cycle_to_cycle_extension} suffices to find the set of cycle extensions  $E$ of $[\tau]$ as given in Definition \ref{restriction_def}.
\end{itemize} 
We prove each of these facts in turn.

\subsection{The collection of parameters $p_Y$ constructed in Algorithm \ref{alg:cycle_to_cycle_extension} is well-defined.}


\begin{restatable}{lemma}{pyindep}
\label{lemma:p_Y independence}
The collection $p_Y$ is independent of the choice of interval decomposition $\IntDecDelta.$
\end{restatable}

\begin{proof}
We first show that $\ell_0$ is well-defined. Let $\IntDecDeltaOne, \IntDecDeltaTwo: \mathbb{I}_{\bc_k(\intFilt)} \to P\HH_k(\intFilt)$ be interval decompositions consisting of the following diagram. Recall that $\HH_k(\SCOne^{\psi} \cap Y^N) = \HH_k(\SCOne^{\psi})$.
\begin{equation}
\label{PM_morphisms}
\begin{tikzcd}
I^{1} \arrow[r, rightarrow]   \arrow[d, shift right, swap, "\IntDecDeltaOne^{1}"]  \arrow[d, shift left, "\IntDecDeltaTwo^{1}"]  &  \cdots \arrow[r, rightarrow] & I^{\ell} \arrow[r, rightarrow] \arrow[d, shift right, swap, "\IntDecDeltaOne^{\ell}"]  \arrow[d, shift left, "\IntDecDeltaTwo^{\ell}"] &\cdots \arrow[r, rightarrow]  & I^{N} \arrow[d, shift right, swap, "\IntDecDeltaOne^{N}"]  \arrow[d, shift left, "\IntDecDeltaTwo^{N}"] \\
\HH_k(\SCOne^{\psi} \cap \SCTwo^{1})  \arrow[r, rightarrow]& \cdots \arrow[r, rightarrow]& \HH_k(\SCOne^{\psi} \cap \SCTwo^{\ell})  \arrow[r, rightarrow] & \cdots \arrow[r, rightarrow] & \HH_k(\SCOne^{\psi}) 
\end{tikzcd}\end{equation}

Let $M$ be the dimension of $\HH_k(Z^\psi)$, and let $\rho_1, \dots, \rho_M$ denote some ordering of the bars of $\bc_k(\intFilt)$ at parameter $N$. Let $\{\vec{e}^{\; N}_1, \dots, \vec{e}^{\; N}_M \}$ denote the corresponding basis vectors of $I^{N}$. Then, the collections $\IntDecDeltaOneLetter = \{ \IntDecDeltaOne^{N}(\vec{e}^{\; N}_1), \dots, \IntDecDeltaOne^{N}(\vec{e}^{\; N}_M) \}$ and $\IntDecDeltaTwoLetter = \{ \IntDecDeltaTwo^{N}(\vec{e}^{\; N}_1), \dots, \IntDecDeltaTwo^{N}(\vec{e}^{\; N}_M) \}$ each form a basis of $\HH_k(\SCOne^{\psi})$. 

Without loss of generality, let $S^{\IntDecDeltaOne, N}_{[\tau]} = \{ \IntDecDeltaOneLower_1 \intXdelta_1, \dots, \IntDecDeltaOneLower_m \intXdelta_m\}$ and $S^{\IntDecDeltaTwo, N}_{[\tau]} = \{ \IntDecDeltaTwoLower_{i_1} \intXdelta_{i_1}, \dots, \IntDecDeltaTwoLower_{i_t} \intXdelta_{i_t}\}$ be the $\IntDecDeltaOne$-bar and $\IntDecDeltaTwo$-bar representations of $[\tau]$ at $N$. That is,
\begin{align*}
[\tau] & = 
\IntDecDeltaOneLower_1[\intXdelta^{\IntDecDeltaOne, N}_1] + \dots + \IntDecDeltaOneLower_m [\intXdelta^{\IntDecDeltaOne, N}_m]  \\
[\tau] &= 
\IntDecDeltaTwoLower_{i_1} [\intXdelta^{\IntDecDeltaTwo, N}_{i_1}] + \dots + \IntDecDeltaTwoLower_{i_t} [\intXdelta^{\IntDecDeltaTwo, N}_{i_t}],
\end{align*}
and all coefficients are nonzero. It follows that 
\[ \Rho^{\IntDecDeltaOne}_{\tau} = \{ \rho_1, \dots, \rho_m \} \text{ and } \Rho^{\IntDecDeltaTwo}_{\tau} = \{\rho_{i_1}, \dots, \rho_{i_t} \}. \]
Let $[\tau]_{\IntDecDeltaOneLetter}$ and $[\tau]_{\IntDecDeltaTwoLetter}$ each denote the coordinate vectors relative to the basis $\IntDecDeltaOneLetter$ and $\IntDecDeltaTwoLetter$. Given a vector $\vec{v}$, we'll use $(\vec{v}\,)_r$ to refer to its $r^{\text{th}}$ coordinate. Then,
$([\tau]_{\IntDecDeltaOneLetter})_r$ is $\IntDecDeltaOneLower_r$ for $1 \leq r \leq m$ and $0$ otherwise. Similarly, $([\tau]_{\IntDecDeltaTwoLetter})_r$ is $\IntDecDeltaTwoLower_r$ for $r \in \{ i_1, \dots, i_t\}$ and 0 otherwise. 

Let $\mathcal{L}= (\IntDecDeltaTwo)^{-1} \circ \IntDecDeltaOne$. Let $F^{N}, G^{N}, L^{N}$ each denote the matrix representation of $\IntDecDeltaOne^{N}, \IntDecDeltaTwo^{N}$, and $\mathcal{L}^{N}$.
Consider the matrix $L^{N} = ({\IntDecDeltaTwoLetter}^{N})^{-1} \circ {\IntDecDeltaOneLetter}^{N}$. Note that $[\tau]_{\IntDecDeltaTwoLetter} = L^{N} [\tau]_{\IntDecDeltaOneLetter}$. Then, $[\tau]_{\IntDecDeltaTwoLetter}$ is a linear combination of the first $m$ column vectors of $L^{N}$. That is, for any $r$, we have
\begin{equation}
\label{coordinates}
([\tau]_{\IntDecDeltaTwoLetter})_r = \IntDecDeltaOneLower_1 (L^{N})_{r, 1} + \dots + \IntDecDeltaOneLower_m (L^{N})_{r,m}.
\end{equation}
We will now show that $\max_{\rho \in \Rho^{\IntDecDeltaOne}_{\tau}} \{\intFiltBirth(\rho)\} = \max_{\rho \in \Rho^{\IntDecDeltaTwo}_{\tau}} \{\intFiltBirth(\rho)\}$. Without loss of generality, assume that $ \intFiltBirth(\intXdelta_{1}) = \max_{\rho \in \Rho^{\IntDecDeltaOne}_{\tau}} \{\intFiltBirth(\rho)\}$. That is,
\begin{equation}
\label{birth_rho1_assumption}
\intFiltBirth(\rho_j) \leq \intFiltBirth(\rho_1) \text{ for all } 1 \leq j \leq m.
\end{equation}

We first show that $\rho_1 \in \Rho^{\IntDecDeltaTwo}_{\tau}$ by showing $([\tau]_{\IntDecDeltaTwoLetter})_{1} \neq 0$. From Equation \ref{coordinates}, we know that 
\[ ([\tau]_{\IntDecDeltaTwoLetter})_{1} = \IntDecDeltaOneLower_1 (L^{N})_{1, 1} + \dots + \IntDecDeltaOneLower_m (L^{N})_{1,m}. \]
By Equation \ref{birth_rho1_assumption}, $\intFiltBirth(\intXdelta_j) < \intFiltBirth(\intXdelta_{1})$ for all $2 \leq j \leq m $. From Lemmas \ref{lemma:basis_change} and \ref{lemma:diagonal}, we know that $(L^{N})_{1,1} \neq 0$  and $(L^{N})_{1,j} = 0$ for all $2 \leq j \leq m$. Thus, $([\tau]_{\IntDecDeltaTwoLetter})_{1} = \IntDecDeltaOneLower_1 (L^{N})_{1, 1} \neq 0 $. In particular, this shows that $\intXdelta_1 \in \Rho^{\IntDecDeltaTwo}_{\tau}$.

We now show that $\intFiltBirth(\intXdelta_i) \leq \intFiltBirth(\intXdelta_{1})$ for every $\intXdelta_i \in \Rho^{\IntDecDeltaTwo}_{\tau}$. Let $\intXdelta_i \in \Rho^{\IntDecDeltaTwo}_{\tau}$. If $\intXdelta_i \in \Rho^{\IntDecDeltaOne}_{\tau}$ as well, then $\intFiltBirth(\intXdelta_{i}) \leq \intFiltBirth(\intXdelta_{1})$ by assumption. If $\rho_i \notin \Rho^{\IntDecDeltaOne}_{\tau}$, recall from Equation \ref{coordinates} that 
		\[ ([\tau]_{\IntDecDeltaTwoLetter})_{i} = \IntDecDeltaOneLower_1 (L^{N})_{i, 1} + \dots + \IntDecDeltaOneLower_m (L^{N})_{i, m} \]
	Since $([\tau]_{\IntDecDeltaTwoLetter})_i \neq 0$, there must be some $ 1 \leq j \leq m$ such that $(L^{N})_{i,j} \neq 0$. By Lemma \ref{lemma:basis_change}, this implies that $\intFiltBirth(\intXdelta_{i}) \leq \intFiltBirth(\intXdelta_j)$, and by  Equation \ref{birth_rho1_assumption}, we have $\intFiltBirth(\intXdelta_{i}) \leq \intFiltBirth(\rho_j) \leq \intFiltBirth(\intXdelta_{1})$. 

So far, we showed that $\intFiltBirth(\intXdelta_i) \leq \intFiltBirth(\intXdelta_{1})$ for every $\intXdelta_i \in \Rho^{\IntDecDeltaTwo}_{\tau}$, and that $\intXdelta_1 \in \Rho^{\IntDecDeltaTwo}_{\tau}$. Then, $\intFiltBirth(\rho_1) =  \max_{\rho \in \Rho^{\IntDecDeltaTwo}_{\tau}} \{\intFiltBirth(\rho)\}$, which coincides with $\max_{\rho \in \Rho^{\IntDecDeltaOne}_{\tau}} \{\intFiltBirth(\rho)\}$. Thus, $\ell_0$ is independent of the choice of the interval decomposition $\IntDecDeltaOne: \mathbb{I}_{\bc_k(\intFilt)} \to P\HH_k(\intFilt)$. 

The rest of the parameters of $p_Y$ are determined from the birth parameter of various bars, which is independent of the interval decomposition. Thus, $p_Y$ is independent of the interval decomposition $\IntDecDeltaOne: \mathbb{I}_{\bc_k(\intFilt)} \to P\HH_k(\intFilt)$. 
\end{proof}

Note that by construction, $\ell_0$ is the smallest parameter at which a restriction of $[\tau]$ can be found.

\begin{restatable}{lemma}{minezero}
\label{min_e0}
Given parameter $\ell$, let $\chi_{\ell}: \HH_k(Z^{\psi} \cap Y^{\ell}) \to \HH_k(Z^\psi)$ be the map induced by inclusion. Given $[\tau] \in \HH_k(\SCOne^{\psi})$, if there exists parameter $\ell$ and $[w] \in \HH_k(\SCOne^{\psi} \cap \SCTwo^{\ell})$ such that $[\tau]= \chi_{\ell}[w]$, then $\ell \geq \ell_0$
\end{restatable}

\begin{proof}
Assume that $S^{\IntDecDeltaOne}_{[\tau]} = \{ \IntDecDeltaOneLower_1 \intXdelta_1, \dots, \IntDecDeltaOneLower_m \intXdelta_m\}$ is the $\IntDecDeltaOne$-bar representation of $[\tau]$ at $N$. That is,
\begin{equation}
\label{Int_x2}
[\tau] = 
\IntDecDeltaOneLower_1[\intXdelta^{\IntDecDeltaOne, N}_1] + \dots + \IntDecDeltaOneLower_m [\intXdelta^{\IntDecDeltaOne, N}_m] = \IntDecDeltaOneLower_1 [\IntDecDeltaOne^N(\vec{e}^{\; N}_{\intXdelta_1})] + \dots + \IntDecDeltaOneLower_m [\IntDecDeltaOne^N(\vec{e}^{\; N}_{\intXdelta_m})]
\end{equation}
in $\HH_k(\SCOne^{\psi})$ for nonzero $f_1, \dots, f_m$.

We prove via contradiction. Assume that there exists some $\ell < \ell_0$ and $[w] \in \HH_k(\SCOne^{\psi} \cap \SCTwo^{\ell})$ such that $[\tau] = \mapx_{\ell}([w])$. Consider the following portion of the diagram for the interval decomposition of $P\HH_k(\intFilt).$

\[
\begin{tikzcd}
I^{\ell} \arrow[r, "\iota"] \arrow[d,"\IntDecDelta^{\ell}"] & I^{N} \arrow[d, "\IntDecDelta^{N}"] \\
\HH_k(\SCOne^{\psi} \cap \SCTwo^{\ell}) \arrow[r, "\chi_{\ell}" ]& \HH_k(\SCOne^{\psi})
\end{tikzcd}
\]
Since $\IntDecDelta^{\ell}$ is an isomorphism, there exists $w' \in I^{\ell}$ such that $[w] = [\IntDecDelta^{\ell}(w')]$. By commutativity, $\IntDecDelta^{N} \circ \iota (w') = [\tau]$. By Equation \ref{Int_x2} and the fact that $\IntDecDelta^{N}$ is an isomorphism, $\iota(w') = \IntDecDeltaOneLower_1 \vec{e}^{\; N}_{\intXdelta_1} + \dots + \IntDecDeltaOneLower_m \vec{e}^{\; N}_{\intXdelta_m}$. In particular, this implies that bars $\intXdelta_1, \dots, \intXdelta_m$ are present at parameter $\ell$. That is, $\intFiltBirth(\intXdelta_1), \dots, \intFiltBirth(\intXdelta_m) \leq \ell$. Recall that $\ell_0 = \max \{ \intFiltBirth(\intXdelta_1), \dots, \intFiltBirth(\intXdelta_m) \} $.
This contradicts $\ell < \ell_0$.
\end{proof}

\subsection{ $\restrictionalgo^{\IntDecDelta}_{\ell}$ is well-defined and is the desired set of restrictions.}

We now show that $\restrictionalgo^{\IntDecDelta}_{\ell}$ from Algorithm \ref{alg:cycle_to_cycle_extension} step (3-b-iii) is independent of the choice of interval decomposition $\IntDecDelta:\mathbb{I}_{\bc_k(\intFilt)} \to P\HH_k(\intFilt)$ and that $\restrictionalgo_{\ell}^{\IntDecDelta}$ coincides with the set of restrictions $\restrictiondef_{\ell}$ defined in Definition \ref{restriction_def}.

We first set the notations. Let $\intFilt$, $\IntDecDelta$, $\Rho^{\IntDecDelta}_{\tau}$, $\Rho_{\text{short}}$, and $p_Y$ be as in Algorithm \ref{alg:cycle_to_cycle_extension}. Let
\[
\Rho^{\IntDecDelta}_{\text{long}} = \{ \rho \in \bc_k(\intFilt) \; | \; \intFiltDeath(\rho) = \infty \text{ and } \rho \notin \Rho^{\IntDecDelta}_{\tau} \}. \]
At any parameter $\ell$, we partition the bars of $\bc_k(\intFilt)$ at parameter $\ell$ into  following three sets
\begin{align}
\label{eq:Rho}
\Rho^{\ell, \IntDecDelta}_{\tau} &= \{ \intXdelta \in \Rho^{\IntDecDelta}_{\tau} \; | \; \intFiltBirth(\intXdelta) \leq \ell  \} \nonumber \\
\Rho^{\ell, \IntDecDelta}_{\text{long}} &= \{ \intXdelta \in \Rho^{\mathcal{F}}_{\text{long}} \; | \; \intFiltBirth(\intXdelta) \leq \ell \}
\\
\Rho^{\ell}_{\text{short}} &= \{ \intXdelta \in \Rho_{\text{short}} \; | \; \intFiltBirth(\intXdelta) \leq \ell < \intFiltDeath(\intXdelta) < \infty \}.  \nonumber
\end{align}

Finally, we state a lemma that is frequently used.

\begin{lemma}
\label{lemma:restriction_to_tau}
Let $\mapx_{\ell}: \HH_k(\SCOne^{\psi} \cap \SCTwo^{\ell}) \to \HH_k(\SCOne^{\psi})$ be the map induced by inclusion. 
Then, 
\begin{align*}
\chi_{\ell} \big( [\intXdelta^{\IntDecDelta, \ell}] \big) &= [\intXdelta^{\IntDecDelta, N}] \text{ for } \intXdelta \in \Rho^{\ell, \IntDecDelta}_{\tau} \cup \Rho^{\ell, \IntDecDelta}_{\emph{long}}, \\
\chi_{\ell} \big([\intXdelta^{\IntDecDelta, \ell}] \big) &= 0 \text{ for } \intXdelta \in \Rho_{\emph{short}}^{\ell}
\end{align*} 
In particular, if $[w] \in \restrictionalgo^{\mathcal{F}}_{\ell}$, then $\chi_{\ell}([w]) = [\tau]$.
\end{lemma}

\begin{proof}
Follows from the commutativity of the following portion of the diagram for the interval decomposition of $P\HH_k(\intFilt).$ 
\[
\begin{tikzcd}
I^{\ell} \arrow[r, "\iota"] \arrow[d,"\IntDecDelta^{\ell}"] & I^{N} \arrow[d, "\IntDecDelta^{N}"] \\
\HH_k(\SCOne^{\psi} \cap \SCTwo^{\ell}) \arrow[r, "\chi_{\ell}" ]& \HH_k(\SCOne^{\psi})
\end{tikzcd}
\]
\end{proof}

\begin{restatable}{theorem}{restrictalgobasisindep}
\label{restrictalgo_basis_independence}
Let $\IntDecDeltaOne, \IntDecDeltaTwo: \mathbb{I}_{\bc_k(\intFilt)} \to P\HH_k(\intFilt)$ be two different interval decompositions. Given $\ell \in p_Y$, let $\restrictionalgo^{\IntDecDeltaOne}_{\ell}$ and $\restrictionalgo^{\IntDecDeltaTwo}_{\ell}$ each denote the set of restrictions in step (3-b-iii) of Algorithm \ref{alg:cycle_to_cycle_extension}. Then, $\restrictionalgo^{\IntDecDeltaOne}_{\ell} = \restrictionalgo^{\IntDecDeltaTwo}_{\ell}$.
\end{restatable}

\begin{proof}

Without loss of generality, let $S^{\IntDecDeltaOne}_{[\tau]} = \{ \IntDecDeltaOneLower^*_{1} \intXdelta_{1}, \dots, \IntDecDeltaOneLower^*_{m} \intXdelta_{m}\}$ be the $\mathcal{F}$-bar representation of $[\tau]$ at $N$. That is,
\begin{equation}
\label{eq:x_C}
[\tau] = \IntDecDeltaOneLower^*_{1} [\intXdelta^{\IntDecDeltaOne, N}_{1}] + \dots + \IntDecDeltaOneLower^*_{m} [\intXdelta^{\IntDecDeltaOne, N}_{m}] \text{ in } \HH_k(\SCOne^{\psi}) \text{ for nonzero }f_1^*, \dots, f_m^*.
\end{equation}

Recall the partition of $\bc_k(Z^\psi \cap Y^\bullet)$ at parameter $\ell$ into sets in Equation \ref{eq:Rho}. The collection $\IntDecDeltaOneLetter = \{ [\intXdelta^{\IntDecDeltaOne, \ell}] \; | \; \rho \in \Rho^{\ell, \IntDecDeltaOne}_{\tau} \cup \Rho^{\ell, \IntDecDeltaOne}_{\text{long}} \cup \Rho^{\ell}_{\text{short}} \}$ form a basis of $\HH_k(\SCOne^{\psi} \cap \SCTwo^{\ell} )$. To prove Theorem \ref{restrictalgo_basis_independence}, we will take $[w] \in \restrictionalgo^{\IntDecDeltaTwo}_{\ell}$, express it using the basis $\IntDecDeltaOneLetter$, and show that $[w] \in \restrictionalgo^{\IntDecDeltaOne}_{\ell}$. 

Let $[w] \in \restrictionalgo^{\IntDecDeltaTwo}_{\ell}$. When we express $[w]$ using basis $F$, it will have the form 
\begin{equation}
\label{w_G}
[w] = \sum_{\rho_j \in \Rho^{\ell, \IntDecDeltaOne}_{\tau}} \IntDecDeltaOneLower^{\tau}_j [\intXdelta_j^{\IntDecDeltaOne, \ell}] + \sum_{\rho_j \in \Rho^{\ell, \IntDecDeltaOne}_{\text{long}}} \IntDecDeltaOneLower^{\text{long}}_j [\intXdelta_j^{\IntDecDeltaOne, \ell}] + \sum_{\rho_j \in \Rho^{\ell}_{\text{short}}} \IntDecDeltaOneLower^{\text{short}}_j [\intXdelta_j^{\IntDecDeltaOne, \ell}] 
\end{equation}
for some coefficients in $\field$. 
From Lemma \ref{lemma:restriction_to_tau}, we know that $\chi_{\ell}([w]) = [\tau]$. From Equation \ref{w_G} and Lemma \ref{lemma:restriction_to_tau},  
	\[ \mapx_{\ell}([w]) = \sum_{\intXdelta_j \in \Rho^{\ell, \IntDecDeltaOne}_{\tau}} \IntDecDeltaOneLower^{\tau}_j [\intXdelta_j^{\IntDecDeltaOne, N}] + \sum_{\intXdelta_j \in \Rho^{\ell, \IntDecDeltaOne}_{\text{long}}} \IntDecDeltaOneLower^{\text{long}}_j [\intXdelta_j^{\IntDecDeltaOne, N}], \]
and this must equal $[\tau]$. From Equation \ref{eq:x_C}, we know that $\IntDecDeltaOneLower^{\tau}_j = \IntDecDeltaOneLower^*_j$ for all $\intXdelta_j \in \Rho^{\ell, \IntDecDeltaOne}_{\tau}$, and $\IntDecDeltaOneLower^{\text{long}}_j= 0$ for all $\rho_j \in \Rho^{\ell, \mathcal{F}}_{\text{long}}$. We can re-write Equation \ref{w_G} as 
$$[w] = \sum_{i=1}^m f_i^* [\rho_i^{\IntDecDeltaOne, \ell}] + \sum_{\rho_j \in \Rho^{\ell}_{\text{short}}} \IntDecDeltaOneLower^{\text{short}}_j [\intXdelta_j^{\IntDecDeltaOne, \ell}], $$
for some coefficients $f^{\text{short}}_j$ in $\field$. Thus,  $[w] \in \restrictionalgo^{\IntDecDeltaOne}_{\ell}$.

Conversely, if $[w] \in \restrictionalgo^{\IntDecDeltaOne}_{\ell}$, the same argument shows that $[w] \in \restrictionalgo^{\IntDecDeltaTwo}_{\ell}$ as well. Thus, $\restrictionalgo^{\IntDecDeltaTwo}_{\ell} = \restrictionalgo^{\IntDecDeltaOne}_{\ell}$.
\end{proof}	

Since the set of restrictions $\restrictionalgo^{\IntDecDeltaOne}_\ell$ is independent of the interval decomposition $\IntDecDeltaOne$, we can omit the $\IntDecDeltaOne$. 
\begin{restatable}{theorem}{restrictionequality}
\label{thm:restriction_equality} 
Given parameter $\ell \in p_Y$, the set $\restrictionalgo_{\ell}$ found in Algorithm \ref{alg:cycle_to_cycle_extension} coincides with $\restrictiondef_{\ell}$ from Definition \ref{restriction_def}.
\end{restatable}

\begin{proof}
If $[w] \in \restrictionalgo_{\ell}$, then by Lemma \ref{lemma:restriction_to_tau}, we know that $\mapx_{\ell}([w]) = [\tau]$. So $[w] \in R_{\ell}$.

If $[w] \in R_{\ell}$, then $[w] \in \HH_k(\SCOne^{\psi} \cap Y^{\ell})$. Recall the partition of bars in Equation \ref{eq:Rho}. We can express $[w]$ using the basis $\{ [\intXdelta^{\IntDecDeltaOne, \ell}] \; | \; \intXdelta \in \Rho^{\ell, \IntDecDelta}_{\tau} \cup \Rho^{\ell, \IntDecDelta}_{\text{long}} \cup \Rho^{\ell}_{\text{short}} \}$ of $\HH_k(Z^\psi \cap Y^\ell)$ as 
\[[w] = \sum_{\rho_j \in \Rho^{\ell, \IntDecDelta}_{\tau}} \IntDecDeltaLower^{\tau}_j [\intXdelta^{\IntDecDeltaOne, \ell}_j] + \sum_{\rho_j \in \Rho^{\ell, \IntDecDelta}_{\text{long}}} \IntDecDeltaLower^{\text{long}}_j [\intXdelta^{\IntDecDeltaOne, \ell}_j] + \sum_{\rho_j \in \Rho^{\ell}_{\text{short}}} \IntDecDeltaLower^{\text{short}}_j [\intXdelta^{\IntDecDeltaOne, \ell}_j] \]
for some coefficients in $\field$.

Recall $S^{\IntDecDelta}_{[\tau]} =    \{\IntDecDeltaLower^*_1\intXdelta_1, \dots, \IntDecDeltaLower^*_m\intXdelta_m \}$, the $\IntDecDelta$-bar representation of $[\tau]$ at $N$ from step (2-a) of Algorithm \ref{alg:cycle_to_cycle_extension}. Since $[w] \in R_\ell$, we know $\mapx_{\ell}([w]) = [\tau]$. By Lemma \ref{lemma:restriction_to_tau}, the coefficients $\IntDecDeltaLower^{\tau}_j$ must agree with $\IntDecDeltaLower^*_j$ for all $\rho_j \in \Rho^{\ell, \mathcal{F}}_{\tau}$, and $\IntDecDeltaLower^{\text{long}}_j = 0$ for all $\rho_j \in \Rho^{\ell, \mathcal{F}}_{\text{long}}$. Thus, $[w] \in \restrictionalgo^{\ell}$.
\end{proof}

\subsection{The output of Algorithm \ref{alg:cycle_to_cycle_extension} suffices to find all cycle extensions of $[\tau]$. }

We now prove the main theorem.
\mainthm*

\begin{proof}
We first show that if $[y] \in \cealgo^*$, then $[y] \in E$. Assume $[y] \in \cealgo^*$. So $[y] \in \HH_k(\SCTwo^{\ell'})$ for some parameter $\ell'$, and there exists some $\ell \in p_Y$ and $[y_{\ell}] \in \cealgo_{\ell}$ such that $[y] = \lambda^{\ell}_{\ell'}([y_{\ell}])$. That is, $[y_{\ell}]$ is a cycle extension at $\ell$ via some $[w] \in \restrictionalgo_{\ell}=R_{\ell}$. By Lemma \ref{extension_later}, we know that $[y] = \lambda^{\ell}_{\ell'}([y_{\ell}])$ is a cycle extension via $\eta^{\ell}_{\ell'}([w])$ at $\ell'$, where $\eta^{\ell}_{\ell'}: \HH_k(\SCOne^{\psi} \cap Y^\ell) \to \HH_k(\SCOne^{\psi} \cap Y^{\ell'})$ is the map induced by inclusion. Note that $\eta^{\ell}_{\ell'}([w]) \in R_{\ell'}$. Thus, $[y] \in E_{\ell}$, and hence, $[y] \in E$.

Now, assume that $[y] \in E$. There exists some $\ell$ and $[w] \in R_{\ell}$ such that $[y] = \mapy_{\ell}([w])$, where $\mapy_{\ell}: \HH_k(\SCOne^{\psi} \cap \SCTwo^{\ell}) \to \HH_k(\SCTwo^{\ell})$ is the map induced by inclusion. Assume that the birth parameter of $[w]$ in $H_k(\SCOne^{\psi} \cap \SCTwo^\bullet)$ is $\bar{\ell}$. Let $\eta^{\bar{\ell}}_{ \ell}: \HH_k(\SCOne^{\psi} \cap \SCTwo^{\bar{\ell}}) \to \HH_k(\SCOne^{\psi} \cap \SCTwo^{\ell})$ be the map induced by inclusion. Then, $[w] = \eta^{\bar{\ell}}_{ \ell}([\bar{w}])$ for some $[\bar{w}] \in \HH_k(\SCOne^{\psi} \cap \SCTwo^{\bar{\ell}})$. By Lemma \ref{extension_earlier}, there exists a cycle extension $[\bar{y}] \in \HH_k(\SCTwo^{\bar{\ell}})$ at $\bar{\ell}$ via $[\bar{w}]$, and $\lambda^{\bar{\ell}}_{ \ell}([\bar{y}]) = [y]$, where $\lambda^{\bar{\ell}}_{\ell}: \HH_k(\SCTwo^{\bar{\ell}}) \to \HH_k(\SCTwo^{\ell})$ is the map induced by inclusion. One can check that $[\bar{w}]$ is a restriction of $[\tau]$. We thus know from Lemma \ref{min_e0} that $\ell_0 \leq \bar{\ell}$. Furthermore, since a new cycle is born at $\bar{\ell}$, the parameter should coincide with the birth of a new bar of $\bc_k(\intFilt)$. Thus, $\bar{\ell} \in p_Y$, $[\bar{w}] \in \restrictionalgo_{\bar{\ell}}$, and $[\bar{y}] \in \cealgo_{\bar{\ell}}$. That is, we must have found $[\bar{y}]$ as a cycle extension in Algorithm \ref{alg:cycle_to_cycle_extension}. Thus, $[y] = \lambda^{\bar{\ell}}_{\ell}([\bar{y}]) \in \cealgo^*$.
\end{proof}

\section{Proof of Lemma \ref{lemma:LZ}}
\label{Appendix:pf_LZ}
\LZ*
\begin{proof}
Let $M_Z$ be the collection of matrix representations of $\mathcal{L}^{\psi}$ for every possible persistence module automorphism $\mathcal{L}$. 

We first show that $L_Z \subseteq M_Z$. Let $L \in L_Z$. From Lemma \ref{lemma:linear_independence}, we know that $L$ is invertible. Further, any such isomorphism $L$ immediately extends to an automorphism of persistence modules $\mathcal{L}: \mathbb{I}_{\bc_k(\filtOne)} \to \mathbb{I}_{\bc_k(\filtOne)}$ because $\mathbb{I}_{\bc_k(\filtOne)}$ is a direct sum of interval modules. Thus, $L_Z \subseteq M_Z$.

We now show that $M_Z \subseteq L_Z$. Let $M$ be a matrix in $M_Z$. That is, $M$ is the matrix representation of the linear isomorphism $\mathcal{L}^{\psi}$ for some automorphism of persistence modules $\mathcal{L}: \mathbb{I}_{\bc_k(\filtOne)} \to \mathbb{I}_{\bc_k(\filtOne)}$. It follows from Lemmas \ref{lemma:basis_change} and \ref{lemma:diagonal} that $M \in L_Z$. Thus, $M_Z \subseteq L_Z$.

\end{proof}

\section{Proof of Lemma \ref{lemma:well_defined_cycles}}
\label{Pf_well_defined_cycles}
\welldefinedcycles*
\begin{proof}
For notational convenience, let $\psi = \delta(\tau)-1$. Let $\mathcal{L}: \mathbb{I}_{\bc_k(\filtOne)} \to \mathbb{I}_{\bc_k(\filtOne)}$ be the isomorphism $\mathcal{L} = \mathcal{C}^{-1} \circ \mathcal{B}$. Let $I^{\psi} = (\mathbb{I}_{\bc_k(\filtOne)})^{\psi}$. Let $B^{\psi}, C^{\psi}, L^{\psi}$ each be the matrix representation of linear isomorphisms $\mathcal{B}^{\psi}, \mathcal{C}^{\psi}: I^{\psi} \to \HH_k(\SCOne^{\psi})$ and $\mathcal{L}^{\psi}: I^{\psi} \to I^{\psi}$.

Let $\tau_1, \dots, \tau_m$ be some ordering of the bars of $\bc_k(\filtOne)$ alive at parameter $\psi$. Let $B$ and $C$ each denote the collection of homology class that $\mathcal{B}$ (or $\mathcal{C}$)-correspond to the bars at parameter $\psi$:
\[ B = \{ [\tau_1^{\mathcal{B}, \psi}], \dots, [\tau_m^{\mathcal{B}, \psi}] \}, \text{ and } C = \{ [\tau_1^{\mathcal{C}, \psi}], \dots, [\tau_m^{\mathcal{C}, \psi}] \}. \]
Note that both $B$ and $C$ are valid choices of basis for $\HH_k(\SCOne^\psi).$

Without loss of generality, assume that $\tau_1 = \tau$. Then, 
\begin{equation}
\label{tau_BC}
[\tau^{\mathcal{B}}_*] = [\tau_1^{\mathcal{B}, \psi}], \quad [\tau^{\mathcal{C}}_*] = [\tau_1^{\mathcal{C}, \psi}].
\end{equation}
Now consider the change of basis matrix $L^{\psi} = (C^{\psi})^{-1} \circ B^{\psi}$. Using the coordinates vectors with respect to the basis $B$ and $C$, we know that 
\[ [\tau^{\mathcal{B}}_*]_{C} = L^{\psi} [\tau^{\mathcal{B}}_*]_{B}. \]
We will show that $[\tau_*^{\mathcal{B}}]_{C} = [\tau_*^{\mathcal{C}}]_C$. From Equation \ref{tau_BC}, we know that $[\tau^{\mathcal{B}}_*]_{B} = [1, 0, \dots, 0]^T $, so $[\tau^{\mathcal{B}}_*]_{C}$ coincides with the $1^{\text{st}}$ column vector of $L^{\psi}$. By Lemma \ref{lemma:diagonal} and the fact that $\field = \field_2$, we know that $(L^{\psi})_{(1,1)} = 1$. Furthermore, given any $j \neq 1$, the corresponding interval $\tau_j$ must satisfy $\delta(\tau_j) > \delta(\tau)$, since the death parameters are unique. By Lemma \ref{lemma:basis_change}, we know that $(L^{\psi})_{(j, 1)} = 0$ for all $j \neq 1$. Thus, $[\tau^{\mathcal{B}}_*]_{C} = [1, 0, \dots, 0]^T$, which coincides with $[\tau^{\mathcal{C}}_*]_C$ by Equation \ref{tau_BC}. Thus, $[\tau^{\mathcal{B}}_*]_C = [\tau^{\mathcal{C}}_*]_C$, and 
$[\tau^{\mathcal{B}}_*]  = [\tau^{\mathcal{C}}_*]$ in $\HH_k(\SCOne^{\psi})$.
\end{proof}

\bibliographystyle{amsplain}
\bibliography{paper.bib}

\end{document}